\documentclass[11pt,reqno]{amsart}
\usepackage{multirow} 

\DeclareMathAlphabet{\mathcalold}{OMS}{cmsy}{m}{n}
\DeclareMathAlphabet{\bmathcalold}{OMS}{cmsy}{b}{n}
\DeclareMathAlphabet{\mathbcal}{OMS}{cmsy}{b}{n}

\usepackage{graphicx}
\usepackage{subfigure} % subfiguras
\graphicspath{{Graphics/}}
\usepackage{amsmath,amssymb,amsopn,amsfonts}
\usepackage{mathtools}
\usepackage{hyperref}
\usepackage{color}
\usepackage{tcolorbox}
\tcbuselibrary{xparse}
\usepackage{relsize}
\usepackage{latexsym}
\usepackage{cancel}
\usepackage{array}
\usepackage{cite} 
 \usepackage{booktabs} %table
\numberwithin{equation}{section}

\hypersetup{
    colorlinks = true,
    linkcolor  = blue,
    filecolor  = magenta,      
    urlcolor   = cyan,
}
\allowdisplaybreaks

\newtheorem{lemma}{Lemma}[section]

\newcommand{\rmd}{\mathrm{d}}
\newcommand{\rmT}{\mathrm{T}}

\newcommand{\Qe}{Q_{\mathrm{e}}}
\newcommand{\Qf}{Q_{\mathrm{f}}}
\newcommand{\Qu}{Q_{\mathrm{u}}}
\newcommand{\qe}{q_{\mathrm{e}}}
\newcommand{\qu}{q_{\mathrm{u}}}
\newcommand{\qf}{q_{\mathrm{f}}}

\newcommand{\Xf}{X_{\mathrm{f}}}

\newcommand{\bC}{\boldsymbol{C}}

\newcommand{\bPhi}{\boldsymbol{\Phi}}

\newcommand{\bR}{\boldsymbol{R}}

\newcommand{\bSf}{\boldsymbol{S}_{\mathrm{f}}}

\newcommand{\bS}{\boldsymbol{S}}
\newcommand{\bp}{\boldsymbol{p}}

\newcommand{\bzero}{\boldsymbol{0}}

\newcommand{\sme}{\sigma_{\mathrm{e}}}
\newcommand{\vhs}{v_{\mathrm{hs}}}
\newcommand{\sgn}{\operatorname*{sgn}} 
\newcommand{\diag}{\operatorname*{diag}} 
\newcommand{\Upw}{\operatorname*{Upw}}

\definecolor{ros}{RGB}{148,35,9}   
\usepackage{upgreek}
\newcommand{\bCf}{\bC_{\rm f}}

\newcommand{\br}{\boldsymbol{r}}
\newcommand{\bLambda}{\boldsymbol{\Lambda}}
\newcommand{\bsigmaC}{\boldsymbol{\sigma}_\mathrm{\!\bC}}
\newcommand{\bsigmaS}{\boldsymbol{\sigma}_\mathrm{\!\bS}}

\title[Numerical schemes for 
 sequencing batch reactors]{Numerical schemes for a moving-boundary
convection-diffusion-reaction~model of sequencing~batch~reactors}
\author[B\"urger]{Raimund B\"urger$^{\mathrm{A}}$}
\author[Careaga]{Julio Careaga$^{\mathrm{B}}$}
\author[Diehl]{Stefan Diehl$^{\mathrm{C}}$}
\author[Pineda]{Romel Pineda$^{\mathrm{A},*}$}

\DeclareMathAlphabet\mathbfcal{OMS}{cmsy}{b}{n}

\date{\today} 
\thanks{$^{*}$Corresponding author.} 
\thanks{$^{\mathrm{A}}$CI$^{\mathrm{2}}$MA and Departamento de Ingenier\'{\i}a Matem\'{a}tica,
Facultad de Ciencias F\'{\i}sicas y Matem\'{a}ticas, Universidad de
Concepci\'{o}n, Casilla 160-C, Concepci\'{o}n, Chile.
  E-Mail: 
{\tt rburger@ing-mat.udec.cl}, {\tt rpineda@ing-mat.udec.cl}}  

\thanks{$^{\mathrm{B}}$Department of Mathematics, Radboud University, Heyendaalseweg 135, 6525 AJ Nijmegen, The Netherlands. E-Mail:  {\tt jcareaga@science.ru.nl}}

\thanks{$^{\mathrm{C}}$Centre for Mathematical Sciences, Lund University, P.O. Box 118, S-221 00 Lund, Sweden. 
 E-Mail:  {\tt stefan.diehl@math.lth.se}}

\keywords{Sequencing batch reactor, convection-diffusion-reaction model, moving boundary, monotone scheme, semi-implicit method, invariant region principle} 

\subjclass{65M06, 35K57,  35Q35} 
 
\begin{document}

%\linenumbers

\begin{abstract}
 Sequencing batch reactors (SBRs) are devices widely used in wastewater treatment, chemical engineering, and other areas. 
  They allow for the sedimentation and compression of solid particles of biomass simultaneously with biochemical reactions 
   with nutrients dissolved in the liquid.  The kinetics of these reactions may be given by one of the established 
    activated sludge models (ASMx). An SBR is operated in various stages and is equipped with a movable extraction and fill 
     device and a discharge opening. A one-dimensional model of this unit can be formulated as a moving-boundary 
      problem for a degenerating system of convection-diffusion reaction equations whose unknowns are the concentrations of the components forming the 
       solid and liquid phases, respectively. This model is transformed to a fixed computational domain and is discretized by 
        an explicit monotone scheme along with an alternative semi-implicit variant. The semi-implicit variant is based 
         on solving, during each time step, a system of nonlinear equations for the total solids concentration followed 
          by the solution of linear systems for the solid component  percentages and liquid component concentrations. 
          It is demonstrated that the semi-implicit scheme is well posed and that 
          both variants 
           produce approximations  that satisfy an invariant region principle:  
          solids concentrations are nonnegative and less or equal to a set maximal one, percentages are nonnegative and 
           sum up to one, and substrate concentrations are nonnegative. These properties are achieved under a  Courant-Friedrichs-Lewy (CFL) 
                    condition   that is less restrictive for the semi-implicit than the explicit variant. 
          Numerical examples with realistic parameters illustrate that as a consequence, the semi-implicit 
           variant is more efficient than the explicit one. 
\end{abstract}
\maketitle

\section{Introduction}

\subsection{Scope} 

Sequencing batch reactors (SBRs) are widely used 
  in the pharmaceutical~\cite{Wang2011, Popple2016}, petrochemical~\cite{Caluw2016}, and  chemical~\cite{Hu2005}  industries 
   as well as in  wastewater treatment~\cite{Amin2014,Song2021}. An SBR is a tank designed for 
the sedimentation of a suspension composed of solid particles of biomass that react with substrates (nutrients) dissolved in a liquid. 
 This process is a fundamental stage    of   purification   in wastewater treatment plants.
Due to the living biomass (activated sludge; bacteria), biochemical reactions always occur. %  whether they are  advantageous or not.
 In particular, these reactions are the basis of the well-known activated sludge process in wastewater treatment. 
The simultaneous process of sedimentation and reactions, called reactive settling, occurs both in plants with continuously operated secondary settling tank (SST) and in SBRs when such are used.
We study here a general model of an SBR, which is operated in a number of sequential stages \cite{metcalf,droste,chen2020book}.
Depending on the stage, the suspension is continuously fed or extracted from the top of its surface by a floating device; see Figure~\ref{fig:zbar}, and concentrated mixture can be extracted from the bottom of the vessel.
The cycle of the SBR takes into account five principal stages: fill, react, settle, draw and idle. The bulk flows at the different stages lead to a moving boundary; see Figure~\ref{fig:SBRfig}.

A nonlinear, strongly degenerating  convection-diffusion-reaction model of an SBR was introduced in~\cite{SDAMM_SBR1} and a numerical scheme defined on a fixed spatial grid in~\cite{SDAMM_SBR2}.
That scheme will here be referred to as SBR2.
Within that approach one needs to carefully track the upper boundary as it moves through the computational grid for the numerical solution.
It is the purpose of the present work to introduce alternative explicit scheme that after a suitable transformation of the time-dependent spatial domain is defined on a fixed computational grid.
Moreover, a semi-implicit variant is introduced that allows for more efficient simulations due to a more favourable Courant-Friedrichs-Lewy (CFL) condition. 

\begin{figure}[t]
 \centering
 \includegraphics[scale=0.6]{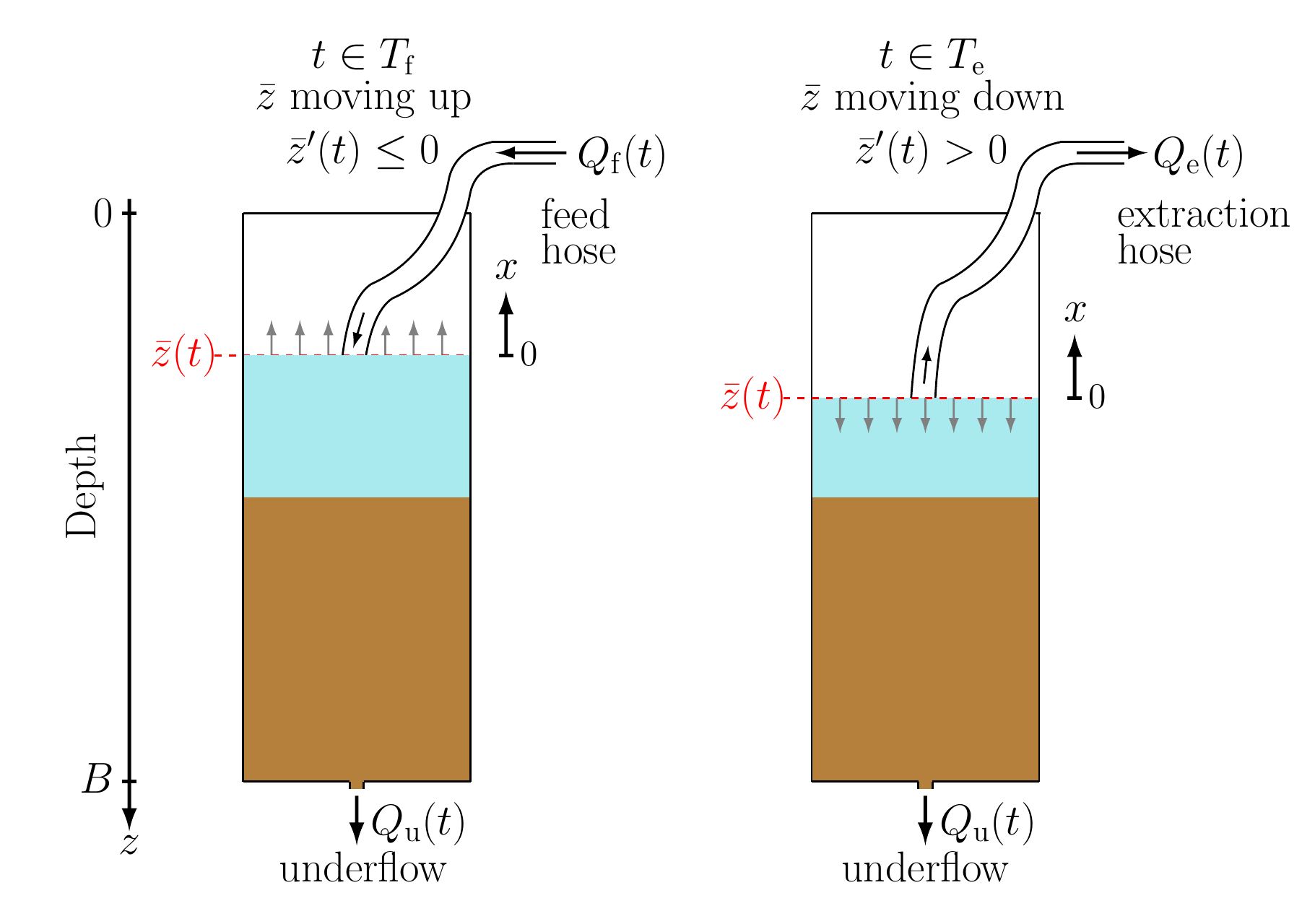}
 \caption{Illustration of an SBR in the two stages that involve flow through the hose:  (left) feed and (right) extraction. 
  The boundary $\bar{z}$  moves upward at time~$t$ when $\bar{z}'(t)<0$ and downward when $\bar{z}'(t)>0$, and $T_{\rm f}$ and 
    $T_{\rm e}$ are the feed and extraction time intervals, respectively.} \label{fig:zbar}
\end{figure}%

The moving boundary, denoted by $\bar{z}\coloneqq\bar{z}(t)$, where $t\geq 0 $ is
  time, is a known function given by the bulk flows acting on the SBR. 
The curve $t \mapsto \bar{z}(t)$ defines a spatial change of variables, where in the new variable $\xi$, the top and bottom boundaries are 
      located at~$\xi=0$ and $\xi=1$, respectively.  
The sought vector unknowns, in the variable $\xi$, are the solid phase and liquid substrate components
\begin{align*}
\bC &\coloneqq \bC(\xi,t) \coloneqq (C^{(1)}(\xi,t),\dots,C^{(k_{\bC})}(\xi,t))^\rmT,\quad 
\bS \coloneqq \bS(\xi,t) \coloneqq (S^{(1)}(\xi,t),\dots,S^{(k_{\bS})}(\xi,t))^\rmT.
\end{align*}
The model equations for $0<\xi<1$ are given by
  \begin{align} \label{syst:intro} \begin{split} 
\partial_t  \boldsymbol{C} + \partial_{\xi} \bigl(  U_{\boldsymbol{C}} 
 (X, \xi, t) \boldsymbol{C} \bigr) 
 & = \partial_{\xi} \bigl( \gamma(\xi) \beta(t)^2 \partial_{\xi} \mathcal{D} (X) 
  \bigr) \\ & \quad +  \beta(t)\bar{z}'(t)\bC + \delta (\xi)\beta(t)\qf(t)\bCf(t) + \gamma(\xi)\bR_{\bC}(\bC,\bS),   \\ 
 \partial_t{\bS}  + \partial_\xi \bigl(U_{\boldsymbol{S}} (X, \xi,t)\boldsymbol{S}\big)  & = 
 \partial_{\xi} \biggl( \frac{ \gamma(\xi) \beta(t)^2}{\rho_X - X}  \partial_{\xi} \mathcal{D} (X) 
  \biggr) \\ & \quad + 
 \beta(t)\bar{z}'(t)\bS + \delta (\xi)\beta(t)\qf(t)\bSf(t)  + \gamma(\xi)\bR_{\bS}(\bC,\bS),  \end{split} 
\end{align} 
where $U_{\boldsymbol{C}}$ and $U_{\boldsymbol{S}}$ are velocity functions related to the velocities of the solid and liquid phases, respectively, $X \coloneqq (C^{(1)} + \cdots + C^{(k_{\bC})})/c$ is the total solids concentration, $c$ is a conversion factor, $\beta$ is a given time dependent function, $\rho_X$~is the (constant) mass density of solids, and $\mathcal{D}(X)$ 
  is an integrated diffusion coefficient  describing compressibility of the sediment. It is assumed 
   that $\mathcal{D} (X) =0$ for $X \leq X_{\mathrm{c}}$, where $X_{\mathrm{c}} >0$ is a critical concentration 
    beyond which the solid particles are assumed to touch each other, so the system \eqref{syst:intro} 
     is of first order wherever $X \leq X_{\mathrm{c}}$, and therefore is strongly degenerate. 
  The 
   second term in the right-hand side of both equations is due to the change of variables in the spatial partial derivative, while the 
    third term contains the feed concentrations~$\bC_{\rm f}$ and~$\bS_{\rm f}$ for the solid and liquid phases, respectively, the feed velocity~$q_{\rm f}$, and the Dirac symbol~$\delta$. The last terms  model the chemical reactions between the components of the solid and liquid phases, 
     where $\bR_{\bC}$ and $\bR_{\bS}$ are vectors of reaction terms, and $\gamma$ is the characteristic function that equals one inside the tank 
       and zero outside. 
The system~\eqref{syst:intro} is coupled to transport equations modelling the concentrations in the extraction pipe connected at $\xi=0$. 
 (The  variables and functions arising in \eqref{syst:intro} are precisely defined in later parts of the paper.)

\begin{figure}[t]
 \centering
 \includegraphics[scale=0.6]{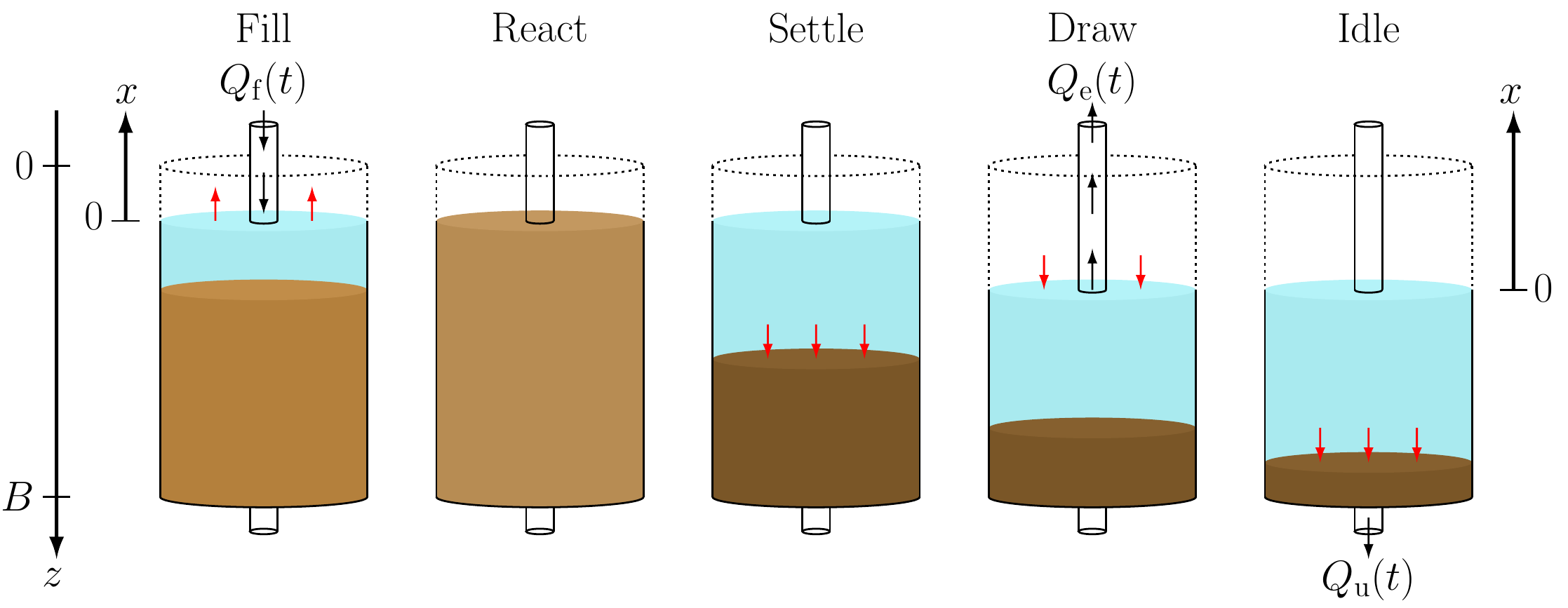}
 \caption{Schematic of the five stages of an SBR. The red arrows show the movement of the boundary or the sediment level during each stage. The black arrows show the direction of the bulk flows $Q_{\rm f}$ (feed), $Q_{\rm e}$ (extraction) and $Q_{\rm u}$ (underflow). 
 } \label{fig:SBRfig}
\end{figure}%

\subsection{Related work}
Several works posed for SBRs have mostly considered the fill and react stages, where a continuously mixed tank is assumed and the biochemical reactions are modelled by ODEs from some of the established activated sludge models (ASM1, ASM2, ASM2d, or ASM3; in short, AMSx)~\cite{Henze2000ASMbook,makinia20}.
 We  refer to handbooks (e.g., \cite{chen2020book,droste,metcalf} for
  a broader introduction to the background of  wastewater treatment.
PDE models for the reactive-settling process have recently been developed~\cite{SDIMA_MOL,SDAMM_SBR1,SDAMM_SBR2,SDm2an_reactive}. One of the first PDE-based models combining the sedimentation-consolidation process for the batch denitrification in wastewater treatment was presented in \cite{SDcace_reactive}. That model consists of five unknowns, two solid components and three substrates, and includes nonlinear terms for the reactions between the solid and liquid components. 
An extension to the case of continuous sedimentation was introduced in \cite{SDm2an_reactive}, where the model was based on the percentages of the concentrations of the solid and liquid phases.
A reliable numerical scheme and simulations of the denitrification process were also included.
The present work also utilizes the concept of percentages for sedimentation models, which was introduced in~\cite{SDmathmeth1}.
In~\cite{SDIMA_MOL}, the authors presented an alternative reactive settling model which also included non-cylindrical vessels. A numerical scheme based on a method-of-lines formulation was developed and the numerical results are compared with the ones of \cite{SDm2an_reactive}. 
Recently, a slightly modified version of that model including extra diffusion terms modelling mixing near the feed inlet has been calibrated to experiments from a pilot plant in \cite{Burger2023}. 

The present SBR model is the one in~\cite{SDAMM_SBR1}, which is a development of the one in \cite{SDIMA_MOL} to take into account the moving boundary due to bulk flows of the SBR stages. 
The settling of particles is allowed in the model in all five stages with the exception of the react one.
A conservative and positivity-preserving numerical scheme developed on a fixed mesh grid was introduced in~\cite{SDAMM_SBR2}.
The scheme of~\cite{SDAMM_SBR2} employs careful mass distributions in cells near the moving boundary to ensure the conservation of mass.
Simulations of the activated sludge model no.~1 (ASM1)~\cite{Henze2000ASMbook} were included in~\cite{SDAMM_SBR2} for the case of a cylindrical vessel.

The trajectory of~$\bar{z}(t)$ is given, and not part of the solution of the problem, so the governing PDE model 
 \eqref{syst:intro} is a moving boundary problem, but not a free  boundary problem. Nevertheless, 
  extensions of the SBR model to a free  boundary problem are conceivable. For one scalar PDE related 
   models of filtration have been studied in~\cite{bfk02,bck01}.  Finally, we mention that the formulation and 
    partial analysis of the  semi-implicit scheme presented herein 
     is based on anlyses of schemes of that type 
      for degenerate parabolic PDE in~\cite{Burger&C&S2006,ekhyp2000}.

\subsection{Outline of the paper} The remainder of this work is organized as follows. 
 In Section~\ref{sec:SBRmodel}, we review the SBR model, starting 
  with assumptions on the tank and the one-dimensional model, and explaining the relation 
   between the moving boundary and bulk flows~(Section~\ref{subsec:onedass}).  
   The one-dimensional model of the reactive-sedimentation process in the tank is obtained from the balance laws of mass of all the particulate and soluble components. The details of the formulation of the model are outlined in 
    Section~\ref{subsec:model}, starting from the solid and liquid components and reactions  within the biokinetic reaction model  
    (as an example, we here employ the model ASM1) followed by a summary of the 
     sedimentation-compression model that has also been employed in previous works. The 
      balance laws for the SBR are coupled with transport equations for the pipe. During the react stage with full mixing, the PDEs are replaced 
       by a system of ordinary differential equations (ODEs). 
         While these model ingredients are reviewed from our previous treatment \cite{SDAMM_SBR1,SDAMM_SBR2}, 
          one of the three decisive novelties of the present approach, namely the  
           transformation of the moving boundary model to a fixed 
            computational domain, is formulated in Section~\ref{model:fixeddomain}. 
             The second novelty is the reformulation of the governing model in terms of the vector of percentages~$\boldsymbol{p}$ 
               (of the solid components, with respect to~$X$) 
               done in Section~\ref{model:percentages}. 

               An explicit numerical  scheme for the solution of the transformed governing model is 
                introduced in Section~\ref{sec:numscheme}.
 This scheme is based on a standard discretization in space and time of the computational domain (Section~\ref{subsec:discr}). 
   The explicit numerical scheme, presented in Section~\ref{subsec:explicheme},   
    is based on the percentage  and fixed-domain formulations of Sections~\ref{model:fixeddomain} and~\ref{model:percentages} 
     and combines upwind discretizations for transport terms, the Engquist-Osher numerical flux \cite{eopaper} for nonlinear flux terms, and 
      a central difference formula for nonlinear degenerating diffusion terms arising in the governing models, 
      combined with  appropriate discretizations of the reaction terms. The scheme 
       is complemented by  a suitable discretization of the ODEs describing the mixing stage (Section~\ref{subsec:num-mix}). 
       It is assumed that the spatial  mesh width~$\Delta \xi$ and the time step~$\tau$ are related by a CFL condition 
        outlined in~Section~\ref{subsec:CFL}, and we prove in Section~\ref{subsec:mon} that this scheme 
         is montone, with the consequence that an  invariant-region principle holds, i.e., 
          solids concentrations are nonnegative and less or equal a maximal one, percentages are nonnegative and 
           sum up to one, and substrate concentrations are nonnegative. The CFL condition for the explicit scheme essentially 
            bounds~$\tau/ \Delta \xi^2$. 

A more favorable CFL condition that only bounds~$\tau/ \Delta \xi$ is associated 
             with a semi-implicit scheme for the governing PDE model, which is the third novelty introduced in Section~\ref{scheme:si}. 
At each time step, it
              consists of a nonlinear semi-implicit scheme for the PDE for~$X$, which is described 
              in Section~\ref{subsec:siX}, where we also demonstrate that the nonlinear equations 
               are well posed, i.e., possess a unique solution that depends continuously on data. 
                The solution of the nonlinear equations is achieved by
               Newton-Raphson method (Section~\ref{subsec:nr}). Once the scalar function $X$~has been updated, one proceeds to update the vectors~$\boldsymbol{p}$ and~$\boldsymbol{S}$. 
                  This is done by an implicit version of the corresponding $\boldsymbol{p}$- and $\boldsymbol{S}$-schemes of 
                   Section~\ref{sec:numscheme}. This  version is, however, linearly implicit 
                    and only requires the solution of one linear system per time step, see Section~\ref{subsec:updperc}. 
                     In Section~\ref{subsec:simon}, it is proved that the 
                      semi-implicit scheme has the same monotonicity and invariant-region properties 
                        as its explicit counterpart (cf.\ Section~\ref{subsec:mon}) but does so under 
                         the more favorable CFL condition. 
                         Numerical examples  that illustrate the performance of the numerical schemes  
                          of Sections~\ref{sec:numscheme} and~\ref{scheme:si} are presented in 
                        Section~\ref{sec:numer}. In particular, it is demonstrated that the semi-implicit scheme indeed 
                         leads to the expected gain in efficiency. Finally, conclusions are collected in Section~\ref{sec:conc}.

\section{A model of a sequencing batch reactor (SBR)}\label{sec:SBRmodel}

We here review the assumptions of the general SBR model in~\cite{SDAMM_SBR1} and refer to~\cite{SDAMM_SBR2} for a full description of the numerical scheme SBR2.
That model will then be slightly reformulated, which allows for a semi-implicit numerical scheme to be developed.
The sedimentation and compaction properties of the flocculated sludge are namely assumed to depend on the total concentration~$X$ and not on the concentrations of the individual components since these are flocculated to larger particles.
Each flocculated particle consists of several individual components, which can be described as percentages of the total concentration~$X$.
For simplicity of writing, we confine here to a constant cross-sectional area~$A$, which is the most common case in the applications.

\subsection{Assumptions on the tank and the 1D model} \label{subsec:onedass} 

The reactor tank is assumed to be a cylindrical vessel with constant horizontal cross-sectional area~$A$; see Figure~\ref{fig:SBRfig}.
We place a  fixed $z$-axis indicating the depth from the top ($z=0$) to the bottom at~$z=B$.
At the surface of the mixture, located at~$z=\bar{z}(t)$,  a floating device connected to a pipe allows one to feed the tank at given volumetric feed flow~$\Qf(t)$ [m$^3$/s] and   feed concentrations~$\bCf(t)$ and $\bSf(t)$.
The floating device can alternatively extract mixture at a given volume rate~$\Qe(t)$ during the draw stage.
One cannot fill and draw simultaneously.
The extraction pipe is modelled by a half-line and the flow through it by a linear advection PDE.
The  concentrations in the pipe are denoted by~$\bC_\mathrm{e}(t)$ and~$\bS_\mathrm{e}(t)$.
At the bottom, $z=B$, one can withdraw mixture at a given volume rate $\Qu(t)\geq 0$, and the corresponding output concentrations there are denoted by~$\bC_\mathrm{u}(t)$ and $\bS_\mathrm{u}(t)$.
We define the bulk velocities
\begin{equation*}
\qu(t)\coloneqq  \Qu(t) / A,\qquad \qe\coloneqq \Qe(t) / A,\qquad \qf\coloneqq \Qf(t)/A.
\end{equation*}

If~$[0,T]$ denotes the total time interval of modelling (and simulation in Section~\ref{sec:numscheme}), we assume that $T\coloneqq {T}_\mathrm{e}\cup{T}_\mathrm{f}$, 
where \begin{align*}
{T}_\mathrm{e}&\coloneqq  \bigl\{t\in\mathbb{R}_+:\Qe(t)>0,\Qf(t)=0 \bigr\},
\qquad 
{T}_\mathrm{f}\coloneqq  \bigl\{t\in\mathbb{R}_+:\Qe(t)=0,\Qf(t)\geq0 \bigr\}
\end{align*}
(such that ${T}_\mathrm{e}\cap{T}_\mathrm{f} = \varnothing$).  
The volume of the mixture at time~$t$ is
\begin{equation}\label{eq:V}
\bar{V}(t)\coloneqq  A\big(B-\bar{z}(t)\big).
\end{equation}
The surface location~$\bar{z}(t)$ is determined by the given volumetric flows, since by differentiation of~\eqref{eq:V}, 
 \begin{align*} 
\bar{z}'(t)=-\frac{\bar{V}'(t)}{A}
=\frac{\Qu(t)-\bar{Q}(t)}{A},
\quad\text{where}\quad
\bar{Q}(t)\coloneqq \begin{cases}
-\Qe(t)<0 & \text{if $t\in T_\mathrm{e}$}, \\
\Qf(t)\geq 0 &  \text{if $t\in T_\mathrm{f}$.} 
\end{cases}
\end{align*}
Clearly, $\bar{z}$ is given at any time, so the model under consideration is a {\em moving} boundary problem, 
 but not a {\em free} boundary problem. That said, we emphasize that because of the volumetric flows, no boundary 
 {\em conditions}  need to be imposed since
  the conservation law implies natural output concentrations when reactions and sedimentation are assumed to  occur inside the tank only.

\subsection{A model of reactive settling with moving boundary} \label{subsec:model}

\subsubsection{Biochemical reaction model and solid and liquid components} 
Two constitutive functions describe  the sedimentation-com\-pres\-sion process of the flocculated particles that consist of several components. 
These functions are stated in terms of the solids in suspension~$X$. 
This quantity equals  the sum of either  all or of most of the particulate concentrations; the precise definition of~$X$ depends on the specific   reaction model. 
Any biochemical reaction model can be used such as one of the  standard ASMx activated sludge models.
Within the ASMx models concentrations are usually expressed in terms of more easily  measurable units such as chemical oxygen demand (COD) (cf.\ Table~\ref{table:AMS1_vari} in Appendix~A), wherefore conversion factors have to be used to obtain the mass concentrations.
We here use the ASM1 model (see Appendix~A), in which the particulate concentrations are (in ASM1 units)
\[
X_\mathrm{I}, X_\mathrm{S}, X_\mathrm{B,H}, X_\mathrm{B,A}, X_\mathrm{P}, X_{\rm ND},
\]
and the corresponding definition of the total suspended solids concentration is 
\begin{equation*}
X \coloneqq  c(X_\mathrm{I} + X_\mathrm{S} + X_\mathrm{B,H} + X_\mathrm{B,A} + X_\mathrm{P}), \quad\text{where $ c=0.75\, {\rm g/(g\, COD)}$}.
\end{equation*}
The concentration $X_{\rm ND}$ is not in the definition of~$X$, since $X_{\rm ND}$ represents the nitrogen that is already part of~$X_\mathrm{S}$.
To ensure (for mathematical reasons) that the total solids concentration~$X$ equals the sum of all particulate components, we replace 
 the variable~$X_\mathrm{S}$ by 
$X_\mathrm{S-ND}\coloneqq X_\mathrm{S}-X_\mathrm{ND}$,
and  define (in ASM1 units)
\begin{align*}
\bC& \coloneqq \big(X_\mathrm{I}, X_\mathrm{S-ND}, X_\mathrm{B,H}, X_\mathrm{B,A}, X_\mathrm{P}, X_{\rm ND}\big)^\rmT \quad 
 \text{(i.e., $k_{\boldsymbol{C}} =6$),} \\
\bS& \coloneqq \big(S_\mathrm{I},S_\mathrm{S},S_\mathrm{O},S_\mathrm{NO},S_\mathrm{NH},S_\mathrm{ND}\big)^\rmT
\quad  \text{(i.e., $k_{\boldsymbol{S}} =6$).}
\end{align*}
Moreover, we define
\begin{equation}\label{eq:Xdef}
X\coloneqq  c \big(C^{(1)}+\cdots + C^{(k_{\bC})}\big),\qquad p^{(k)} X\coloneqq c C^{(k)},\qquad \bp X\coloneqq c\bC.
\end{equation}
Similar conversion factors as~$ c$ appear for the soluble concentrations; however, we will divide these factors away directly, since the left-hand sides of the governing equations to be presented are linear in $\bC$ and $\bS$ apart from the coefficients, which are nonlinear functions of~$X$.

\subsubsection{Reaction terms} 
 The nonlinear reaction terms are given by
\begin{equation*}
\begin{split}
&\bLambda_{\bC}\bR_{\bC}(\bC,\bS),\quad\text{where}\quad\bR_{\bC}(\bC,\bS)\coloneqq\bsigmaC\br(\bC,\bS),\\
&\bLambda_{\bS}\bR_{\bS}(\bC,\bS),\quad\text{where}\quad\bR_{\bS}(\bC,\bS)\coloneqq\bsigmaS\br(\bC,\bS),
\end{split}
\end{equation*}
which model the increase in (COD) concentration per time unit; see Table~\ref{table:AMS1_vari}.
Here, $\bLambda_{\bC}=\operatorname{diag}( c, c, c, c, c,1)$ and $\bLambda_{\bS}$ are diagonal matrices with conversion factors, $\bsigmaC$ and $\bsigmaS$ are constant stoichiometric matrices, and $\br(\bC,\bS)\geq\bzero$ is a vector of nonlinear functions modelling the reaction processes; see Appendix~\ref{app}.

We assume that if a solid component is not present, $p^{(k)}=0$, then no more such can vanish, i.e.\  
$\smash{R_{\boldsymbol{C}}^{(k)}(\bC,\bS)|_{p^{(k)}=0}=0}$, where $\smash{R_{\boldsymbol{C}}^{(k)}}$
denotes the $k$-th component of $\boldsymbol{R}_{\boldsymbol{C}}$. 
Finally, to be able to  establish an invariant-region property for the numerical solution, we make some  additional technical assumptions.
To ensure that  the numerical solution for the solids does not exceed the maximal concentration~$\hat{X}$, we assume that 
\begin{equation}\label{eq:techRC}
\text{there exists an $\varepsilon>0$ such that $\bR_{\bC}(\bC,\boldsymbol{S})=\bzero$ for all $X\geq\hat{X}-\varepsilon$}.
\end{equation}
This condition means that when the concentration~$X$ is near the maximal one~$\hat{X}$, biomass cannot grow any more.
To obtain positivity of component~$k$ of the concentration vector~$\bC$, we let 
\begin{equation*}
I_{\bC,k}^-:=\big\{l\in\mathbb{N}:\sigma_{\bC}^{(k,l)}<0\big\},
\qquad
I_{\bC,k}^+:=\big\{l\in\mathbb{N}:\sigma_{\bC}^{(k,l)}>0\big\},
\end{equation*}
denote the sets of indices~$l$ that have negative and positive stoichiometric coefficients, respectively, and assume that 
\begin{align*}
\text{if $l\in I_{\bC,k}^-$, then $r^{(l)}(\bC,\bS)=\bar{r}^{(l)}(\bC,\bS)C^{(k)}$ with $\bar{r}^{(l)}$ bounded}.
\end{align*}

\subsubsection{Bulk velocities and constitutive equations}  
We define the characteristic function $\chi_\omega=1$ if the statement~$\omega$ is true, otherwise $\chi_\omega=0$.
The characteristic function for the tank is thus $\gamma(z,t)=\chi_{\{\bar{z}(t)<z<B\}}(z)$.
The bulk velocity of the mixture in the tank and below it in the underflow pipe is defined as
$\smash{q(z,t) \coloneqq \qu(t)\chi_{\{z>\bar{z}(t)\}}}$  
and the excess velocity due to sedimentation and compression is
\begin{equation*}%\label{eq:v}
v\coloneqq v(X,\partial_z X,z,t) 
\coloneqq  \gamma(z,t)\vhs(X)\left(
1-\dfrac{\rho_X\sme'(X)}{Xg\Delta\rho}\partial_z{X}
\right)
= \gamma(z,t)\big(\vhs(X) - \partial_z{D(X)}\big)
\end{equation*}
(see \cite{SDwatres3,Burger&K&T2005a}), where
\begin{equation*}
D(X) \coloneqq  \int_{X_c}^{X}d(s)\,{\rm d}s,
\qquad
d(X)\coloneqq \vhs(X)\dfrac{\rho_X\sme'(X)}{gX\Delta\rho}.
\end{equation*}
Here, $\Delta\rho\coloneqq \rho_X-\rho_L$ is the density difference of the flocculated particles and the liquid, $g$ is the acceleration of gravity, $\vhs(X)$ is the hindered-settling velocity function, which is assumed to satisfy 
\begin{align} \label{vhsass} 
\vhs(X) \begin{cases} >0 & \text{for $X\in[0,\hat{X})$,} \\  
= 0 & \text{for $X\geq\hat{X}$}, 
\end{cases} \end{align} 
where~$\hat{X}<\rho_X$ is a maximum packing concentration (for computational purposes).
The second constitutive function is the effective solids stress~$\sme(X)$, which satisfies 
\begin{align*} 
\sme'(X)  \begin{cases} = 0 & \text{for $X \leq X_\mathrm{crit}$,} \\ > 0 &  \text{for $X>X_\mathrm{crit}$,} 
\end{cases} \end{align*}   
where $X_\mathrm{crit}$ is a critical concentration above which the particles touch each other and form a network that can bear a certain stress.

\subsubsection{Balance laws} 
The balance law of each material component gives the  system of PDEs 
\begin{align} \label{finalmod} \begin{split} 
\partial_t{\bC}+\partial_z\big(\mathcal{V}_{\bC}(X,\partial_zX,z,t)\bC\big)& = \delta \bigl(z-\bar{z}(t)\bigr)\qf(t)\bCf + \gamma(z,t)\bR_{\bC}(\bC,\bS), \\
\partial_t{\bS} +\partial_z\big(\mathcal{V}_{\bS}(X,\partial_zX,z,t)\bS\big) 
& = \delta \bigl(z-\bar{z}(t)\bigr)\qf(t)\bSf
+ \gamma(z,t)\bR_{\bS}(\bC,\bS), \end{split} 
\end{align} 
modelling reactive settling for~$z\in\mathbb{R}$, where we  have divided away the COD factors~$c$ etc.\ in each equation, 
 $\delta(\cdot)$ [m$^{-1}$] is the delta symbol, and the total velocities are
\begin{equation*}
\begin{split}
\mathcal{V}_{\bC}=\mathcal{V}_{\bC}(X,\partial_zX,z,t) &\coloneqq  q(z,t) + \gamma(z,t) \big(\vhs(X) - \partial_z{D(X)}\big), \\
\mathcal{V}_{\bS}=\mathcal{V}_{\bS}(X,\partial_zX,z,t) &\coloneqq  q(z,t)
- \gamma(z,t)\frac{ (\vhs(X) - \partial_z{D(X)}) X }{\rho_X - X}.
\end{split}
\end{equation*}
The pipe of extraction is modelled as a half line $x\geq 0$ (upwards) with $x=0$ coupled at $z=\bar{z}(t)$.
Any concentration is assumed to follow the advection equations 
\begin{equation}\label{eq:pipe}
\partial_t\bC_\mathrm{pipe}+\qe(t)\partial_x\bC_\mathrm{pipe}=\boldsymbol{0} , \quad 
\partial_t\boldsymbol{S}_\mathrm{pipe}+\qe(t)\partial_x\boldsymbol{S}_\mathrm{pipe}= \boldsymbol{0}. 
\end{equation} 
The extraction concentrations in the pipe are given by complicated formulas due to the moving boundary; see~\cite{SDAMM_SBR1}.
With the  variable change suggested below, these will be obtained more easily.

\subsubsection{Equations during the react stage} 

During periods of mixing, the system of PDEs~\eqref{finalmod} reduces to the system of ODEs 
%\begin{subequations}\label{eq:mixingCS}
\begin{align*}
\bar{V}(t) \boldsymbol{C}' (t)  & = \big(\Qu(t)-\bar{Q}(t)\big)\bC + 
\Qf(t)\bCf(t) + \bar{V}(t)\bR_{\bC}(\bC,\bS), \\
\bar{V}(t)  \boldsymbol{S}' (t)   & = \big(\Qu(t)-\bar{Q}(t)\big)\bS + 
\Qf(t)\bSf(t) + \bar{V}(t)\bR_{\bS}(\bC,\bS),
\end{align*}
for the homogeneous concentrations in $\bar{z}(t)<z<B$,  
where all concentrations depend only on time since they are averages in the tank.
In the region $z<\bar{z}(t)$ all concentrations are zero, whereas the outlet concentrations are $\bC_{\rm u}(t)=\bC(t)$ and $\bC_{\rm e}(t)=\bC(t)$  (if $Q_\mathrm{e}(t)>0$) (analogously for~$\bS$).

\subsection{Model equations on a fixed domain} \label{model:fixeddomain} 
\begin{figure}[t]
\includegraphics[scale=0.55]{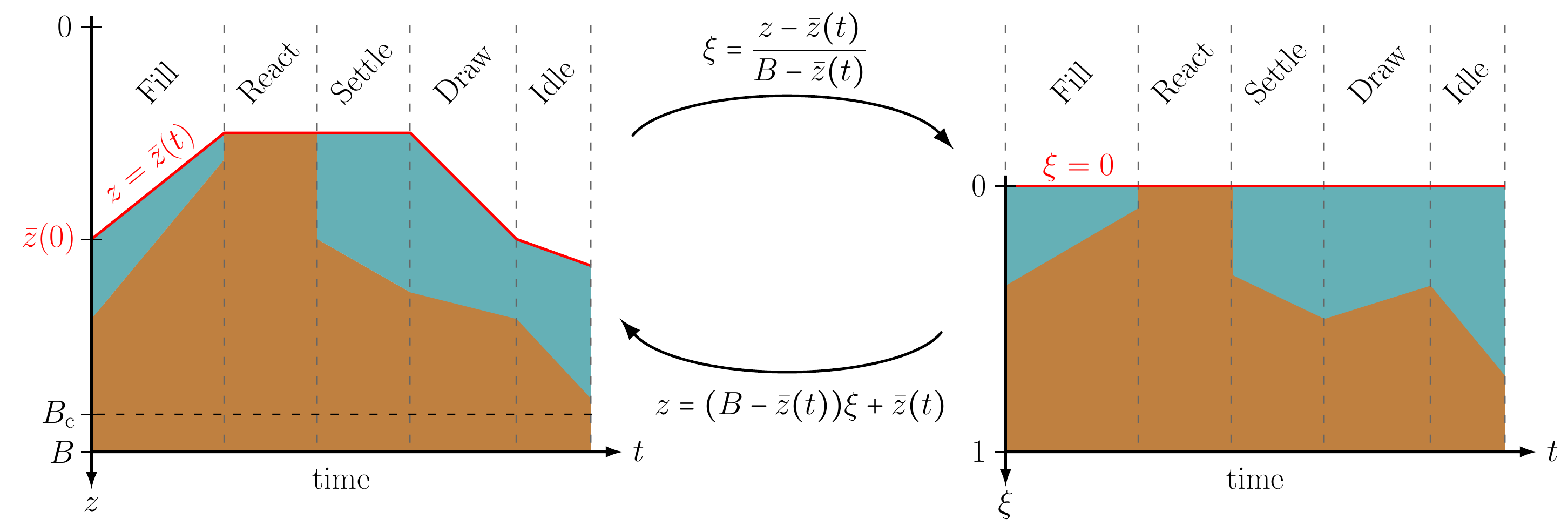}
 \caption{Evolution of a concentration profile varying with respect to time at the five SBR stages in the space variable $z$ (left) and new variable $\xi$ (right). The moving boundary $z=\bar{z}(t)$ (red line) is mapped to the constant line $\xi = 0$. The mappings $z=z(\xi,t)$ and $\xi = \xi(z,t)$ are shown in the middle. } \label{fig:changevar}
\end{figure}

 To solve the model equations \eqref{finalmod}--\eqref{eq:pipe} without complicated formulas for the outlet concentrations, we transform the time-varying interval $[\bar{z}(t),B]$ to the fixed domain $[0,1]$ for all $t\geq 0$ by introducing the space variable, for all $z\in\mathbb{R}$,
\begin{equation}\label{eq:xidef}
 \xi=\xi(z,t) \coloneqq \dfrac{z-\bar{z}(t)}{B-\bar{z}(t)}\quad 
\quad \Leftrightarrow\quad z= (B-\bar{z}(t))\xi+\bar{z}(t).
 \end{equation}
where it is assumed that $B-\bar{z}(t)\geq B_\mathrm{c}>0$ for all $t\geq 0$ for some constant~$B_\mathrm{c}$.
We define the unknowns of the model in the new variable $\xi$ by $\tilde{X}(\xi(z,t),t)\coloneqq X(z,t)$ (analogously for the rest of the unknowns and space dependent functions).
The partial derivatives are 
 \begin{align*}
  \partial_t  \xi  =
  -\bar{z}'(t)\dfrac{1-\xi}{B-\bar{z}(t)} \eqqcolon\alpha(\xi,t),\qquad 
  \partial_z  \xi  = \dfrac{1}{B-\bar{z}(t)}  \eqqcolon \beta(t).
\end{align*}
Clearly,  the sign of $\alpha$ depends uniquely on the slope $\bar{z}'$ and therefore on $t$, while $\beta(t)>0$  for all $t>0$.  
Then the time and space partial derivatives are transformed as
\begin{align*}
 \partial_t  X  &= 
\partial_{\xi}   \tilde{X} \partial_t  \xi  +  \partial_t  \tilde{X}  =  
 \alpha(\xi,t) \partial_{\xi}   \tilde{X}  +  \partial_t  \tilde{X},\quad 
  \partial_z  X   = 
 \partial_{\xi} \tilde{X} \partial_z  \xi  = \beta(t) \partial_{\xi} \tilde{X}  =  \partial_{\xi} \bigl(\beta(t)\tilde{X} \bigr). 
\end{align*}
It will be convenient to rewrite the  term  
\begin{equation*}
 \alpha(\xi,t) \partial_{\xi}   \tilde{X} 
 = \partial_{\xi} \bigl(\alpha(\xi,t)\tilde{X}\bigr)- \partial_{\xi}  \alpha (\xi,t)\tilde{X} 
 = \partial_{\xi} \bigl(\alpha(\xi,t)\tilde{X}\bigr) - \beta(t)\bar{z}'(t)\tilde{X} 
\end{equation*}
(anticipating the transformation of $\partial_t\bC$ in~\eqref{finalmod} etc.). 
The characteristic function becomes
\begin{equation*}
\gamma(z,t) =\chi_{\{\bar{z}(t)<z<B\}}(z)= \tilde{\chi}_{\{0<\xi<1\}}(\xi)=:\tilde{\gamma}(\xi).
\end{equation*}
The delta symbol in \eqref{finalmod} is formally transformed via the Heaviside function~$H$ as 
\begin{align*}
H \bigl(z-\bar{z}(t) \bigr) &=H\big(\xi(B-\bar{z}(t))\big)=:\tilde{H}(\xi),\\
\delta \bigl(z-\bar{z}(t) \bigr)&= H' \bigl(z-\bar{z}(t) \bigr) 
=  \tilde{H}'(\xi) 
= \tilde{H}'(\xi)\beta(t)
= \tilde{\delta}(\xi)\beta(t).
\end{align*}
The bulk velocity becomes $q(z,t) =\qu(t)\chi_{\{z>\bar{z}(t)\}}(z)=\qu(t)\tilde{H}(\xi).$
Equations~\eqref{finalmod} can be written in the new variable $\xi$ as the system
\begin{align} \label{finalmodxi} \begin{split} 
\partial_t{\bC}  + \partial_\xi\big(
\tilde{\mathcal{V}}_{\boldsymbol{C}}(X,\partial_\xi X, \xi,t)\bC\big)& = \beta(t)\bar{z}'(t)\bC + \delta (\xi)\beta(t)\qf(t)\bCf + \gamma(\xi)\bR_{\bC}(\bC,\bS) \\
\partial_t{\bS}  + \partial_\xi\big(\tilde{\mathcal{V}}_{\boldsymbol{S}}(X,\partial_\xi X, \xi,t)\bS\big)  & =  \beta(t)\bar{z}'(t)\bS + \delta (\xi)\beta(t)\qf(t)\bSf  + \gamma(\xi)\bR_{\bS}(\bC,\bS), 
\end{split} 
\end{align} 
where we directly have removed the tilde above~$\bC$, $\bS$, $X$, $\gamma$, $H$, and $\delta$, and 
\begin{align*} 
 \tilde{\mathcal{V}}_{\boldsymbol{C}} (X,\partial_\xi X, \xi,t) & \coloneqq  \tilde{q}(\xi,t) + \gamma(\xi)\beta(t)\left( \vhs(X)- \beta(t)\partial_\xi D(X)\right),  \\
 \tilde{\mathcal{V}}_{\boldsymbol{S}} (X,\partial_\xi X, \xi,t) & \coloneqq  \tilde{q}(\xi,t) -
\gamma(\xi)\beta(t)\frac{f(X) - \beta(t)\partial_\xi \mathcal{D}(X) }{\rho_X - X},
\end{align*}
where $\tilde{q}(\xi,t)\coloneqq \alpha(\xi,t) + \beta(t)\qu(t)\tilde{H}(\xi)$ (temporary definition), 
\begin{align} \label{fXdef} 
f(X)\coloneqq\vhs(X)X, 
\end{align} 
and
\begin{equation*}
\mathcal{D}(X)\coloneqq \int_{X_\mathrm{c}}^{X}a(s)\,\rmd s,\quad\text{where}\quad a(s)\coloneqq sd(s).
\end{equation*}
Clearly, the governing equations~\eqref{finalmodxi} can be expressed in  the form \eqref{syst:intro} if we define 
\begin{align*} 
U_{\boldsymbol{C}} (X, \xi, t) \coloneqq  \tilde{q}(\xi,t) + \gamma(\xi)\beta(t) \vhs(X), 
 \quad U_{\boldsymbol{S}} (X, \xi, t) \coloneqq  \tilde{q}(\xi,t) - \gamma(\xi)\beta(t)\frac{f(X)}{\rho_X - X}.
 \end{align*}

Equations~\eqref{finalmodxi} hold for $\xi\in\mathbb{R}$ when~$t\in T_\mathrm{f}$ if it is assumed that all concentrations are zero above the surface $\xi<0$.
In particular, for the underflow zone $z>B \Leftrightarrow \xi >1$, the equations are
\begin{equation*}
\partial_t\bC+\big(\alpha(\xi,t) + \beta(t)\qu(t)\big)\partial_\xi\bC=\beta(t)\bar{z}'(t)\bC, 
 \quad \partial_t\boldsymbol{S}+\big(\alpha(\xi,t) + \beta(t)\qu(t)\big)\partial_\xi \boldsymbol{S} =\beta(t)\bar{z}'(t) \boldsymbol{S}.
\end{equation*}

To ensure that  fluxes have  correct units across the surface during  extraction ($t\in T_\mathrm{e}$), 
 we also need to transform the extraction pipe.
The pipe is originally modelled by the upwards-pointing $x$-axis with bulk flow upwards~$\qe(t)$ and coupled to the $z$-axis by $-x=z-\bar{z}(t)$.
Consequently, the transformation for the extraction pipe is
\begin{equation}\label{eq:xidefx}
\xi(x,t)\coloneqq\frac{-x}{B-\bar{z}(t)},
\end{equation}
and $\bC_\mathrm{pipe}(x,t)=\boldsymbol{\tilde{C}}(\xi(x,t),t)$.
Then we get
\begin{align*}
 \partial_t \bC_\mathrm{pipe}  &= - \partial_{\xi} \boldsymbol{\tilde{C}} \xi\bar{z}'(t)\beta(t) + \partial_t \boldsymbol{\tilde{C}} 
= -\bigl( \partial_{\xi} (\xi\boldsymbol{\tilde{C}}) -\boldsymbol{\tilde{C}}\bigr)\bar{z}'(t)\beta(t) + \partial_t \boldsymbol{\tilde{C}},\quad 
 \partial_x \bC_\mathrm{pipe} = - \partial_{\xi} \boldsymbol{\tilde{C}} \beta(t).
\end{align*}
Equations~\eqref{eq:pipe} are  transformed to (we immediately remove the tildes)
\begin{align*}  
\partial_t\bC -\partial_\xi\big(\beta(t)(\xi\bar{z}'(t) + \qe(t))\bC\big)&=-\bar{z}'(t)\beta(t)\bC,  \\ 
\partial_t\boldsymbol{S} -\partial_\xi\big(\beta(t)(\xi\bar{z}'(t) + \qe(t)) \boldsymbol{S} \big)&=-\bar{z}'(t)\beta(t) \boldsymbol{S}, \quad \xi <0. 
\end{align*}  
(In comparison to~\eqref{finalmodxi}, there is a minus sign on the right-hand side here.) 
With the bulk velocity  redefined as
\begin{equation*}
\tilde{q}(\xi,t)\coloneqq 
\begin{cases}
0 & \text{if $\xi<0$ and $\qe(t)=0$},\\
-\beta(t)(\xi\bar{z}'(t) + \qe(t)) = -\beta(t)(\xi(\qu(t)+\qe(t)) + \qe(t)) & 
 \text{if $\xi<0$ and  $\qe(t)>0$,} \\
\alpha(\xi,t) + \beta(t)\qu(t) = \beta(t)(\xi\bar{z}'(t) +\qu(t)- \bar{z}'(t)) & 
 \text{if $\xi>0$},
\end{cases}
\end{equation*}
we thus get the  governing equations 
\begin{align}\label{finalmodxi_e}
\begin{split}
\partial_t{\bC}  + \partial_\xi\big(
\tilde{\mathcal{V}}_{\boldsymbol{C}}(X,\partial_\xi X, \xi,t)\bC\big)& = \sgn(\xi)\bar{z}'(t)\beta(t)\bC + \delta (\xi)\beta(t)\qf(t)\bCf + \gamma(\xi)\bR_{\bC}(\bC,\bS),  \\
\partial_t{\bS}  + \partial_\xi\big(\tilde{\mathcal{V}}_{\boldsymbol{S}}(X,\partial_\xi X, \xi,t)\bS\big)  & = \sgn(\xi)\bar{z}'(t)\beta(t)\bS + \delta (\xi)\beta(t)\qf(t)\bSf + \gamma(\xi)\bR_{\bS}(\bC,\bS),
\end{split} 
\end{align} 
for $\xi\in\mathbb{R}$ and $t\in T_\mathrm{e}$. 
 These equations hold for all $t>0$ if we for $t\in T_\mathrm{f}$ define all concentrations in $\xi<0$ to be zero; then the system is reduced to~\eqref{finalmodxi}. 
The salient point of these transformations  is that the outlet concentrations are now simply defined by
\begin{equation}\label{eq:outputs}
\begin{alignedat}2
\bC_\mathrm{e}(t) &= \bC(0^-,t),&\qquad& \bS_\mathrm{e}(t) = \bS(0^-,t) \qquad \text{if $t\in T_\mathrm{e}$,} \\
\bC_\mathrm{u}(t) &= \bC(1^+,t),&\qquad& \bS_\mathrm{u}(t) = \bS(1^+,t)\qquad \text{if $t\in T_\mathrm{f}$ and $\Qu(t)>0$.} 
\end{alignedat}
\end{equation}
For other times   the outlet concentrations are defined to be zero.

During periods of mixing and $t\in T_\mathrm{f}$, the following ODEs for time-dependent concentrations~$\bC(t)$ and $\bS(t)$ are obtained by averaging Equations~\eqref{finalmodxi}, i.e., integrating over the interval $[0^-,1)$ when the convective and diffusive terms are zero:
\begin{equation}\label{eq:mixing}
\begin{aligned}
 \boldsymbol{C}'(t) & = \beta(t)\bar{z}'(t)\bC + 
\beta(t)\qf(t)\bC_\mathrm{f}(t) + \bR_{\bC}(\bC,\bS), \\
\boldsymbol{S}'(t) & = \beta(t)\bar{z}'(t)\bS + 
\beta(t)\qf(t)\bS_\mathrm{f}(t) + \bR_{\bS}(\bC,\bS).
\end{aligned}
\end{equation}
The same ODEs are indeed obtained for $t\in T_\mathrm{e}$ (then $\bC_\mathrm{f}(t)=\bS_\mathrm{f}(t)=0$).

\subsection{Model equations with percentages} \label{model:percentages} 

The restrictive part of the explicit CFL condition of an explicit scheme is due to the second-order derivative term containing the function $D(X)$, which depends only on the scalar variable~$X$. 
The governing model will be rewritten so that the convective and diffusive terms are clearly seen.
Furthermore, to establish boundedness on the total particulate concentration, $0\leq X\leq \hat{X}$, we will rewrite the model in terms of percentages~$\bp$; see~\eqref{eq:Xdef}.
We define the flux and reaction term of the total solids concentration $X$ by
\begin{align*}
F(\xi,t,X) &\coloneqq\tilde{q} (\xi, t) X +\gamma (\xi) \beta (t) f(X),\qquad
{R}(\bC,\bS)\coloneqq c\sum_{k=1}^{k_{\bC}}R_{\bC}^{(k)}(\bC,\bS).
\end{align*}
%and source terms 
%\begin{align*}
%\Psi(\xi,t,X,\bp,\bS)&\coloneqq\beta(t)\delta (\xi)\qf(t)\Xf + \gamma(\xi){R}(\bp X/c,\bS),\\
%\Psi_{\bp}(\xi,t,X,\bp,\bS)&\coloneqq\beta(t)\delta (\xi)\qf(t)\bp_\mathrm{f} X_\mathrm{f} + \gamma(\xi) c\bR_{\bC}(\bp X/c,\bS),\\
%\Psi_{\bS}(\xi,t,X,\bp,\bS)&\coloneqq\beta(t)\delta (\xi)\qf(t)\bSf + \gamma(\xi)\bR_{\bS}(\bp X/c,\bS),
%\end{align*}
%where $\bp_\mathrm{f} X_\mathrm{f}= c\bCf$.
Then we can write
\begin{align*}
\tilde{\mathcal{V}}_{\boldsymbol{C}}X &= F(\xi,t,X) - \gamma(\xi)\beta(t)^2\partial_{\xi}\mathcal{D}(X),\quad 
\tilde{\mathcal{V}}_{\boldsymbol{S}} = \frac{\tilde{q}(\xi,t)\rho_X - F(\xi,t,X) + \gamma(\xi)\beta(t)^2\partial_{\xi}\mathcal{D}(X)}{\rho_X-X}.
\end{align*}
By first multiplying the first equation of~\eqref{finalmodxi_e} by $c$ and adding the vector components corresponding to~\eqref{eq:Xdef}, one obtains the scalar equation 
\begin{align}\label{eq:PDE_X} \begin{split} 
\partial_t{X}  + \partial_\xi F(\xi,t,X)&  = \partial_{\xi}\big(\gamma(\xi)\beta(t)^2\partial_{\xi}\mathcal{D}(X)\big)\\
& \quad  + \beta(t)\sgn(\xi)\bar{z}'(t)X + \beta(t)\delta (\xi)\qf(t)\Xf + \gamma(\xi){R}(\bp X/c,\bS)
\end{split} \end{align} 
for~$X$.  The concentrations $\bC=\bp X/c$ and $\bS$ are given by the system~\eqref{finalmodxi_e}, which we now can rewrite with the unknowns~$X$, $\bp$ and~$\bS$:
\begin{align}
\partial_t{(\bp X)}  + \partial_\xi(F(\xi,t,X)\bp) &= \partial_\xi\big(\bp\gamma(\xi)\beta(t)^2\partial_{\xi}\mathcal{D}(X)\big)+ \beta(t)\sgn(\xi)\bar{z}'(t)\bp X\notag\\
& \qquad + \beta(t)\delta (\xi)\qf(t)\bp_\mathrm{f} X_\mathrm{f} + \gamma(\xi) c\bR_{\bC}(\bp X/c,\bS),\label{eq:PDE_p}\\
\partial_t{\bS}  + \partial_\xi\big(\tilde{\mathcal{V}}_{\boldsymbol{S}}\bS\big)  & = \beta(t)\sgn(\xi)\bar{z}'(t)\bS + \beta(t)\delta (\xi)\qf(t)\bSf + \gamma(\xi)\bR_{\bS}(\bp X/c,\bS).  \label{eq:PDE_S}
\end{align}
Not all equations in~\eqref{eq:PDE_p} need to be solved; one may solve  only the first $k_\mathrm{\bC}-1$ ones and set $p^{(k_{\bC})}=1-(p^{(1)}+\cdots+p^{(k_{\bC}-1)})$.
The output concentrations for~$\bp$ and $X$ are defined as in~\eqref{eq:outputs}.

The mixing ODEs~\eqref{eq:mixing} are converted analogously:
\begin{equation}\label{eq:mixingpXS}
\begin{aligned}
X'(t) & = \beta(t)\bar{z}'(t)X + \beta(t)\qf(t)X_\mathrm{f}(t) + {R}(\bp X/c,\bS), \\
(\boldsymbol{p}  X)'(t) & = \beta(t)\bar{z}'(t)\bp X + \beta(t)\qf(t)\bp_\mathrm{f}(t)X_\mathrm{f}(t) + c\bR_{\bC}(\bp X/c,\bS), \\
\boldsymbol{S}'(t) & = \beta(t)\bar{z}'(t)\bS + \beta(t)\qf(t)\bS_\mathrm{f}(t) + \bR_{\bS}(\bp X/c,\bS).
\end{aligned}
\end{equation}

\section{Explicit numerical scheme} \label{sec:numscheme}

\subsection{Discretization in space and time}\label{subsec:discr}  

To discretize the PDE system~\eqref{eq:PDE_X}--\eqref{eq:PDE_S}, we define $\Delta\xi\coloneqq 1/(N+1/2)$ for an integer~$N$,  $\xi_j\coloneqq j\Delta\xi$, and let cell~$j$ be the interval $I_j\coloneqq [\xi_{j-1/2},\xi_{j+1/2}]$; see Figure~\ref{fig:scheme_celd}.
\begin{figure}[tbp]
	\centering
	\includegraphics[scale=0.4]{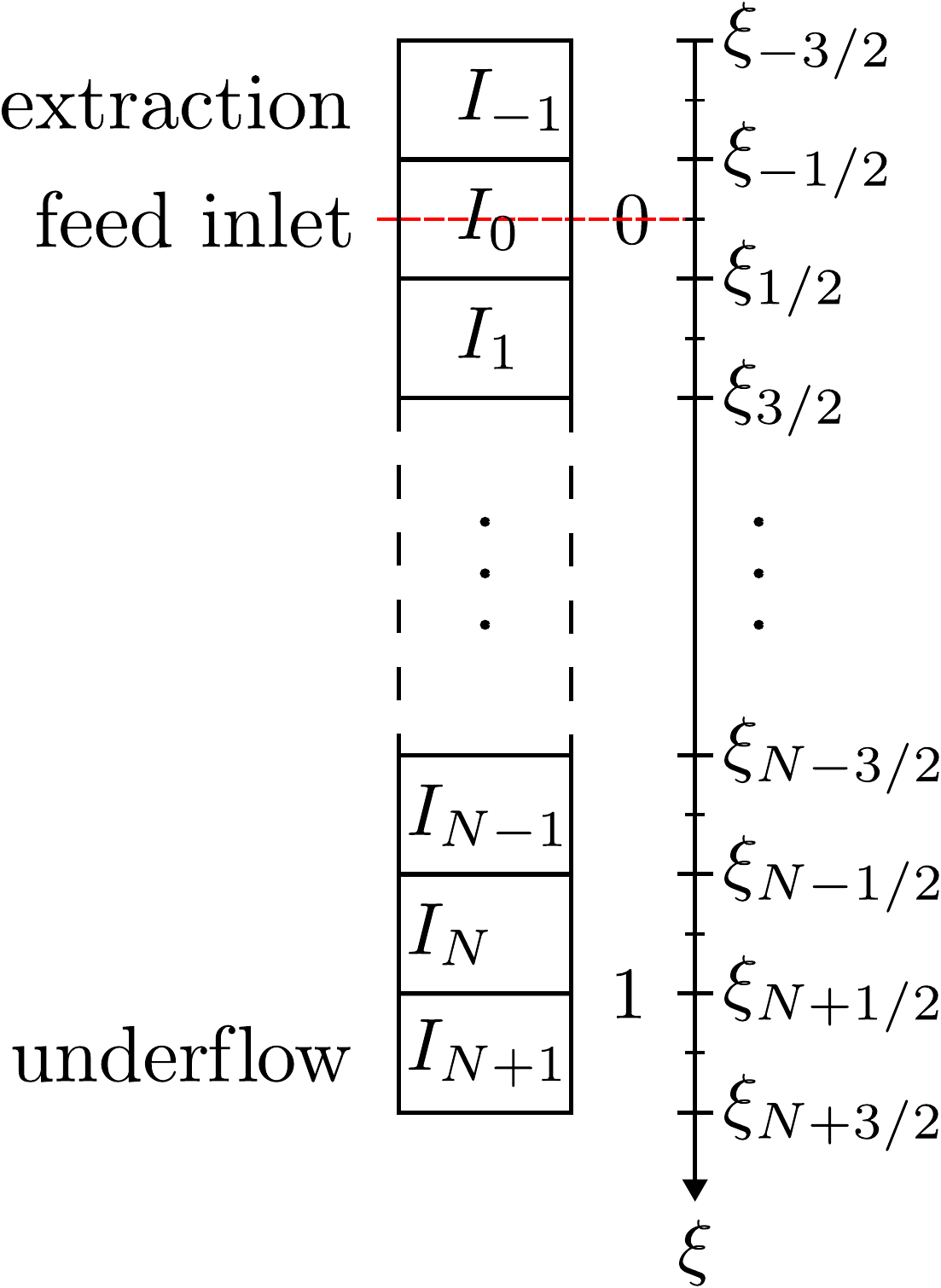}
	\caption{Schematic of the division of the computational domain into cells.} \label{fig:scheme_celd}
\end{figure}%
Thus, the feed inlet is located in the middle of~$I_0$, which makes the numerical fluxes at the cell boundaries easy to define.
The bottom of the tank is  located at $\xi = \xi_{N+1/2}=(N+1/2)\Delta\xi=1$.
Given a total simulation time~$T$ and the total number of discrete time points~$N_T$, we define the time step $\tau\coloneqq T/N_T$, which is supposed 
 to satisfy a suitable CFL condition, and the discrete time points~$t^n$, $n = 0,1,\dots, N_T$. 
  The total particulate concentration   is denoted by $\smash{X_j^n\approx X(\xi_{j},t^n)}$, and analogous notation is used for other 
   concentrations.
The underflow concentration is captured by  an additional cell~$j=N+1$ below the tank.
The extraction concentrations can be obtained in an analogous way in an additional cell~$I_{-1}$.
 For both the explicit and implicit  schemes the outlet concentrations are defined by
$\smash{\bp_\mathrm{e}=\bp_{-1}^n}$ and $\smash{\bp_\mathrm{u}=\bp_{N+1}^n}$,
and similarly for~$X$ and~$\bS$. 

\subsection{Explicit scheme}  \label{subsec:explicheme}
For ease of notation, we introduce $a^+=\max \{ 0,a \}$, $a^-=\min \{0,a\}$, the upwind and the divergence operators by
\begin{align*}
 \Upw(a;b,c) 
 &\coloneqq  \max\{a,0\}b + \min\{a,0\}c = a^+b + a^- c,\qquad\mbox{for } a,b,c\in \mathbb{R},\\
[\Delta\boldsymbol{\Phi}]_j^n
&\coloneqq \boldsymbol{\Phi}_{j+1/2}^n-\boldsymbol{\Phi}_{j-1/2}^n.
\end{align*}
We define the Kronecker delta and the characteristic function for the mixture in the tank by
\begin{equation*}
\delta_{j,0}\coloneqq\int_{I_{j}}\delta(\xi)\,\rmd\xi
=\begin{cases}
1&\text{if $j=0$},\\
0 &\text{if $j\neq 0$},
\end{cases}\qquad
\gamma_j\coloneqq \begin{cases}
\gamma(\xi_j)& \text{if $j\neq 0$,} \\
\frac{1}{2} & \text{if $j=0$.} 
\end{cases}
\end{equation*}
 The term $\sgn(\xi)\beta(t)\bar{z}'(t)X(\xi,t)$ in~\eqref{eq:PDE_X}
is approximated  by the average  
\begin{equation*}
\frac{1}{\Delta\xi}\int_{\xi_{-1/2}}^{\xi_{1/2}}\sgn(\xi)\beta^n(\bar{z}')^n X(\xi,)\,\rmd\xi =
\begin{cases}
0 &\text{if $\qe^n>0$},\\
\frac{1}{2}\beta^n(\bar{z}')^n X_0^n = \frac{1}{2}\beta^n(\qu^n-\qf^n) X_0^n  &\text{if $\qe^n=0$},
\end{cases}
\end{equation*}
for $j=0$, where we  recall that $X(\xi,t)=0$ for $\xi<0$ if $\qe(t)=0$, 
and analogously for $\bS$ instead of~$X$.
For the numerical update formulas, we define
\begin{align} \label{kappajndef} 
\kappa^n_j\coloneqq\begin{cases}
1-\tau\beta^n(\bar{z}')^n = 1-\tau\beta^n(\qu^n+\qe^n) & \text{if $j<0$ and $\qe^n>0$},\\
1 & \text{if $j<0$ and $\qe^n=0$ or $j=0$ and $\qe^n>0$},\\
1-\frac{\tau}{2}\beta^n(\qu^n-\qf^n) & \text{if $j=0$ and  $\qe^n=0$,} \\
1+\tau\beta^n(\bar{z}')^n & \text{if $j>0$,} 
\end{cases}
\end{align} 
Finally, we define
\begin{equation*}
\tilde{q}_{j+1/2}^n\coloneqq\tilde{q}(\xi_{j+1/2},t^n) = 
\begin{cases}
0  & \text{if $j=-2, -1$ and $\qe^n=0$,} \\
-\beta^n (\xi_{j+1/2}(\qu^n+\qe^n)+\qe^n ) & \text{if $j=-2, -1$ and $\qe^n>0$,} \\
\alpha_{j+1/2}^n + \beta^n\qu^n &\text{if $j = 0, \ldots, N+1$} .
\end{cases}
\end{equation*}

An explicit approximation of the system~\eqref{eq:PDE_X}--\eqref{eq:PDE_S} is now obtained by following ideas from \cite{SDm2an_reactive,SDIMA_MOL}.
We assume  that the function~$f$ defined by \eqref{fXdef} has a unique maximum at $X^* \in (0, \hat{X})$.
 The  Engquist-Osher numerical flux \cite{eopaper} 
\begin{align}  \label{EOjphdef} 
\mathcal{E}_{j+1/2}^{n} &\coloneqq\gamma_{j+1/2} \biggl(  f(0) + \int_0^{X_j^n}  \max \bigl\{ 0,f'(s) \bigr\} \, \mathrm{d} s 
+ \int_0^{X_{j+1}^n}  \min \bigl\{0,f'(s)  \bigr\} \, \mathrm{d} s \biggr) 
\end{align} 
is used, which for a unimodal flux function~$f$ is
\begin{align*}  
 \mathcal{E}_{j+1/2}^{n} \bigl(X_{j}^n,X_{j+1}^n\bigr)    = \gamma_{j+1/2} 
\begin{cases} 
f (X_{j}^n) & \text{if $X_{j}^n, X_{j+1}^n \leq X^*$,} \\
 f (X^*)  &  \text{if  $X_{j+1}^n \leq X^*<X_{j}^n$,} \\
-f (X^*)  + f(X_{j}^n) +  f(X_{j+1}^n) &  \text{if $X_{j}^n \leq X^*<X_{j+1}^n$,} \\
 f (X_{j+1}^n) &  \text{if $X^*<X_{j}^n, X_{j+1}^n$.}
\end{cases}  
\end{align*} 
The diffusive term and   the  other fluxes  are discretized by
\begin{align}
\mathcal{J}_{j+1/2}^{n} &\coloneqq\gamma_{j+1/2}\dfrac{(\beta^{n})^2}{\Delta \xi}\big(\mathcal{D}(X^{n}_{j+1})-\mathcal{D}(X^{n}_j)\big), 
 \quad j = -2,\dots,N+1,  \label{Jjphdef}  \\
\mathcal{B}_{j+1/2}^{n} &\coloneqq{\rm Upw}(\tilde{q}_{j+1/2}^{n}; X^{n}_j,X^{n}_{j+1}),  \label{Bjphdef}  \\
\mathcal{F}_{j+1/2}^{n} &\coloneqq\mathcal{B}_{j+1/2}^{n} + \beta^{n} \mathcal{E}_{j+1/2}^{n} \nonumber \\
&=  \label{Fjphdef} 
\begin{cases}
0  & \text{if $j=-2, -1$ and $\qe^n=0$,} \\
-\beta^n\big(\xi_{j+1/2}(\qu^n+\qe^n)+\qe^n\big) X^{n}_{j+1} & \text{if $j=-2, -1$ and $\qe^n>0$,} \\
{\rm Upw}(\alpha_{j+1/2}^n + \beta^n\qu^n; X^{n}_j,X^{n}_{j+1}) + \beta^{n} \mathcal{E}_{j+1/2}^{n}  &\text{if $j = 0, \dots, N-1$,} \\
(\alpha_{j+1/2}^n + \beta^n\qu^n)X^{n}_j  &\text{if  $j = N, N+1$,} 
\end{cases}\\
\Phi_{j+1/2}^n& \coloneqq \mathcal{F}_{j+1/2}^{n} - \mathcal{J}_{j+1/2}^{n}.  \label{Phijphdef} 
\end{align}
We have $\alpha_{N+1/2}=0$, and when $\qe>0$, we assume that $\Delta\xi$ is sufficiently small, so that all fluxes at the top $\xi=\xi_{-1/2}$ and bottom $\xi_{N+1/2}$ are directed out of the tank, which means that no boundary values are needed.
The numerical fluxes are then defined, for $j=-2,\dots,N+1$, by
\begin{align}
 \boldsymbol{\Phi}_{\boldsymbol{p},j+1/2}^n & \coloneqq  \Upw(\Phi_{j+1/2}^n;\bp_j^n,\bp_{j+1}^n), \label{eq:Phipflux}\\
 \boldsymbol{\Phi}_{\boldsymbol{S},j+1/2}^n & \coloneqq  \Upw\biggl(\rho_X\tilde{q}_{j+1/2}^n -  \Phi_{j+1/2}^n;\frac{\bS_j^n}{\rho_X-X_j^n},\frac{\bS_{j+1}^{n}}{\rho_X-X_{j+1}^{n}}\biggr). \nonumber  
\end{align}
With an Euler time step, $\lambda \coloneqq \tau/\Delta \xi$, and  $\mu \coloneqq \tau/\Delta \xi^2$, 
 we can formulate the explicit scheme  as follows. 
For the boundary layers $j=-1$ and $j=N+1$, we set 
\begin{alignat}{2}
&X_{-1}^{n+1}=0,\quad \bp_{-1}^{n+1}=\bzero,\quad \bS_{-1}^{n+1}=\bzero \quad&\text{if $\qe^n=0$,}  \label{eq:j=-1_Tf}\\
& X_{N+1}^{n+1}=0,\quad \bp_{N+1}^{n+1}=\bzero,\quad \bS_{N+1}^{n+1}=\bzero \quad&\text{if $\qu^n=0$.}  \label{eq:j=N+1_qu=0}
\end{alignat}
Otherwise, we have for $j=-1,\dots,N+1$,
\begin{align}
X^{n+1}_j &= \kappa_j^n X^{n}_j - \lambda [\Delta \mathcal{F}]_j^n + \lambda [\Delta \mathcal{J}]_j^n + \lambda\delta_{j,0}\beta^nq_{\rm f}^nX_{\rm f}^n + \tau\gamma_j{R}(\bp_j^nX^{n}_j/c,\bS_j^n),\label{eq:scheme_X}\\
 \bp^{n+1}_j X^{n+1}_j &= \kappa_j^n\bp^{n}_jX^{n}_j - \lambda [\Delta\boldsymbol{\Phi}_{\bp}]_j^n + \lambda\delta_{j,0}\beta^nq_{\rm f}^n\bp_{\rm f}^nX_{\rm f}^n + \tau\gamma_jc\bR_{\bC}(\bp_j^nX^{n}_j/c,\bS_j^n),\label{eq:scheme_p}\\
 \bS^{n+1}_j &= \kappa_j^n\bS^{n}_j - \lambda [\Delta\boldsymbol{\Phi}_{\boldsymbol{S}}]_j^n + \lambda\delta_{j,0}\beta^nq_{\rm f}^n\bS_{\rm f}^n + \tau\gamma_j\bR_{\bS}(\bp_j^nX^{n}_j/c,\bS_j^n).\label{eq:scheme_S}
\end{align}
Scheme~\eqref{eq:scheme_X} is solved first, then the others.
If $\smash{X^{n+1}_j=}0$, then the value of $\smash{\bp_j^{n+1}}$ is irrelevant, and can be set to~$\smash{\bp_j^{n+1}\coloneqq\bp_j^{n}}$.
For the cells outside the tank, the scheme for $X$ is reduced to the following.
If $\qe^n=0$, then \eqref{eq:j=-1_Tf} handles $j=-1$; otherwise, we define 
$q_{\mathrm{out}}^n := \qu^n+\qe^n$ and utilize 
\begin{align} \label{eq:update_-1} 
X^{n+1}_{-1} & = \bigl(1-\tau\beta^n q_{\mathrm{out}}^n  \bigr) X^{n}_{-1}  
 + \lambda\beta^n\bigl( (\xi_{-1/2} q_{\mathrm{out}}^n +\qe^n) X^n_0 - (\xi_{-3/2} q_{\mathrm{out}}^n +\qe^n)X^n_{-1}\bigr),
\end{align}
On the other hand, if $\qu^n = 0$,  then \eqref{eq:j=N+1_qu=0} is in effect for  $j=N+1$; otherwise,  \begin{align}
X^{n+1}_{N+1} & = (1+\tau\beta^n q_{\mathrm{out}}^n) X^{n}_{N+1} - \lambda\bigl( (\alpha^n_{N+3/2}+\beta^n\qu^n)X^n_{N+1} - \beta^n\qu^nX^n_{N}\bigr).\label{eq:update_N+1}
\end{align}
Similar update formulas hold for $\bp$ and $\bS$.
The resulting approximate concentrations are transformed back to the original $z$- and $x$-coordinates via~\eqref{eq:xidef} and \eqref{eq:xidefx}, respectively.

\subsection{Numerics during mixing} \label{subsec:num-mix} 
Suppose a (PDE or numerical) solution $X(\xi,T_0)$ (or $\smash{p^{(k)}(\xi,T_0)}$ or $\smash{S^{(k)}(\xi,T_0))}$ is known at $t=T_0=t_{n_0}$ when a period of complete mixing starts.
The initial concentrations for the ODEs~\eqref{eq:mixingpXS} are defined as the averages (analogously for $p^{(k)}$ and $S^{(k)}$)
\begin{equation*}
X(T_0)\coloneqq 
\int_{0}^{1}X(\xi,T_0)\,\rmd\xi
\approx \Delta\xi \biggl(\frac{X_0^{n_0}}{2}+ {X}_1^{n_0} + \dots + {X}_N^{n_0}   \biggr).
\end{equation*}
The ODE system~\eqref{eq:mixingpXS} can then approximately be  integrated in time with an Euler step.
If an ODE mixing period ends at $t=t^{{n}}$ with the values ${X}(t^n)$, and the PDE model is to be simulated thereafter, then the value of each component~$k$ is allocated to all cells in the tank; ${X}_j^{{n}}\coloneqq {X}(t^n)$, $j=0,\ldots,N$.

\subsection{CFL condition} \label{subsec:CFL} 
We denote the solution variables by $\mathbfcal{U}\coloneqq(X,\bp^\rmT,\bS^\rmT)^\rmT$ and define the constants  
\begin{alignat*}2
&\|f\|\coloneqq\max\limits_{0\le X\le\hat{X}}|f(X)|, &\quad& M_R\coloneqq\frac{1}{c}\sup_{\boldsymbol{\mathcal{U}}\in\Omega\atop 1\le k\le k_{\bC}}
 \left|\frac{\partial R }{\partial C^{(k)}} \right|,\\
&M_{q1}\coloneqq\max_{0\le t\le T}\{\qu(t)+\qe(t),\qf(t)\},&\quad&M_{q2}\coloneqq \max_{0\le t\le T}\big(\max\{\qf(t),\qe(t)\}+2\qu(t)\big)\\
&\zeta\coloneqq\frac{1}{B-B_\mathrm{c}},&\quad&C_1\coloneqq \zeta(M_{q2}+\|f'\|),\\
&C_2\coloneqq\zeta^2 \|a\|, &\quad& M_{\boldsymbol{\xi}}:=\sup_{\boldsymbol{\mathcal{U}}\in\Omega, \atop 1\le k\le k_{\boldsymbol{\xi}}}
\sum_{l\in I_{\boldsymbol{\xi},k}^-}|\sigma_{\boldsymbol{\xi}}^{(k,l)}|\bar{r}_{\boldsymbol{\xi}}^{(l)} (\bC,\bS),
\quad\boldsymbol{\xi}\in\{\bC,\bS\}. 
\end{alignat*}
The time step~$\tau$ and the spatial mesh width~$\Delta \xi$ should satisfy the CFL condition
\begin{align} \label{eq:CFL} \tag{CFL} 
\begin{split} 
\tau\biggl( & \zeta M_{q1} + \max\{M_R, M_{\bC}, M_{\bS}\} %\\  &  
+ \frac{2}{\Delta \xi} \max\biggl\{ C_1  + \frac{C_2}{\Delta\xi}, \frac{1}{\rho_X-\hat{X}}\biggl( \zeta \rho_X+C_1\hat{X} + \frac{C_2\hat{X}}{\Delta\xi}\biggr) \biggr\} \biggr) \leq 1.
  \end{split} 
\end{align}

\subsection{Monotonicity and invariant region property} \label{subsec:mon} 

In what follows, we demonstrate that the solution variables $\mathbfcal{U}=(X,\bp^\rmT,\bS^\rmT)^\rmT$ produced by the explicit numerical scheme 
  stays in the set
\begin{equation*}
\Omega:=\big\{\boldsymbol{\mathcal{U}}\in\mathbb{R}^{1+k_{\bC}+k_{\bS}}:
0 \leq X\leq\hat{X},\ \bp\geq 0,\ p^{(1)}+\cdots+p^{(k_{\bC})}=1,\ \bS\geq 0\big\}
\end{equation*}
provided that this property holds for the initial values. Our proofs rely
 on the following lemma, which follows directly from the definition~\eqref{EOjphdef}.

\begin{lemma}\label{lemma:Engquist}
Assume that $0\leq X_j\leq\hat{X}$  for all~$j$. Then the Engquist-Osher flux
$\mathcal{E}_{j+1/2}={\mathcal{E}}_{j+1/2}(X_j,X_{j+1})$ applied on the unimodal function $0\leq f\in C^1$ satisfies
\begin{align*}
-\|f'\|\leq\frac{\partial\mathcal{E}_{j+1/2}}{\partial X_{j+1}}
\leq 0\leq\frac{\partial\mathcal{E}_{j+1/2}}{\partial X_{j}}\leq\|f'\|,\quad \left|\frac{\partial[\Delta\mathcal{E}]_j}{\partial X_j}\right|\le\|f'\|,\quad
\frac{\mathcal{E}_{j+1/2}}{X_j}\leq\|f'\|,\quad
\frac{\mathcal{E}_{j+1/2}}{X_{j+1}}\leq\|f'\|.
\end{align*}
\end{lemma}

\begin{lemma}\label{lemma:Rbounds}
If $\boldsymbol{\mathcal{U}}_j^n\in\Omega$ for all~$j$, then the following estimates hold  for $j=-1, \ldots, N+1$:
\begin{align}
&\kappa^n_j\geq 1 - \zeta\tau M_{q1}, \label{lemm3.2a} \\
&|\tilde{q}_{j+1/2}^n| \leq \zeta(\max\{\qf^n,\qe^n\}+2\qu^n) \leq \zeta M_{q2}, \label{lemm3.2b} \\
&\frac{\partial [\Delta\mathcal{B}]_{j}^{n}}{\partial X_k^n}
\begin{cases}
=-\tilde{q}_{j-1/2}^{n,+}\leq 0& \text{\em if $k=j-1$,}\\
\leq\zeta M_{q2},& \text{\em if $k=j$,}\\
=\tilde{q}_{j+1/2}^{n,-}\leq 0, & \text{\em if $k=j+1$,}\\
=0 &\text{\em otherwise},
\end{cases} \label{lemm3.2c} \\
&\left|\frac{\partial}{\partial X_k} R(\bp_j^nX_j^n/c,\bS_j^n)\right|
\begin{cases}
\leq{M_R}& \text{\em if $k=j$,}\\
=0 &\text{\em if $k\neq j$,}
\end{cases} \label{lemm3.2e} \\
&\max\big\{\Phi_{j+1/2}^{n,+},- \Phi_{j-1/2}^{n,-}\big\} \leq \left(\zeta({M_{q2} + \|f'\|}) + \frac{\zeta^2\|a\|}{\Delta\xi}\right)X_j^n  \label{lemm3.2d} 
\end{align}
\end{lemma}

\begin{proof}
By the definition of the transformation of variables, 
\begin{align*}
-\qf(t)&\leq \bar{z}'(t)\leq \qu(t)+\qe(t),\qquad
-\zeta{(\qu(t)+\qe(t))} \leq\alpha(\xi,t)\leq{\zeta\qf(t)},\qquad
0<\beta(t)\leq\zeta.
\end{align*}
These inclusions directly imply \eqref{lemm3.2a} since 
\begin{align*}
\kappa^n_j&\geq 1-{\tau\zeta\max\{\qu^n+\qe^n,\qf^n\}}\geq 1-{\tau\zeta M_{q1}}.
%\\
% \tilde{q}_{j+1/2}^n &= 
% \begin{cases}
% 0, & \mbox{if }j=-2, -1,\quad \qe^n=0,\\
% -\beta^n\big(\xi_{j+1/2}(\qu^n+\qe^n)+\qe^n\big)\geq - \zeta{\qe^n},& \mbox{if }j=-2, -1,\quad \qe^n>0,\\
% \alpha_{j+1/2}^n + \beta^n\qu^n,&\mbox{if }j = 0, \ldots, N+1.
% \end{cases}
\end{align*}
To prove \eqref{lemm3.2b} 
 we observe that when $\qe^n = 0$, then $\tilde{q}_{-1/2}^n=\tilde{q}_{-3/2}^n=0$, and when $\qe^n>0$ and $j = -1,-2$,  
\begin{align} \label{pelkoa} 
 \tilde{q}_{j+1/2}^n=-\beta^n\big(\xi_{j+1/2}(\qu^n+\qe^n)+\qe^n\big)\geq - {\zeta\qe^n}. 
\end{align}
For $j=0, \ldots, N$, (we recall that if $\qf^n>0$, then $\qe^n=0$ and vice versa)
\begin{equation}
\tilde{q}_{j+1/2}^n = \alpha_{j+1/2}^n + \beta^n\qu^n = 
\beta^n \left({-(\bar{z}')^n(1-\xi_{j+1/2})+\qu^n}\right)
\begin{cases}
\leq\zeta(\qf^n+\qu^n)& \text{if $\qe^n=0$,} \\
\geq-\zeta(\qu^n+\qe^n) &\text{if $\qe^n>0$,} 
\end{cases} \label{pelkob} 
\end{equation}
and for $j=N+1$, 
\begin{equation}
\tilde{q}_{N+3/2}^n = \frac{(\bar{z}')^n\Delta\xi+\qu^n}{B-\bar{z}^n} \leq \zeta({q_{\mathrm{out}}^n \Delta\xi+\qu^n}) \leq \zeta((1+\Delta\xi)\qu^n+\qe^n). \label{pelkoc} 
\end{equation}
From~\eqref{pelkoa} to~\eqref{pelkoc} we now deduce  \eqref{lemm3.2b}. 
Next, computing the difference of the flux 
\begin{align*} 
\mathcal{B}_{j+1/2}^n = \tilde{q}_{j+1/2}^{n,+}X_j^n +  \tilde{q}_{j+1/2}^{n,-}X_{j+1}^n
\end{align*} 
and differentiating this expression with respect to $X_k^n$ we obtain
\begin{align*}
\frac{\partial [\Delta\mathcal{B}]_{j}^{n}}{\partial X_k^n}
&= \frac{\partial}{\partial X_k^n}\bigl(\tilde{q}_{j+1/2}^{n,+}X_j^n + \tilde{q}_{j+1/2}^{n,-}X_{j+1}^n - \tilde{q}_{j-1/2}^{n,+}X_{j-1}^n - \tilde{q}_{j-1/2}^{n,-}X_{j}^n\bigr)\\
&= \begin{cases}
-\tilde{q}_{j-1/2}^{n,+}\leq 0 & \text{if $k=j-1$,} \\
\tilde{q}_{j+1/2}^{n,+}-\tilde{q}_{j-1/2}^{n,-} %\leq \max_{j\geq 0}\{0,\tilde{q}_{j+1/2}^{n}\} - \tilde{q}_{-3/2}^{n}
   \leq \zeta\big({(1+\Delta\xi)\qu^n+\max\{\qe^n,\qf^n\}}\big)\leq \zeta{M_{q2}} &\text{if $k=j$,} \\
\tilde{q}_{j+1/2}^{n,-}\leq 0 &\text{if $k=j+1$,} \\
0  & \text{otherwise}.
\end{cases}
\end{align*}
This proves \eqref{lemm3.2c}. 
For the reaction term, the cases $k\neq j$ are trivial. 
Assuming that $k=j$ and differentiating we obtain 
\begin{equation*}
\left|\frac{\partial}{\partial X_j} R(\bp_j^nX_j^n/c,\bS_j^n)\right| 
= \frac{1}{c}\left|(\bp_j^n)^\rmT\nabla_{\!\bC}R\right| 
\leq\frac{1}{c} \sum_{i=1}^{k_{\bC}}p_j^{(i),n}\biggl|\frac{\partial R}{\partial C_j^{(i),n}}\biggr| 
= M_R\sum_{i=1}^{k_{\bC}}p_j^{(i),n} = M_R,
\end{equation*}
which implies \eqref{lemm3.2e}. Finally, \eqref{lemm3.2d} follows  from 
\begin{align*}
\Phi_{j+1/2}^{n,+} &= \big(\mathcal{B}_{j+1/2}^{n} + \beta^{n} \mathcal{E}_{j+1/2}^{n} - \mathcal{J}_{j+1/2}^n\big)^+
 \leq \tilde{q}_{j+1/2}^{n,+}X_{j}^n + \zeta{\|f'\|}X_j^n + \gamma_{j+1/2}\frac{(\beta^n)^2}{\Delta\xi}\mathcal{D}(X_{j}^n)\\
& \leq \biggl(\zeta({M_{q2} + \|f'\|}) + \zeta^2\frac{\|a\|}{\Delta\xi}\biggr)X_j^n,\\
- \Phi_{j-1/2}^{n,-} & = -\big(\mathcal{B}_{j-1/2}^{n} + \beta^{n} \mathcal{E}_{j-1/2}^{n} - \mathcal{J}_{j-1/2}^n\big)^-
\leq \tilde{q}_{j-1/2}^{n,+}X_{j}^n + \zeta{\|f'\|}X_j^n + \gamma_{j+1/2}\frac{(\beta^n)^2}{\Delta\xi}\mathcal{D}(X_{j}^n)\\
& \leq \biggl(\zeta({M_{q2} + \|f'\|}) + \frac{\zeta^2\|a\|}{\Delta\xi}\biggr)X_j^n.
\end{align*} 
\end{proof}

\begin{lemma}\label{lemma:monotoneX}
If $\smash{\boldsymbol{\mathcal{U}}_j^n\in\Omega}$ for
all~$j$ and \eqref{eq:CFL} is in effect, then $0\le X_j^{n+1}\le 1$ for all~$j$.
\end{lemma}

\begin{proof}
We write the update formula \eqref{eq:scheme_X} for $j=-1,\ldots,N+1$ as
\begin{equation*}
X_j^{n+1} = \mathcal{H}_{X}^n \bigl(X_{j-1}^{n},X_j^{n},X_{j+1}^{n}\bigr)
\end{equation*}
and we shall show that $\mathcal{H}_{X}^n$ is a monotone function in each of its variables.
We recall that
\begin{equation*}
\lambda[\Delta\mathcal{J}]_j^n = \mu(\beta^n)^2\bigl(\gamma_{j+1/2}\big(\mathcal{D}(X_{j+1}^{n+1}) - \mathcal{D}(X_{j}^{n+1})\bigr) - \gamma_{j-1/2}\big(\mathcal{D}(X_{j}^{n+1}) - \mathcal{D}(X_{j-1}^{n+1})\big)\bigr),
\end{equation*}
$a=\mathcal{D}'$, and we differentiate to obtain, by means of  \eqref{eq:CFL}  and Lemmas~\ref{lemma:Engquist} and \ref{lemma:Rbounds},
\begin{align*}
\frac{\partial X_j^{n+1}}{\partial X_{j-1}^{n}} 
&= - \lambda\frac{\partial[\Delta\mathcal{B}]_j^n}{\partial X_{j-1}^{n}} - \lambda\beta^n\frac{\partial[\Delta\mathcal{E}]_{j}^n}{\partial X_{j-1}^{n}} +  \lambda\frac{\partial[\Delta\mathcal{J}]_j^n}{\partial X_{j-1}^{n}}\\
&= \lambda\tilde{q}_{j-1/2}^{n,+} + \lambda\beta^n\frac{\partial\mathcal{E}_{j-1/2}}{\partial X_{j-1}^{n}} + \mu(\beta^n)^2\gamma_{j-1/2} a(X_{j-1}^n) \geq 0,\\
\frac{\partial X_j^{n+1}}{\partial X_{j}^{n}} 
&=\kappa^n - \lambda\frac{\partial[\Delta\mathcal{B}]_j^n}{\partial X_{j}^{n}} - \lambda\beta^n\frac{\partial[\Delta\mathcal{E}]_{j}^n}{\partial X_{j}^{n}} + \lambda\frac{\partial[\Delta\mathcal{J}]_j^n}{\partial X_{j}^{n}} + \frac{\partial}{\partial X_j} R(\bp_j^nX_j^n/c,\bS_j^n)\\
&\geq 1-{\tau\zeta M_{q1}} - {\lambda\zeta M_{q2}} - \zeta{\lambda}\|f'\| - {\mu\zeta^22\|a\|} -\tau M_R \geq 0,\\
\frac{\partial X_j^{n+1}}{\partial X_{j+1}^{n}}
&= - \lambda\frac{\partial[\Delta\mathcal{B}]_j^n}{\partial X_{j+1}^{n}} - \lambda\beta^n\frac{\partial[\Delta\mathcal{E}]_{j}^n}{\partial X_{j+1}^{n}} 
= -\lambda\tilde{q}_{j+1/2}^{n,-} - \lambda\beta^n\frac{\partial\mathcal{E}_{j+1/2}}{\partial X_{j+1}^{n}} \geq 0.
\end{align*}
The proved monotonicity of $\mathcal{H}_{X}^n$ and the assumptions
\eqref{eq:techRC} imply that, for $j\neq 0$,
\begin{align*}
0& =\mathcal{H}_{X}^n(0,0,0)\leq X_j^{n+1} 
= \mathcal{H}_{X}^n(X_{j-1}^{n},X_j^{n},X_{j+1}^{n})
\leq\mathcal{H}_{X}^n(\hat{X},\hat{X},\hat{X}) \\
 & =  \kappa_j^n\hat{X}-\lambda(\tilde{q}_{j+1/2}^n-\tilde{q}_{j-1/2}^n)\hat{X} \\
& = 
\begin{cases}
\hat{X}\bigl(1-\tau\beta^n q_{\mathrm{out}}^n  -\lambda\beta^n (- (\xi_{-1/2} q_{\mathrm{out}}^n +\qe^n ) + (\xi_{-3/2} q_{\mathrm{out}}^n +\qe^n))\bigr) & \\
 = \hat{X}\bigl(1-\tau\beta^n q_{\mathrm{out}}^n  -\lambda\beta^n (\frac{\Delta\xi}{2} q_{\mathrm{out}}^n -\frac{3\Delta\xi}{2} q_{\mathrm{out}}^n 
  )\bigr) = \hat{X} & \text{if $j=-1$ and $\qe^n>0$,} \\
\hat{X} & \text{if $j=-1$ and $\qe^n=0$,} \\
\hat{X}\bigl(1+\tau\beta^n(\bar{z}')^n & \\ \quad - \lambda\beta^n (-(\bar{z}')^n(1-\xi_{j+1/2})+(\bar{z}')^n(1-\xi_{j-1/2}) )\bigr) =\hat{X} & 
 \text{if $j\geq 1$.} 
\end{cases} 
\end{align*}
For $j=0$ we have, since $\alpha(\xi_{1/2},t)= -\bar{z}'(t)(1-\Delta\xi/2)\beta(t)$, and assuming $X_\mathrm{f}^n\leq\hat{X}$,
\begin{align*}
0 & \leq\lambda\beta^n X_\mathrm{f}^n\qf^n = \mathcal{H}_{X}^n(0,0,0) \leq X_0^{n+1}=\mathcal{H}_{X}^n(X_{-1}^{n},X_0^{n},X_{1}^{n})
\leq\mathcal{H}_{X}^n(\hat{X},\hat{X},\hat{X})\\
& =\begin{cases}
\kappa_0^n\hat{X} - \lambda\big(\alpha_{1/2}^n+\beta^n\qu^n+\beta^n(-\frac{\Delta\xi}{2} q_{\mathrm{out}}^n +\qe^n)\big)\hat{X} &\\
= \hat{X}\bigl(1-\tau\beta^n q_{\mathrm{out}}^n  - \lambda (-q_{\mathrm{out}}^n (1-\frac{\Delta\xi}{2})\beta^n  +
 \beta^n\qu^n+\beta^n(-\frac{\Delta\xi}{2} q_{\mathrm{out}}^n +\qe^n) )\bigr) & \\  = \hat{X} (1-\tau\beta^n q_{\mathrm{out}}^n )\leq\hat{X} 
  & \text{if $\qe^n>0$,} \\
\kappa_0^n\hat{X} - \lambda\big(\alpha_{1/2}^n+\beta^n\qu^n)\big)\hat{X} + \lambda\beta^n X_\mathrm{f}^n\qf^n &\\
= (1-\frac{\tau}{2}\beta^n q_{\mathrm{out}}^n )\hat{X} - \lambda\big(-(\qu^n-\qf^n)(1-\frac{\Delta\xi}{2})\beta^n + \beta^n\qu^n)\big)\hat{X} + \lambda\beta^n X_\mathrm{f}^n\qf^n &\\
= \hat{X}- \lambda\big(-(\qu^n-\qf^n)\beta^n + \beta^n\qu^n)\big)\hat{X} + \lambda\beta^n X_\mathrm{f}^n\qf^n & \\ = \hat{X}- \lambda\beta^n(\hat{X}-X_\mathrm{f}^n)\qf^n \leq \hat{X} & \text{if $\qe^n=0$.} 
\end{cases}
\end{align*} 
\end{proof}

\begin{lemma}\label{lemma:p_pos}
If $\smash{\boldsymbol{\mathcal{U}}_j^n\in\Omega}$ for all~$j$ and \eqref{eq:CFL} holds, then 
\begin{align}  \label{p_posineq} 
p_j^{(k),n+1}\geq  0 \quad \text{\em  for all~$k=1, \dots, 
 k_{\boldsymbol{C}}$ and all~$j$}.
 \end{align} 
\end{lemma}

\begin{proof}  
If $\smash{X_j^{n+1}=0}$, we define 
\begin{align} \label{pkdef} 
 p_j^{(k),n+1}\coloneqq p_j^{(k),n}\in[0,1] \quad \text{for all $k=1, \dots, 
 k_{\boldsymbol{C}}$.} 
 \end{align} 
If $\smash{X_j^{n+1}>0}$, we have for each $k \in \{ 1, \dots, 
 k_{\boldsymbol{C}} \}$
\begin{align*}
p_j^{(k),n+1}X_j^{n+1} & = \kappa_j^n p_j^{(k),n}X^{n}_j - \lambda\bigl(\Phi_{j+1/2}^{n,+}p_j^{(k),n} + \Phi_{j+1/2}^{n,-}p_{j+1}^{(k),n} - \Phi_{j-1/2}^{n,+}p_{j-1}^{(k),n} - \Phi_{j-1/2}^{n,-}p_{j}^{(k),n}\bigr)\\
 & \quad + \lambda\delta_{j,0}\beta^n q_{\rm f}^n p_{\rm f}^{(k),n} X_{\rm f}^n + \tau\gamma_j c R_j^{(k),n}(\bp^{n}_j X^{n}_j/c,\bS_j^n)\\
& \geq \left(1 - \zeta{\tau M_{q1}}\right)p_j^{(k),n}X^{n}_j - 2\lambda\left(\zeta({M_{q2} + \|f'\|}) 
+ \frac{\zeta^2\|a\|}{\Delta\xi}\right)p_j^{(k),n}X_j^n\\
& \quad + \tau c\sum_{l\in I_{\bC,k}^-}\sigma_{\bC}^{(k,l)}\bar{r}^{(l)} \bigl(\bp_j^{n}X_j^n/c, \bS_j^{n} \bigr)p_j^{(k),n}X_j^n/c\\
& \geq \left(1 - \zeta({\tau M_{q1}}) - \frac{2\tau}{\Delta\xi}\left(\zeta({M_{q2} + \|f'\|})
+ \frac{\zeta^2\|a\|}{\Delta\xi}\right) + \tau M_{\bC}\right)p_j^{(k),n}X_j^n \geq 0.
\end{align*}
This implies \eqref{p_posineq}. 
\end{proof}

\begin{lemma}\label{lemma:sum_p}
If $\smash{\boldsymbol{\mathcal{U}}_j^n\in\Omega}$ for all~$j$ and \eqref{eq:CFL} holds, then
\begin{align}  \label{p_sumineq} 
p_j^{(1),n+1} +  \cdots + p_j^{(k_{\boldsymbol{C}}),n+1} =  1 \quad \text{\em  for all~$j$}.
 \end{align} 
\end{lemma}

\begin{proof}
If $X_j^{n+1}=0$, then by \eqref{pkdef}, 
\begin{align*}  
 p_j^{(1),n+1}+\cdots+p_j^{(k_{\bC}),n+1}=p_j^{(1),n}+\cdots+p_j^{(k_{\bC}),n} =1, 
 \end{align*} 
  so let us 
assume that $X_j^{n+1}>0$.
We sum up all  equations in~\eqref{eq:scheme_p} and utilize  that 
\begin{align*} 
p_j^{(1),n}+\cdots+p_j^{(k_{\bC},n)}=1
\end{align*} 
along with 
\begin{align*}
\sum_{k=1}^{k_{\bC}}[\Delta\Phi_{\bp}^{(k)}]_j^n & = \sum_{k=1}^{k_{\bC}} \big(\Upw(\Phi_{\bp,j+1/2}^{(k),n},p_j^{(k),n},p_{j+1}^{(k),n}) -\Upw(\Phi_{\bp,j-1/2}^{(k),n},p_j^{(k),n},p_{j+1}^{(k),n}) \big)\\
&= \sum_{k=1}^{k_{\bC}}\big(\Phi_{j+1/2}^{n,+}p_j^{(k),n} + \Phi_{j+1/2}^{n,-}p_{j+1}^{(k),n} - \Phi_{j-1/2}^{n,+}p_{j-1}^{(k),n} - \Phi_{j-1/2}^{n,-}p_{j}^{(k),n}\big)\\
& = \Phi_{j+1/2}^{n,+} + \Phi_{j+1/2}^{n,-} - \Phi_{j-1/2}^{n,+} - \Phi_{j-1/2}^{n,-} =[\Delta\Phi]_{j}^n.
\end{align*}
Then the sum of the equations in~\eqref{eq:scheme_p} is
\begin{align} \label{eq3.7} 
X^{n+1}_j\sum_{k=1}^{k_{\bC}}p_j^{(k),n+1}  = \kappa_j^n X^{n}_j - \lambda [\Delta{\Phi}]_j^n + \lambda\delta_{j,0}\beta^nq_{\rm f}^nX_{\rm f}^n + \tau\gamma_j R(\bp_j^nX^{n}_j/c,\bS_j^n).
\end{align}
The right-hand side is identical to that   of~\eqref{eq:scheme_X}.
Hence, subtracting \eqref{eq3.7} from~\eqref{eq:scheme_X}  we get 
\begin{align*} 
 X^{n+1}_j \bigl( 1- (p_j^{(1),n+1}+ \cdots + p_j^{(k_{\boldsymbol{C}}),n+1})\bigr) =0, 
 \end{align*} 
which proves  the desired result \eqref{p_sumineq}. 
\end{proof}

\begin{lemma}\label{lemma:S_pos}
If $\smash{\boldsymbol{\mathcal{U}}_j^n\in\Omega}$ for all~$j$ and \eqref{eq:CFL} holds, 
then 
\begin{align*}  
S_j^{(k),n+1}\geq  0 \quad \text{\em  for all~$k=1, \dots, 
 k_{\boldsymbol{S}}$ and all~$j$}.
 \end{align*} 
\end{lemma}

\begin{proof}
From  the update formula for a component $\smash{S_j^{(k),n}}$ of $\smash{\bS_j^n}$ we get 
\begin{align*}
 S_j^{(k),n+1}  & = \kappa_j^n S_j^{(k),n} 
- \lambda\bigg(\frac{S_j^{(k),n}}{\rho_X-X^{n}_j}(\rho_X\tilde{q}_{j+1/2}^n-\Phi_{j+1/2}^{n})^+ + \frac{S_{j+1}^{(k),n}}{\rho_X-X^{n}_{j+1}}(\rho_X\tilde{q}_{j+1/2}^n-\Phi_{j+1/2}^{n})^-\\
& \quad  - \frac{S_{j-1}^{(k),n}}{\rho_X-X^{n}_{j-1}}(\rho_X\tilde{q}_{j-1/2}^n-\Phi_{j-1/2}^{n})^+ 
- \frac{S_{j}^{(k),n}}{\rho_X-X^{n}_{j}}(\rho_X\tilde{q}_{j-1/2}^n-\Phi_{j-1/2}^{n})^-\bigg)\\
 & \quad + \lambda\delta_{j,0}\beta^n q_{\rm f}^n S_{\rm f}^{(k),n} + \tau\gamma_j R_{\bS,j}^{(k),n}(\bp^{n}_j X^{n}_j/c,\bS_j^n)\\
& \geq \left(1 - {\tau\zeta M_{q1}}\right)S_j^{(k),n} 
-  \frac{\lambda S_j^{(k),n}}{\rho_X-X^{n}_j} \bigl((\rho_X\tilde{q}_{j+1/2}^n-\Phi_{j+1/2}^{n})^+ 
- (\rho_X\tilde{q}_{j-1/2}^n-\Phi_{j-1/2}^{n})^-\bigr)\\
& \quad + \tau\sum_{l\in I_{\bS,k}^-}\sigma_{\bS}^{(k,l)}\bar{r}^{(l)} \bigl(\bp_j^{n}X_j^n/c, \bS_j^{n} \bigr)S_j^{(k),n}\\
& \geq \biggl(1 - {\tau\zeta M_{q1}}  
%\\ & \qquad
- \frac{2 \tau \zeta }{\Delta\xi(\rho_X-\hat{X})}\biggl((\rho_X+\hat{X}) M_{q2}  +  {\|f'\|\hat{X}}
+ \frac{\zeta \|a\|\hat{X}}{\Delta\xi} \biggr)  + \tau M_C\biggr)S_j^{(k),n}  \geq 0.
\end{align*} 
\end{proof}
 
\section{A semi-implicit scheme}  \label{scheme:si} 

\subsection{Semi-implicit scheme for the update of~$X$} \label{subsec:siX} 

To obtain a semi-implicit scheme, we write out several terms in~\eqref{eq:scheme_X}--\eqref{eq:scheme_S} and evaluate those containing the coefficient~$\mu=\tau/\Delta\xi^2$ at time~$t_{n+1}$.
Then~\eqref{eq:scheme_X} becomes
\begin{align}\label{eq:SI_scheme_X} \begin{split} 
X^{n+1}_j & = \kappa_j^n X^{n}_j - \lambda [\Delta \mathcal{F}]_j^n \\
& \quad + (\beta^n)^2 \mu \bigl(\gamma_{j+1/2} (\mathcal{D}(X_{j+1}^{n+1}) - \mathcal{D}(X_{j}^{n+1}) ) - \gamma_{j-1/2} (\mathcal{D}(X_{j}^{n+1}) - \mathcal{D}(X_{j-1}^{n+1}))\bigr)\\
 & \quad + \lambda\delta_{j,0}\beta^nq_{\rm f}^nX_{\rm f}^n + \tau\gamma_j{R}(\bp_j^nX^{n}_j/c,\bS_j^n) 
 \end{split} 
\end{align}
For~$j=-1$ and $j= N+1$, many terms are zero and this formula is in fact explicit; see \eqref{eq:update_-1} and \eqref{eq:update_N+1}.
For $j=0, \ldots, N$, one has to solve a system of $(N+1) \times (N+1)$ nonlinear equations. 
For the further analysis, it is 
is useful  to rewrite \eqref{eq:scheme_X} as a two-step implicit-explicit scheme: 
\begin{enumerate} 
\item Given $\smash{X_j^n}$ for $j=-1, \dots, N+1$, calculate $\smash{\tilde{X}_j^{n+1}}$ from 
\begin{align}\label{eq:SI_scheme_Xa} \qquad  \quad
\tilde{X}^{n+1}_j & = 
\kappa_j^n X^{n}_j - \lambda [\Delta \mathcal{F}]_j^n + \lambda\delta_{j,0}\beta^nq_{\rm f}^nX_{\rm f}^n 
+ \tau\gamma_j{R}(\bp_j^nX^{n}_j/c,\bS_j^n), 
\quad j=-1, \dots, N+1.  
\end{align}
\item Let 
\begin{align} 
X_{-1}^{n+1} = \tilde{X}_{-1}^{n +1},   \quad X_{N+1}^{n+1} = \tilde{X}_{N+1}^{n +1}, 
\label{tstep} 
\end{align} 
 and compute $\smash{\boldsymbol{X}^{n+1} = (X_0^{n+1}, \dots, X_{N}^{n+1})^{\mathrm{T}}}$   
 by solving the nonlinear system of equations 
 \begin{align} \label{syst} \begin{split} 
\boldsymbol{X}^{n+1} 
+ (\beta^n)^2 \mu 
\boldsymbol{T} 
   \begin{pmatrix} \mathcal{D} ( X_0^{n+1} )\\ 
\mathcal{D} (X_1^{n+1}) \\
 \vdots  \\ 
 \mathcal{D} (X_N^{n+1} ) \end{pmatrix}
  =  \boldsymbol{\tilde{X}}^{n+1},  \quad \boldsymbol{T} \coloneqq \begin{bmatrix} 
 1  & -1 & & & \\ 
  -1 & 2 & -1 & & \\ 
  & \ddots & \ddots & \ddots & \\
 & & -1 & 2 & -1 \\
 & & & -1 & 1 
\end{bmatrix} 
\end{split} 
\end{align}  
where $\smash{\boldsymbol{\tilde{X}}^{n+1}\coloneqq(\tilde{X}_0^{n+1}, \ldots, \tilde{X}_N^{n+1})^\rmT}$. 
\end{enumerate}  

In what follows, we assume that  the CFL condition for the semi-implicit scheme 
\begin{align} \label{CFLSI} \tag{CFL-SI}   \begin{split} 
\tau\biggl( & \zeta{M_{q1}} + \max\{M_R, M_{\bC}, M_{\bS}\} + \frac{2}{\Delta \xi} \max\biggl\{ C_1, \frac{{\zeta\rho_X}+C_1\hat{X}}{\rho_X-\hat{X}} \biggr\} \biggr) \leq 1  \end{split} 
\end{align} 
    is in effect. Notice 
     that \eqref{CFLSI} arises from \eqref{eq:CFL} from the fact that $a \equiv 0$ and hence $C_2=0$, and therefore 
      stipulates a bound on $\tau / \Delta \xi$, but not on $\tau / \Delta \xi^2$. 
      
     To ensure that the system \eqref{syst} has a unique solution at all,  we follow 
 a strategy similar to that of  \cite{Burger&C&S2006}:  we first  assume that \eqref{syst}
   has a solution and show that  the scheme 
  is monotone and satisfies an invariant region property. We then invoke a topological degree argument to show that \eqref{syst} 
   indeed has a solution, and  show that it depends Lipschitz continuously on the solution at the previous 
    time step. As a consequence, the whole scheme \eqref{eq:SI_scheme_Xa}--\eqref{syst} is well defined.

\begin{lemma} \label{lem:querpu} Assume that $\tau$ and $\Delta \xi$ satisfy \eqref{CFLSI}. Then the 
 scheme \eqref{eq:SI_scheme_Xa}--\eqref{syst} is monotone, i.e.\ there exist  functions $\smash{\mathcal{K}_j^n}$, 
  $j=-1, \dots, N+1$, such that 
  \begin{align*} 
   X_j^{n+1} = \mathcal{K}_{j}^n \bigl( X_{-1}^n, X_0^n, X_1^n, \dots, X_{N}^n, X_{N+1}^n , t^n\bigr), \quad j= -1, \dots, N+1, 
\end{align*} 
that are monotone in each $X$-argument. 
\end{lemma} 

\begin{proof} By repeating the monotonicity part of the proof of Lemma~\ref{lemma:monotoneX} we see that under 
 the condition \eqref{CFLSI}, the scheme \eqref{eq:SI_scheme_Xa} is monotone, i.e.,  i.e., $\partial \tilde{X}^{n+1}_j / \partial X^n_k \geq 0$ for 
 all $-1 \leq j,k \leq N+1$, so by \eqref{tstep} the statement immediately holds for $j=-1$ and $j=N+1$. 
In what follows, we define for a vector $\boldsymbol{X}=(X_0, \dots, X_N)^{\mathrm{T}}$ the Jacobian matrix of the left-hand side of~\eqref{syst}, namely 
\begin{align} \label{jacob-def} 
\boldsymbol{\mathcal{J}}( \boldsymbol{X}) \coloneqq 
 \boldsymbol{I}_{N+1} + (\beta^n)^2 \mu  \boldsymbol{T} 
  \diag \bigl( a(X_0), \dots, a(X_N) \bigr). 
\end{align} 
We wish to show that 
\begin{align*} 
 \frac{\partial X_j^{n+1}}{\partial X_k^n} \geq 0 \quad \text{for all $1 \leq k,j \leq n$}. 
 \end{align*} 
 To this end we introduce for $k=0, \dots, N$  the vectors
 \begin{align*} 
 \frac{\partial \boldsymbol{X}^{n+1}}{\partial X_k^n} := 
  \biggl( \frac{\partial X_0^{n+1}}{\partial X_k^n} , \dots, \frac{\partial X_N^{n+1}}{\partial X_k^n}\biggr)^{\mathrm{T}}, 
   \quad  \frac{\partial \boldsymbol{\tilde{X}}^{n}}{\partial X_k^n} := 
  \biggl( \frac{\partial \tilde{X}_0^{n}}{\partial X_k^n} , \dots, \frac{\partial \tilde{X}_N^{n}}{\partial X_k^n}\biggr)^{\mathrm{T}}
 \end{align*}  
Assume that for given~$\boldsymbol{X}^n$, the vector $\boldsymbol{X}^{n+1}$ is a solution to \eqref{eq:SI_scheme_X}. 
 Then  
 \begin{align*} 
 \boldsymbol{\mathcal{J}}  ( \boldsymbol{X}^{n+1} ) \frac{\partial \boldsymbol{X}^{n+1}} {\partial X_k^n} =  
  \frac{\partial \boldsymbol{\tilde{X}}^{n}}{\partial X_k^n}, 
  \quad k=0, \dots, N. 
 \end{align*} 
We already know that $\smash{\partial \boldsymbol{\tilde{X}}^{n} / \partial X_k^n \geq \boldsymbol{0}}$. On the other hand, 
for any~$\boldsymbol{X}$, the matrix $\smash{(\boldsymbol{\mathcal{J}}( \boldsymbol{X}))^{\mathrm{T}}}$ is a strictly diagonally 
dominant L-matrix and therefore an M-matrix; in particular,  $\smash{(\boldsymbol{\mathcal{J}}( \boldsymbol{X}))^{\mathrm{T}}}$ has 
a non-negative inverse, and therefore also $\smash{(\boldsymbol{\mathcal{J}} ( \boldsymbol{X}))^{-1}}$ is non-negative, hence 
\begin{align*} 
 \frac{\partial \boldsymbol{X}^{n+1}} {\partial X_k^n} =  \bigl(  \boldsymbol{\mathcal{J}}  ( \boldsymbol{X}^{n+1} ) \bigr)^{-1} 
  \frac{\partial \boldsymbol{\tilde{X}}^{n}}{\partial X_k^n} \geq \boldsymbol{0}, 
  \quad k=0, \dots, N. 
 \end{align*} 
\end{proof} 

\begin{lemma}\label{lemma:monotoneXSI}
If $\smash{\boldsymbol{\mathcal{U}}_j^n\in\Omega}$ for
all~$j$ and \eqref{CFLSI} is in effect, then 
\begin{align}  \label{XboundSI} 
0\le X_j^{n+1}\le \hat{X} \quad \text{\em for all~$j=-1, \dots, N+1$.} 
\end{align} 
\end{lemma}

\begin{proof} Repeating the second part of the proof of Lemma~\ref{lemma:monotoneX} under the assumption $a \equiv 0$ 
 we see that under 
 the condition \eqref{CFLSI},
 \begin{align} \label{tdxbounds} 
  0 \leq \tilde{X}_j^{n+1} \leq \hat{X} \quad \text{for  all~$j=-1, \dots, N+1$.}
  \end{align} 
  This directly proves \eqref{XboundSI} for $j=-1$ and $j=N+1$. Furthermore, we define 
  \begin{align*} 
    \varpi_{j+1/2}^n \coloneqq
    \begin{cases} 
     (\mathcal{D} (X_{j+1}^{n}) - \mathcal{D} (X_j^{n}))/(X_{j+1}^{n} - X_{j}^{n}) 
      & \text{if $X_{j+1}^{n} \neq X_{j}^{n}$,}  \\
      0 & \text{otherwise,} 
  \end{cases} \quad j=0, \dots, N-1, 
  \end{align*} 
  and write the nonlinear scheme \eqref{syst} as $\smash{\boldsymbol{M} (\boldsymbol{X}^{n+1}) \boldsymbol{X}^{n+1} 
   = \boldsymbol{\tilde{X}}^{n+1}}$, where the entries $(m_{ij})_{0 \leq i,j \leq N}$ of the tridiagonal matrix 
   $\boldsymbol{M} =\boldsymbol{M} (\boldsymbol{X}^{n+1})$ are given by 
  \begin{align*} 
  m_{j,j-1} & = - (\beta^n)^2 \mu \varpi_{j-1/2}^{n+1}, \quad j=1, \dots, N, \\
  m_{jj} & = 1 + (\beta^n)^2 \mu \bigl( \varpi_{j-1/2}^{n+1} + \varpi_{j+1/2}^{n+1} \bigr), \quad j=1, \dots, N-1, \\
  m_{j,j+1} & = - (\beta^n)^2 \mu \varpi_{j+1/2}^{n+1}, \quad j=0, \dots, N-1, \\
  m_{00} & = 1 + (\beta^n)^2 \mu \varpi_{1/2}^{n+1}, \quad  m_{NN}  = 1 + (\beta^n)^2 \mu \varpi_{N-1/2}^{n+1}.  
  \end{align*} 
  Since $\smash{\varpi_{j+1/2}^{n+1} \geq 0}$ for all~$j$,  
  $\boldsymbol{M}$ is  a strictly diagonally dominant L-matrix and therefore an M-matrix, that is 
$\boldsymbol{M}^{-1}$ exists and $\boldsymbol{M}^{-1} \geq \boldsymbol{0}$, i.e.,  
 if we write $\boldsymbol{M}^{-1} = (\bar{m}_{jk})_{0 \leq j,k \leq N}$, then $\bar{m}_{jk} \geq 0$. Since 
 \begin{align} \label{xjboundprel} 
 X_{j}^{n+1} = \bar{m}_{j,0} \tilde{X}_0^{n+1} + \bar{m}_{j,1} \tilde{X}_1^{n+1} +  \dots +  \bar{m}_{j,N} \tilde{X}_N^{n+1}, 
 \end{align} 
 this property implies that $\smash{ X_{j}^{n+1} \geq 0}$. 
 On the other hand, since
 \begin{align*} 
  m_{jj} - \sum_{k =0 \atop k \neq j}^{N }|  m_{jk} | \geq 1  \quad \text{and} \quad 
 \sum_{k =0}^{N } m_{jk} = 1  \quad \text{for all $j=0, \dots, N$,} 
  \end{align*} 
we have $\boldsymbol{M}\boldsymbol{1}=\boldsymbol{1}$ with $\boldsymbol{1}:=(1, \dots, 1)^{\mathrm{T}}$, hence $\boldsymbol{M}^{-1}\boldsymbol{1}=\boldsymbol{1}$, 
that is $\bar{m}_{j,0}  + \bar{m}_{j,1} + \dots + \bar{m}_{j,N} = 1$.
 In view of the  upper bound in~\eqref{tdxbounds} we then deduce from \eqref{xjboundprel} that 
    $\smash{X_j^{n+1} \leq \hat{X}}$ for $j=0, \dots, N$.  
\end{proof} 

In what follows, we define $\boldsymbol{\mathcal{X}} \coloneqq (X_{-1}, X_0,  X_1, \dots , X_N, X_{N+1})^{\mathrm{T}} \in \mathbb{R}^{N+3}$.   The following lemma and its proof closely follow \cite[Lemma~3.3, part~(a)]{Burger&C&S2006}. 

\begin{lemma} Assume that the condition \eqref{CFLSI} is in effect and that $\boldsymbol{\mathcal{X}}^n \in [0, \hat{X}]^{N+3}$. Then the scheme \eqref{eq:SI_scheme_X}, or  
equivalently, \eqref{eq:SI_scheme_Xa}--\eqref{syst}, admits a solution $\boldsymbol{\mathcal{X}}^{n+1} \in [0, \hat{X}]^{N+3}$. 
\end{lemma} 

\begin{proof} The existence of~$\boldsymbol{\mathcal{X}}^{n+1}$ follows by adopting an  argument used in\cite{eymard00} 
 to prove the existence of a solution  of implicit schemes for hyperbolic equations based on 
  topological degree theory \cite{deim85}. To this end, let us write the scheme in the form 
  \begin{align} \label{bcs063.10} 
   \boldsymbol{\mathcal{X}}^{n+1} - \boldsymbol{\mathcal{E}} (  \boldsymbol{\mathcal{X}}^{n+1}, t^n  ) = 
    \boldsymbol{\tilde{\mathcal{X}}}^{n+1} (  \boldsymbol{\mathcal{X}}^{n}, t^n ).  
    \end{align} 
    Here $\smash{ \boldsymbol{\tilde{\mathcal{X}}}^{n+1} ( \cdot , t^n ): \mathbb{R}^{N+3} \to \mathbb{R}^{N+3}}$ 
      is a continuous function  defined by  \eqref{eq:SI_scheme_Xa}, and 
      $\boldsymbol{\mathcal{E}}(\cdot, t^n) : \smash{\mathbb{R}^{N+3}} \to \smash{\mathbb{R}^{N+3}}$ is another continuous 
        function defined in an obvious way by \eqref{tstep}, \eqref{syst}.  By Lemma~\ref{lemma:monotoneXSI}, 
         if $ \boldsymbol{\mathcal{X}}^{n+1}$ satisfies \eqref{bcs063.10}, then 
         $\smash{\boldsymbol{\mathcal{X}}^{n+1} \in [0, \hat{X}]^{N+3}}$. On the other hand we know (and have used that) 
          if $a \equiv 0$ and hence $\mathcal{D} \equiv 0$, then the explicit scheme for the hyperbolic case 
            \begin{align} \label{bcs063.11} 
   \boldsymbol{\mathcal{X}}^{n+1}  = 
    \boldsymbol{\tilde{\mathcal{X}}}^{n+1} (  \boldsymbol{\mathcal{X}}^{n}, t^n ),   
    \end{align} 
     which corresponds just to the scheme~\eqref{eq:SI_scheme_Xa}, satisfies the same bound, i.e., 
      $\smash{\boldsymbol{\tilde{\mathcal{X}}}^{n+1} (  \boldsymbol{\mathcal{X}}^{n}, t^n )  \in [0, \hat{X}]^{N+3}}$. 
      Consequently, if $B_R\subset \mathbb{R}^{N+3}$ is a ball with center~$\boldsymbol{0}$ and sufficiently large radius~$R$, then 
       \eqref{bcs063.10} has no solution on the boundary of~$B_R$, and one can define the topological degree of the mapping 
        $\mathrm{Id} - \boldsymbol{\mathcal{E}}$ associated with the set~$B_R$ and the point~$\smash{\boldsymbol{\tilde{\mathcal{X}}}^{n+1} (  \boldsymbol{\mathcal{X}}^{n}, t^n )}$, that is, $\smash{\deg( \mathrm{Id} - \boldsymbol{\mathcal{E}}, B_R, \boldsymbol{\tilde{\mathcal{X}}}^{n+1} (  \boldsymbol{\mathcal{X}}^{n}, t^n ))}$.  Furthermore, if $\alpha \in [0,1]$, then the same argument allows us to 
         define  $\smash{\deg( \mathrm{Id} - \alpha \boldsymbol{\mathcal{E}}, B_R, \boldsymbol{\tilde{\mathcal{X}}}^{n+1} (  \boldsymbol{\mathcal{X}}^{n}, t^n ))}$. The property of invariance of degree under continuous transformations then asserts that 
         the latter quantity 
                  does not depend on~$\alpha$, hence 
        \begin{align*} 
           \deg \bigl( \mathrm{Id} - \boldsymbol{\mathcal{E}}, B_R, \boldsymbol{\tilde{\mathcal{X}}}^{n+1} (  \boldsymbol{\mathcal{X}}^{n}, t^n ) \bigr)
            =  \deg \bigl( \mathrm{Id} , B_R, \boldsymbol{\tilde{\mathcal{X}}}^{n+1} (  \boldsymbol{\mathcal{X}}^{n}, t^n ) \bigr) =1, 
             \end{align*} 
              where the equality for $\alpha =0$ holds since we can solve the scheme  \eqref{bcs063.11} in a unique way. 
               This proves that \eqref{bcs063.10} has a solution in~$B_R$, where we already have proved that the solution 
                belongs to~$\smash{[0, \hat{X}]^{N+3}}$. 
               \end{proof} 

\begin{lemma}  Assume that the condition \eqref{CFLSI} is in effect and that $\boldsymbol{\mathcal{X}}^n , \boldsymbol{\mathcal{Y}}^n \in [0, \hat{X}]^{N+3}$, and that~$\boldsymbol{\mathcal{X}}^{n+1}$ and~$\boldsymbol{\mathcal{Y}}^{n+1}$ are both computed by 
 scheme \eqref{eq:SI_scheme_X}, or  
equivalently, \eqref{eq:SI_scheme_Xa}--\eqref{syst}. Then there exists a constant~$C>0$ such  that 
\begin{align}  \label{lemm4.4ineq} 
 \| \boldsymbol{\mathcal{X}}^{n+1} - \boldsymbol{\mathcal{Y}}^{n+1}  \|_1 \leq 
 (1+ C\Delta t)  \| \boldsymbol{\mathcal{X}}^{n} - \boldsymbol{\mathcal{Y}}^{n} \|_1. 
 \end{align}  
 This means that the solution of  \eqref{eq:SI_scheme_Xa}--\eqref{syst} depends  Lipschitz continuously on~$\boldsymbol{\mathcal{X}}^{n}$, and 
  in  particular, setting $ \boldsymbol{\mathcal{X}}^{n}  =  \boldsymbol{\mathcal{Y}}^{n}$, we obtain uniqueness.    
\end{lemma} 

\begin{proof} 
 We define $\smash{\vartheta_j^n := Y_j^n- X_j^n}$ for $j=-1, \dots, N+1$ and the quantities 
   \begin{align*} 
    \theta_j^{n+1} := \begin{cases} 
    \mu (\beta^n)^2 (\mathcal{D} (Y_j^{n+1}) - \mathcal{D} (X_j^{n+1}))/\vartheta_j^{n+1} \geq 0 
     & \text{if $ \vartheta_j^{n+1} \neq 0$,} \\
     0 & \text{otherwise},  
    \end{cases}  \quad j=-1, \dots, N+1.   
     \end{align*} 
   Furthermore, we let $\mathcal{E}_{j+1/2} (\boldsymbol{X}^n)$ and $\mathcal{E}_{j+1/2} (\boldsymbol{Y}^n)$ 
    be the Engquist-Osher numerical fluxes \eqref{EOjphdef} applied to~$\boldsymbol{X}^n$ and~$\boldsymbol{Y}^n$, and similarly 
     for $\mathcal{B}_{j+1/2}^n$ and $\mathcal{F}_{j+1/2}^n$. Clearly, 
     \begin{align*} 
     &\mathcal{F}_{j+1/2}^n ( \boldsymbol{Y}^n ) -  \mathcal{F}_{j+1/2}^n  (\boldsymbol{X}^n )  \\ 
    &   = \mathcal{B}_{j+1/2}^n ( \boldsymbol{Y}^n ) -  \mathcal{B}_{j+1/2}^n ( \boldsymbol{X}^n ) 
      + \beta^n \bigl( \mathcal{E}_{j+1/2}^n ( \boldsymbol{Y}^n ) -  \mathcal{E}_{j+1/2}^n ( \boldsymbol{X}^n ) \bigr)  \\
       & = \Upw \bigl( \tilde{q}_{j+1/2}^n ; \vartheta_j^n, \vartheta_{j+1}^n \bigr) + \beta^n \gamma_{j+1/2} \biggl( \int_{X_j^n}^{Y_j^n} 
        (f')^+ (s) \, \mathrm{d} s + \int_{X_{j+1}^n}^{Y_{j+1}^n} 
        (f')^- (s) \, \mathrm{d} s  \biggr)  \\ 
      &  = \eta_{j+1/2}^n  \vartheta_j^n   - \nu_{j+1/2}^n  \vartheta_{j+1}^n , 
      \end{align*} 
       where we define 
       \begin{align*} 
       \eta_{j+1/2}^n
       & \coloneqq   
       ( \tilde{q}_{j+1/2}^n)^+   + \gamma_{j+1/2}  \int_0^1 (f')^+ \bigl( X_{j}^n + \sigma (Y_{j}^n - X_{j}^n) \bigr) \, \mathrm{d} \sigma \geq 0  ,  \\
            \nu_{j+1/2}^n & \coloneqq -  \biggl( 
        ( \tilde{q}_{j+1/2}^n)^-  +  \gamma_{j+1/2} \int_0^1 (f')^- \bigl( X_{j+1}^n + \sigma (Y_{j+1}^n - X_{j+1}^n) \bigr) \, \mathrm{d} \sigma\biggr) 
       \geq 0,           \end{align*} 
       hence 
   \begin{align*} 
 -  \lambda  \bigl( [\Delta \mathcal{F} ( \boldsymbol{Y})]_j^n -   [\Delta \mathcal{F} ( \boldsymbol{X})]_j^n \bigr) 
&   =   - \lambda \bigl( \eta_{j+1/2}^n + \nu_{j-1/2}^n \bigr) \vartheta_j^n + \lambda \eta_{j-1/2}^n \vartheta_{j-1}^n + \lambda \nu_{j+1/2}^n \vartheta_{j+1}^n.  
       \end{align*} 
         With this in mind, we obtain from \eqref{eq:SI_scheme_X} and the corresponding scheme for~$\smash{Y_j^n}$ 
   \begin{align}  \label{praffeqns1} \begin{split} 
    \vartheta_{-1}^{n+1} & = \bigl( \kappa_{-1}^n - \lambda  ( \eta_{-1/2}^n + \nu_{-3/2}^n ) \bigr) 
    \vartheta_{-1}^n + \lambda \nu_{-1/2}^n \vartheta_0^n,  \\
   \bigl(1 + \theta_0^{n+1} \bigr) \vartheta_0^{n+1} & = \bigl( \kappa_0^n - \lambda  ( \eta_{1/2}^n + \nu_{-1/2}^n ) \bigr) 
      \vartheta_0^n + \lambda \nu_{1/2}^n \vartheta_1^n +  \lambda \eta_{-1/2}^n \vartheta_{-1}^n +   \theta_1^{n+1} \vartheta_1^{n+1} \\ & \quad +  \tau \gamma_0  
              \bigl( R (\boldsymbol{p}_0^n Y_0^n / c, \boldsymbol{S}_0^n ) -   R (\boldsymbol{p}_0^n X_0^n / c, \boldsymbol{S}_0^n )\bigr),  \\
               \bigl(1 + 2\theta_j^{n+1} \bigr) \vartheta_j^{n+1} & = \bigl( \kappa_j^n - \lambda  ( \eta_{j+1/2}^n + \nu_{j-1/2}^n ) \bigr)
               \vartheta_j^n + \lambda \nu_{j+1/2}^n \vartheta_{j+1}^n + \lambda \eta_{j-1/2}^n \vartheta_{j-1}^n \\ & \quad  + \theta_{j+1}^{n+1} \vartheta_{j+1}^{n+1} 
                + \theta_{j-1}^{n+1} \vartheta_{j-1}^{n+1}  \\ & \quad  
+ \tau \gamma_j \bigl( R (\boldsymbol{p}_j^n Y_j^n / c, \boldsymbol{S}_j^n ) -   R (\boldsymbol{p}_j^n X_j^n / c, \boldsymbol{S}_j^n )\bigr),  \quad j=1, \dots, N-1, \\ 
 \bigl(1 + \theta_N^{n+1} \bigr) \vartheta_N^{n+1} & = \bigl(  \kappa_N^n - \lambda  ( \eta_{N+1/2}^n + \nu_{N-1/2}^n ) \bigr)
               \vartheta_N^n  + \lambda \nu_{N+1/2}^n \vartheta_{N+1}^n   + \lambda \eta_{N-1/2}^n \vartheta_{N-1}^n        \\ & \quad        + \theta_{N-1}^{n+1} \vartheta_{N-1}^{n+1}   
+ \tau \gamma_N \bigl( R (\boldsymbol{p}_N^n Y_N^n / c, \boldsymbol{S}_N^n ) -   R (\boldsymbol{p}_N^n X_N^n / c, \boldsymbol{S}_N^n )\bigr),  \\
 \vartheta_{N+1}^{n+1} & = \bigl(  \kappa_{N+1}^n - \lambda  ( \eta_{N+3/2}^n + \nu_{N+1/2}^n ) \bigr)
               \vartheta_{N+1}^n + \lambda \eta_{N+1/2}^n \vartheta_{N}^n.  
\end{split} 
   \end{align} 
    Finally, we define  
   \begin{align*} 
   \rho_j^n := 
    \int_0^1 \partial_s R \Bigl( (\boldsymbol{p}_j^n / c) \bigl( X_j^n + \sigma (Y_j^n-X_j^n)  \bigr), \boldsymbol{S}_j^n \Bigr)  \, \mathrm{d} \sigma, 
     \quad j=0, \dots, N,   
   \end{align*} 
   such that 
   \begin{align*} 
    \tau \gamma_j \bigl( R (\boldsymbol{p}_j^n Y_j^n / c, \boldsymbol{S}_j^n ) -   R (\boldsymbol{p}_j^n X_j^n / c, \boldsymbol{S}_j^n )\bigr) 
    = \tau \gamma_j \rho_j^n \vartheta_j^n %= \tau \gamma_j \rho_j^n \sgn \bigl( D_j^n \bigr) \bigl|  D_j^n \bigr|   
    , \quad j=0, \dots, N. 
   \end{align*}
   Consquently, we may write \eqref{praffeqns1} as 
   \begin{align*} 
    \vartheta_{-1}^{n+1} & = \bigl( \kappa_{-1}^n - \lambda  ( \eta_{-1/2}^n + \nu_{-3/2}^n ) \bigr) 
    \vartheta_{-1}^n + \lambda \nu_{-1/2}^n \vartheta_0^n,  \\
   \bigl(1 + \theta_0^{n+1} \bigr) \vartheta_0^{n+1} & = \bigl( \kappa_0^n - \lambda  ( \eta_{1/2}^n + \nu_{-1/2}^n ) + \tau \gamma_0 \rho_0^n \bigr) 
      \vartheta_0^n + \lambda \nu_{1/2}^n \vartheta_1^n   +  \lambda \eta_{-1/2}^n \vartheta_{-1}^n  + \theta_1^{n+1} \vartheta_1^{n+1}, \\ 
               \bigl(1 + 2\theta_j^{n+1} \bigr) \vartheta_j^{n+1} & =          \bigl( \kappa_j^n - \lambda  ( \eta_{j+1/2}^n + \nu_{j-1/2}^n ) + \tau \gamma_j \rho_j^n \bigr)
               \vartheta_j^n + \lambda \nu_{j+1/2}^n \vartheta_{j+1}^n  \\ & \quad+ \lambda \eta_{j-1/2}^n \vartheta_{j-1}^n  + \theta_{j+1}^{n+1} \vartheta_{j+1}^{n+1} 
                + \theta_{j-1}^{n+1} \vartheta_{j-1}^{n+1},  \quad j=1, \dots, N-1, \\ 
 \bigl(1 + \theta_N^{n+1} \bigr) \vartheta_N^{n+1} & = \bigl(  \kappa_N^n - \lambda  ( \eta_{N+1/2}^n + \nu_{N-1/2}^n ) + \tau \gamma_N \rho_N^N \bigr)
               \vartheta_N^n + \lambda \eta_{N-1/2}^n \vartheta_{N-1}^n          \\ & \quad      + \lambda \nu_{N+1/2}^n \vartheta_{N+1}^n  
                 + \theta_{N-1}^{n+1} \vartheta_{N-1}^{n+1}, \\ 
            \vartheta_{N+1}^{n+1} & = \bigl(  \kappa_{N+1}^n - \lambda  ( \eta_{N+3/2}^n + \nu_{N+1/2}^n ) \bigr)
            \vartheta_{N+1}^n + \lambda \eta_{N+1/2}^n \vartheta_{N}^n.      
   \end{align*} 
   Since by \eqref{CFLSI}, all coefficients in these equations  are nonnegative, we get 
   \begin{align*} 
     \bigl| \vartheta_{-1}^{n+1}\bigr|  & \leq  \bigl( \kappa_{-1}^n - \lambda  ( \eta_{-1/2}^n + \nu_{-3/2}^n ) \bigr) 
    | \vartheta_{-1}^n |+ \lambda \nu_{-1/2}^n |\vartheta_0^n|,  \\
   \bigl(1 + \theta_0^{n+1} \bigr) \bigl|\vartheta_0^{n+1} \bigr| & \leq  \bigl( \kappa_0^n - \lambda  ( \eta_{1/2}^n + \nu_{-1/2}^n ) + \tau \gamma_0 \rho_0^n \bigr) 
     | \vartheta_0^n|  + \lambda \nu_{1/2}^n |\vartheta_1^n |+    \lambda \eta_{-1/2}^n |\vartheta_{-1}^n|  +  \theta_1^{n+1} \bigl| \vartheta_1^{n+1} \bigr|, \\ 
               \bigl(1 + 2\theta_j^{n+1} \bigr) \bigl|\vartheta_j^{n+1} \bigr| & \leq           \bigl( \kappa_j^n - \lambda  ( \eta_{j+1/2}^n + \nu_{j-1/2}^n ) + \tau \gamma_j \rho_j^n \bigr)
              \bigl| \vartheta_j^n \bigr|  + \lambda \nu_{j+1/2}^n \bigl| \vartheta_{j+1}^n \bigr|  \\ & \quad  + \lambda \eta_{j-1/2}^n \bigl| \vartheta_{j-1}^n\bigr| + \theta_{j+1}^{n+1} \bigl| \vartheta_{j+1}^{n+1} \bigr| 
                + \theta_{j-1}^{n+1} \bigl| \vartheta_{j-1}^{n+1} \bigr| ,  \quad j=1, \dots, N-1, \\ 
 \bigl(1 + \theta_N^{n+1} \bigr) \bigl| \vartheta_N^{n+1}  \bigr| & \leq  \bigl(  \kappa_N^n - \lambda  ( \eta_{N+1/2}^n + \nu_{N-1/2}^n ) + \tau \gamma_N \rho_N^N \bigr)
             | \vartheta_N^n |  + \lambda \eta_{N-1/2}^n \bigl| \vartheta_{N-1}^n   \bigr|        \\ & \quad  
               + \lambda \nu_{N+1/2}^n  \bigl| \vartheta_{N+1}^n \bigr|      + \theta_{N-1}^{n+1}\bigl|  \vartheta_{N-1}^{n+1} \bigr|, \\ 
 \bigl| \vartheta_{N+1}^{n+1} \bigr| & \leq  \bigl(  \kappa_{N+1}^n - \lambda  ( \eta_{N+3/2}^n + \nu_{N+1/2}^n ) \bigr)
               \bigl| \vartheta_{N+1}^n  \bigr| + \lambda \eta_{N+1/2}^n  | \vartheta_{N}^n |.  
   \end{align*} 
   Summing over these inqualities, cancelling terms, rewriting the result again in terms of~$\smash{\{Y_j^n\}}$ 
    and~$\smash{\{X_j^n\}}$, and taking into account that $\smash{\nu_{-3/2}^n  \geq 0}$ and 
     $\smash{\eta_{N+3/2}^n}$  we obtain the inequality 
    \begin{align*} 
     \sum_{j=-1}^{N+1} \bigl| Y_j^{n+1}-X_j^{n+1} \bigr| & \leq \sum_{j=-1}^{N+1} 
    \bigl(  \kappa_j^n    +  \tau \gamma_j \rho_j^n \bigr)   \bigl| Y_j^{n}-X_j^{n} \bigr| \\ & \quad 
          - \lambda \nu_{-3/2}^n |Y_{-1}^n-X_{-1}^n| - \lambda \eta_{N+3/2}^n  |  Y_{N+1}^n - X_{N+1}^n |  \\ 
           & \leq \sum_{j=-1}^{N+1} 
    \bigl(  1 + C \tau)    \bigl| Y_j^{n}-X_j^{n} \bigr|, 
    \end{align*} 
    that is, \eqref{lemm4.4ineq}, 
    where we take into account that $\kappa_j^n \leq 1+ C \tau$ with a suitable constant $C >0$ (see \eqref{kappajndef}) and that 
     the quantities $\smash{\rho_j^n}$ are  uniformly bounded. 
\end{proof}

System~\eqref{syst} can be written as $\boldsymbol{\varphi} ( \boldsymbol{X}^{n+1} ) = \boldsymbol{0}$, where $\boldsymbol{\varphi}$ is a nonlinear vector function defined by the left-hand side of \eqref{syst} and $\mathbfcal{J}_{\boldsymbol{\varphi}}\in\mathbb{R}^{(N+1)\times (N+1)}$ is its associated Jacobian matrix.

 \subsection{Numerical solution of the nonlinear system}  \label{subsec:nr} 
The Newton-Raphson method applied to \eqref{syst} reads
\begin{equation} \label{NRit} 
\mathbfcal{J}_{\boldsymbol{\varphi}}(\boldsymbol{u}^{[k]}) (\boldsymbol{u}^{[k+1]}-\boldsymbol{u}^{[k]} ) = -\boldsymbol{\varphi}(\boldsymbol{u}^{[k]}),\quad k=0,1,\dots,
\end{equation}
with the Jacobian matrix $\smash{\mathbfcal{J}_{\boldsymbol{\varphi}}(\boldsymbol{u})}$ given by \eqref{jacob-def} (recall that $a=\mathcal{D}'$). 
The iteration  starts with the initial vector $\boldsymbol{u}^{[0]}: =\boldsymbol{X}^{n}$ and evolves formally according to \eqref{NRit} until  the termination criterion 
\begin{align*}
 \dfrac{\lVert\boldsymbol{u}^{[k+1]} - \boldsymbol{u}^{[k]} \rVert_1}{\lVert\boldsymbol{u}^{[k]}\rVert_1} < \varepsilon
\end{align*}
is reached, 
where $\varepsilon>0$ is the tolerance and $\lVert\cdot\rVert_1$ is the $\ell_1$-norm. 
After convergence, we set $\boldsymbol{X}^{n+1}\coloneqq \boldsymbol{u}^{[k+1]}$. Since the matrix $\smash{\mathbfcal{J}_{\boldsymbol{\varphi}} ( \boldsymbol{u})}$ is strictly diagonally dominant by columns for 
all~$\boldsymbol{u}$,  it is invertible and the iteration \eqref{NRit} is well defined.

\subsection{Update of the percentage vector~$\smash{\bp_j^n}$ and the soluble concentrations~$\smash{\boldsymbol{S}_j^n}$} \label{subsec:updperc}  An inspection of the 
proof of Lemma~\ref{lemma:p_pos} reveals that although the update formula for the percentages \eqref{eq:scheme_p} 
 is an explicit upwind scheme, it  is still associated with a CFL condition that imposes a bound on 
  $\tau / \Delta \xi^2$ due to the presence of differences of $\mathcal{D}$-values divided by 
  $\Delta \xi$ in the convective flux, cf.\ \eqref{Jjphdef}, \eqref{Phijphdef} and 
   \eqref{eq:Phipflux}. Consequently, to remove this shortcoming so that the whole semi-implicit scheme
   (and not just the update formula for~$X$)  
    is associated with a CFL bound on $\tau / \Delta \xi$ only, we need to resort to an implicit difference scheme.

We write out all terms in \eqref{eq:scheme_p} and evaluate those containing~$\mu$ at the time~$t_{n+1}$:
\begin{align} \label{eq:SI_scheme_p} 
\begin{split}
\bp^{n+1}_j X^{n+1}_j &= \kappa_j^n\bp^{n}_jX^{n}_j 
- \lambda \big(\Phi_{j+1/2}^{n,n+1,+}\bp_j^{n+1} + \Phi_{j+1/2}^{n,n+1,-}\bp_{j+1}^{n+1} -\Phi_{j-1/2}^{n,n+1,+}\bp_{j-1}^{n+1} - \Phi_{j-1/2}^{n,n+1,-}\bp_{j}^{n+1} \big)\\
&\quad + \lambda\delta_{j,0}\beta^nq_{\rm f}^n\bp_{\rm f}^nX_{\rm f}^n + c\tau\gamma_j\bR_{\bC}(\bp_j^nX^{n}_j/c,\bS_j^n),
\end{split}
\end{align}
where
\begin{equation*}
\Phi_{j+1/2}^{n,n+1}\coloneqq \mathcal{F}_{j+1/2}^{n}-\mathcal{J}_{j+1/2}^{n+1}.
\end{equation*}
For the cells outside the tank, this reduces to \eqref{eq:j=-1_Tf} if $\qe=0$ and $j=-1$; otherwise,
\begin{align*}
\bp^{n+1}_{-1} X^{n+1}_{-1} &= (1-\tau\beta^nq_{\mathrm{out}}^n) \bp^{n}_{-1}X^{n}_{-1} + \lambda\beta^n\bigl(
 (\xi_{-1/2} q_{\mathrm{out}}^n +\qe^n ) \bp^{n}_{0}X^{n}_{0} - \big(\xi_{-3/2}q_{\mathrm{out}}^n+\qe^n\big) \bp^{n}_{-1}X^{n}_{-1}
\bigr),\\
\bp^{n+1}_{N+1}X^{n+1}_{N+1} & =  (1+\tau\beta^n q_{\mathrm{out}}^n) \bp^{n}_{N+1}X^{n}_{N+1} % \\
% &\qquad 
- \lambda\bigl((\alpha^n_{N+3/2}+\beta^n\qu^n)\bp^{n}_{N+1}X^n_{N+1} - \beta^n\qu^n\bp^{n}_{N}X^n_{N}\bigr),
\end{align*}
where we recall  that $q_{\mathrm{out}}^n := \qu^n+\qe^n$.
Let us focus the cells $j =0, \dots, N$ and put all the unknowns in a matrix (cells in rows and solid components in columns):
\begin{align*} 
 \boldsymbol{P}^{n}  & \coloneqq 
\begin{bmatrix}
(\bp_0^{n})^\rmT \\
(\bp_1^{n})^\rmT \\
\vdots \\
(\bp_N^{n})^\rmT
\end{bmatrix},\qquad
 \boldsymbol{W}^n \coloneqq 
\begin{bmatrix}
(\lambda\beta^nq_{\rm f}^n\bp_{\rm f}^n X_{\rm f}^{n} + c\frac{\tau}{2}\bR_{\bC}(\bp_0^nX_0^n/c,\bS_0^n))^\rmT\\
c\tau\bR_{\bC}(\bp_1^nX_1^n/c,\bS_1^n)^\rmT \\
\vdots\\
c\tau\bR_{\bC}(\bp_N^nX_N^n/c,\bS_N^n)^\rmT 
\end{bmatrix},
\\
\boldsymbol{M}(\bPhi,\boldsymbol{X})  &  \coloneqq  \diag(\boldsymbol{X} ) \\ & \quad  + \lambda 
\begin{bmatrix} 
\Phi_{1/2}^{+}-\Phi_{-1/2}^{-} &    \Phi_{1/2}^{-} & 0 & \cdots & 0 \\
- \Phi_{1/2}^{+}  & \Phi_{3/2}^{+}  - \Phi_{1/2}^{-} & \ddots & \ddots & \vdots \\ 
0 &  -\Phi_{3/2}^{+} & \ddots & \ddots & 0  \\
   \vdots & \ddots &   \ddots & \Phi_{N-1/2}^{+} - \Phi_{N-3/2}^{-} & \Phi_{N-1/2}^{-} \\
   0 & \cdots & 0  & - \Phi_{N-1/2}^{+} & \Phi_{N+1/2}^{+} - \Phi_{N-1/2}^{-} 
   \end{bmatrix}. 
\end{align*}  
where for a vector $\boldsymbol{X}$ we define $\diag (\boldsymbol{X})  \coloneqq \diag(X_0, \dots, X_N)$.  
With $\boldsymbol{\kappa} \coloneqq(\kappa_0, \ldots, \kappa_N)^\rmT$, we get the linear system
\begin{equation}\label{eq:lin_syst}
 \boldsymbol{M}(\bPhi^{n,n+1},\boldsymbol{X}^{n+1}) \boldsymbol{P}^{n+1} = 
  \diag(\boldsymbol{\kappa}^n)\diag(\boldsymbol{X}^n)\boldsymbol{P}^n + \boldsymbol{W}^n \eqqcolon \boldsymbol{\Theta}^n. 
\end{equation}
In the case $\smash{X_j^{n+1}=0}$ the percentage vector $\smash{\bp_j^{n+1}}$ is irrelevant and one can define $\smash{\bp_j^{n+1}\coloneqq\bp_j^{n}}$.
Then Equation~\eqref{eq:lin_syst} should be modified in the following way.
Row~$j$ is removed in $\smash{\bp^{n+1}}$, $\smash{\bp^{n}}$ and $\smash{\boldsymbol{W}^n}$, and both row~$j$ and column~$j$ should be removed in the matrices $\smash{\boldsymbol{M}(\bPhi^{n,n+1},\boldsymbol{X}^{n+1})}$, $\smash{\diag(\boldsymbol{\kappa}^n)}$ and $\smash{\diag(\boldsymbol{X}^n)}$.
Then, one verifies that $\smash{\boldsymbol{M}(\bPhi^{n,n+1},\boldsymbol{X}^{n+1})}$ is strictly diagonally dominant by columns, 
 and  therefore  invertible without any restrictions. Thus,   the implicit scheme \eqref{eq:lin_syst} is well defined. 
Furthermore, $\smash{\boldsymbol{M}(\bPhi^{n,n+1},\boldsymbol{X}^{n+1})^\rmT}$ is an M-matrix and hence has a  non-negative inverse.

\subsection{Monotonicity and invariant region property} \label{subsec:simon} 

\begin{lemma}\label{lemma:p_pos_SI}
If $\smash{\boldsymbol{\mathcal{U}}_j^n\in\Omega}$ for all~$j$ and \eqref{CFLSI} holds, then 
\begin{align*} 
p_j^{(k),n+1}\geq  0 \quad \text{\em  for all~$k=1, \dots, k_{\boldsymbol{C}}$ and all~$j$}.
 \end{align*} 
\end{lemma}

\begin{proof}
Since  $\smash{(\boldsymbol{M}(\bPhi^{n,n+1},\boldsymbol{X}^{n+1}))^{-1}\geq \boldsymbol{0} }$, we  estimate each 
 entry of $\smash{\boldsymbol{\Theta}^n \eqqcolon (\Theta_{jk}^n)}$. 
By Lemma~\ref{lemma:Rbounds}, 
\begin{align*}
\Theta_{0,k}^n  &= \kappa^n X_0^np_0^{(k),n} + \lambda\beta^nq_{\rm f}^n\bp_{\rm f}^n X_{\rm f}^{n} + c\frac{\tau}{2}R_{\bC}^{(k)}(\bp_0^nX_0^n/c,\bS_0^n)\\
&\geq
\left(1 - {\tau\zeta M_{q1}}\right)X_0^np_0^{(k),n} + 0 + c\frac{\tau}{2}\sum_{l\in I_{\bC,k}^-}\sigma_{\bC}^{(k,l)}\bar{r}^{(l)} \bigl(\bp_0^{n}X_0^n/c, \bS_0^{n} \bigr) p_0^{(k),n}X_0^n/c\\
&\geq\left(1 - {\tau\zeta M_{q1}} - {\tau}M_{\bC}/2 \right)p_0^{(k),n}X_0^n\geq 0, 
\end{align*}
whereas for all the other $j\neq 0$,  
\begin{align*}
\Theta_{j,k}^n 
&=\kappa^n X_j^n p_j^{(k),n} + c\tau R_{\bC}^{(k)}(\bp_j^nX_j^n/c,\bS_j^n)
\geq\left(1 - {\tau\zeta M_{q1}} - \tau M_{\bC}
\right)p_j^{(k),n}X_j^n\geq 0.
\end{align*}
\end{proof}

\begin{lemma}\label{lemma:sum_p_SI}
If $\smash{\boldsymbol{\mathcal{U}}_j^n\in\Omega}$ for all~$j$ and \eqref{CFLSI} holds, then
\begin{align}  \label{p_sumeq_SI} 
p_j^{(1),n+1} +  \cdots + p_j^{(k_{\boldsymbol{C}}),n+1} =  1 \quad \text{\em for all~$j$}.
 \end{align} 
\end{lemma}

\begin{proof}
If $X_j^{n+1}=0$, then by definition, 
\begin{align*}  
 p_j^{(1),n+1}+\cdots+p_j^{(k_{\bC}),n+1}=p_j^{(1),n}+\cdots+p_j^{(k_{\bC}),n} =1, 
 \end{align*} 
  so let us 
assume that $X_j^{n+1}>0$.
We sum up all equations in~\eqref{eq:SI_scheme_p} and use the notation
\begin{align*} 
\mathcal{P}_j^{n+1}\coloneqq p_j^{(1),n+1}+\cdots+p_j^{(k_{\bC}),n+1}
\end{align*} 
to obtain
\begin{align*}
X^{n+1}_j\mathcal{P}_j^{n+1} &= \kappa_j^n X^{n}_j \\
& \quad - \lambda \big(\Phi_{j+1/2}^{n,n+1,+}\mathcal{P}_j^{(k),n+1} + \Phi_{j+1/2}^{n,n+1,-}\mathcal{P}_{j+1}^{(k),n+1} - \Phi_{j-1/2}^{n,n+1,+}\mathcal{P}_{j-1}^{(k),n+1} - \Phi_{j-1/2}^{n,n+1,-}\mathcal{P}_{j}^{(k),n+1}\big)\\
& \quad + \lambda\delta_{j,0}\beta^nq_{\rm f}^nX_{\rm f}^n + \tau\gamma_j R(\bp_j^nX^{n}_j/c,\bS_j^n).
\end{align*}
We subtract component~$j$ of Equation~\eqref{eq:SI_scheme_Xa}, let $y_j^{n+1}\coloneqq\mathcal{P}_j^{n+1}-1$ and obtain
\begin{align*} 
X^{n+1}_jy_{j+1}^{n+1} &= - \lambda \big(\Phi_{j+1/2}^{n,n+1,+}y_j^{n+1} + \Phi_{j+1/2}^{n,n+1,-}y_{j+1}^{n+1} - \Phi_{j-1/2}^{n,n+1,+}y_{j-1}^{n+1} - \Phi_{j-1/2}^{n,n+1,-}y_{j}^{n+1}\big).
\end{align*}
Thus, with $\smash{\boldsymbol{y}_{j+1}^{n+1}\coloneqq(y_{0}^{n+1}, \ldots, y_{N}^{n+1})^\rmT}$, we get 
$\smash{\boldsymbol{M}^{n,n+1}\boldsymbol{y}_{j+1}^{n+1}=\bzero}$, which implies $\smash{\boldsymbol{y}_{j+1}^{n+1}=\bzero}$,
i.e.~\eqref{p_sumeq_SI}.
\end{proof}

For Equation~\eqref{eq:PDE_S} we have
\begin{align*}
  \bS^{n+1}_j &= \kappa_j^n\bS^{n}_j -
\lambda \Bigg(
\dfrac{(\rho_X\tilde{q}^{n}-\mathcal{F}^{n}+\mathcal{J}^{n+1})_{j+1/2}^{+}}{\rho_X - X_{j}^{n+1}}\bS_j^{n+1} + 
\dfrac{(\rho_X\tilde{q}^{n}-\mathcal{F}^{n}+\mathcal{J}^{n+1})_{j+1/2}^{-}}{\rho_X - X_{j+1}^{n+1}}\bS_{j+1}^{n+1}\\
&\quad -
\dfrac{(\rho_X\tilde{q}^{n}-\mathcal{F}^{n}+\mathcal{J}^{n+1})_{j-1/2}^{+}}{\rho_X - X_{j-1}^{n+1}}\bS_{j-1}^{n+1} - 
\dfrac{(\rho_X\tilde{q}^{n}-\mathcal{F}^{n}+\mathcal{J}^{n+1})_{j-1/2}^{-}}{\rho_X - X_{j}^{n+1}}\bS_{j}^{n+1} \Bigg)\\
&\quad + \lambda\delta_{j,0}\beta^nq_{\rm f}^n\bS_{\rm f}^n + \tau\gamma_j\bR_{\bS}(\bp_j^nX^{n}_j/c,\bS_j^n).\label{eq:scheme_S}
\end{align*}
For $j=-1, N+1$, this scheme is explicit and analogous to those above.
For $j=0, \ldots, N$, we write the formula in matrix form as follows. 
Define
\begin{align*}
\theta_{j+1/2}^{n,n+1}&\coloneqq (\rho_X\tilde{q}^{n}-\mathcal{F}^{n}+\mathcal{J}^{n+1})_{j+1/2},\qquad y_j^n\coloneqq\frac{1}{\rho_X - X_{j}^{n}},\\
\mathbfcal{S}^n  & \coloneqq 
\begin{bmatrix}
(\boldsymbol{S}_0^{n})^\rmT\\
(\boldsymbol{S}_1^{n})^\rmT\\
\vdots\\
(\boldsymbol{S}_{N}^{n})^\rmT
\end{bmatrix},
\qquad
\mathbfcal{W}^n \coloneqq 
\begin{bmatrix}
(\lambda\beta^nq_{\rm f}^n\bS_{\rm f}^n + \frac{\tau}{2}\bR_{\bS}(\bp_0^nX_0^n/c,\bS_0^n))^\rmT\\
\tau\bR_{\bS}(\bp_1^nX_1^n/c,\bS_1^n)^\rmT\\
\vdots\\
\tau\bR_{\bS}(\bp_N^nX_N^n/c,\bS_N^n)^\rmT
\end{bmatrix}.
\end{align*}
Since $\tilde{q}_{-1/2}^{n}\leq 0$ and $\tilde{q}_{N+1/2}^{n}\geq 0$, we have
\begin{align*}
\theta_{-1/2}^{n,n+1,+} =  \rho_X\tilde{q}_{-1/2}^{n+1,+} = 0,\quad 
\theta_{N+1/2}^{n,n+1,-} =  \rho_X\tilde{q}_{N+1/2}^{n+1,-} = 0.
\end{align*}
Then we form the tridiagonal matrix
\begin{multline*}
\boldsymbol{M}_{\!\bS}(\boldsymbol{\theta},\boldsymbol{y})  \coloneqq \boldsymbol{I}_{N+1}\\
 + \lambda 
   \begin{bmatrix} 
   (\theta_{1/2}^{+}-\theta_{-1/2}^{-})y_{0} & \theta_{1/2}^{-}y_{1} & 0 & \cdots & 0 \\
   -\theta_{1/2}^{+} y_{0}  & (\theta_{3/2}^{+} - \theta_{1/2}^{-})y_{1}  &  \ddots& \ddots & \vdots \\ 
0 & -\theta_{3/2}^{+} y_{1} & \ddots & \ddots & 0  \\
   \vdots & \ddots & \ddots & (\theta_{N-1/2}^{+} - \theta_{N-3/2}^{-})y_{N-1} & \theta_{N-1/2}^{-}y_{N} \\
   0 & \cdots & 0 & - \theta_{N-1/2}^{+}y_{N-1}& (\theta_{N+1/2}^{+}  - \theta_{N-1/2}^{-}) y_{N}
   \end{bmatrix}.
\end{multline*}
Then we get the linear system
\begin{equation}\label{eq:lin_syst_S}
 \boldsymbol{M}_{\!\bS}(\boldsymbol{\theta}^{n,n+1},\boldsymbol{y}^{n+1}) \mathbfcal{S}^{n+1} = \diag(\boldsymbol{\kappa}^n) \mathbfcal{S}^n + \mathbfcal{W}^n.
\end{equation}
The matrix $\boldsymbol{M}_{\!\bS}(\boldsymbol{\theta}^{n,n+1},\boldsymbol{y}^{n+1})$ is diagonally dominant by columns; hence its transpose is an M-matrix and invertible with a non-negative inverse, so that \eqref{eq:lin_syst_S} defines~$\mathbfcal{S}^{n+1}$.

\begin{lemma}\label{lemma:S_pos_SI}
If $\smash{\boldsymbol{\mathcal{U}}_j^n\in\Omega}$ for all~$j$ and \eqref{CFLSI} holds, then 
\begin{align*}  \label{p_posineq_SI} 
S_j^{(k),n+1}\geq  0 \quad \text{\em  for all~$k=1, \dots, k_{\boldsymbol{S}}$ and all~$j$}.
 \end{align*} 
\end{lemma}

\begin{proof}
This is similar to the proof of Lemma~\ref{lemma:p_pos_SI}.
The element $(j,k)$, $j\neq 0$, of the matrix on the right-hand side of~\eqref{eq:lin_syst_S} is estimated by
\begin{align*}
&\kappa^n S_j^{(k),n} + \tau R_{\bS}^{(k)}(\bp_j^nX_j^n/c,\bS_j^n)
\geq\left(1 - {\tau\zeta M_{q1}} - \tau M_{\bS}
\right)S_j^{(k),n}\geq 0.
\end{align*}
\end{proof}

\section{Numerical simulations} \label{sec:numer}

For all the examples, we consider a cylindrical SBR of depth $B=3\,\rm m$ and cross-sectional area $A =400\, \rm  m^2$, and the reactive settling process of an activated sludge described by the ASM1 model~\cite{Henze2000ASMbook}; see Table~\ref{table:AMS1_vari}. 
The constitutive functions used in the examples are
\begin{align*}
\vhs(X)  \coloneqq  
\dfrac{v_0}{ 1 + (X/ \breve{X})^\eta} , \quad 
\sigma_{\mathrm{e}}(X) \coloneqq  \begin{cases} 0  &  \text{if $X  < X_{\mathrm{c}}$,}  \\
 \sigma_0 (X-X_{\mathrm{c}}) & \text{if $X \geq X_{\mathrm{c}}$}
 \end{cases}  
\end{align*}
with $v_0 = 1.76\times 10^{-3}\, \rm m/s$, $\breve{X} = 3.87\, \rm kg/m^3$, $\eta = 3.58$, $X_{\mathrm{c}} = 5\, \rm kg/m^3$ and $\sigma_0 = 0.2\,\rm m^2/s^2$. 
Other parameters are $\rho_X = 1050\, \rm kg/m^3$, $\rho_L = 998\, \rm kg/m^3$, $g = 9.81 \,\rm m/s^2$, and $B_{\mathrm{c}} = 2\, \mathrm{m}$.  
To satisfy \eqref{vhsass}, one could multiply $\vhs(X)$ by a function that is one for most concentrations but tends to zero smoothly as $X\to\hat{X}^-$, where the maximum concentration~$\hat{X}$ is set to a large number; e.g.\ $\hat{X}=30$~kg/m$^3$ for activated sludge that we simulate here.
When $\sigma_0>0$, the second-order derivative compression term will balance the convective flux and the particulate concentration~$X$ never reaches~$\hat{X}$ in  simulations.
To run the scheme with $\sigma_0=0$, an alternative is to set a concentration $X^{\mathrm{t}}$  from which we redefine and extend the settling velocity function by its tangent as
\begin{equation*}
\vhs(X) \coloneqq\begin{cases}
v_0/( 1 + (X/ \breve{X})^\eta)  & \text{if $0\leq X\leq X^{\mathrm{t}}$,} \\
\vhs(X^{\mathrm{t}}) + \vhs'(X^{\mathrm{t}})(X-X^{\mathrm{t}}) & \text{if $X^{\mathrm{t}}<X\leq \hat{X}$,} 
\end{cases}
\end{equation*}
where $\hat{X}$ is given by the intersection of the tangent with the $X$-axis (zero velocity), i.e., 
\begin{align*} 
\hat{X}\coloneqq X^{\mathrm{t}}- \vhs(X^{\mathrm{t}})/ \vhs'(X^{\mathrm{t}}). \end{align*}  
We here utilize $X^\mathrm{t}=25 \, \mathrm{kg}/\mathrm{m}^3$, such that $\hat{X}=31.992  \, \mathrm{kg}/\mathrm{m}^3$. 

The initial concentrations have been chosen as 
\begin{align*}  
       \bC^0(z) & =  \begin{cases}  \boldsymbol{0}   &  \text{if $z<2.0 \,{\rm m}$,} \\
      ( 0.8889,    0.0295,    1.4503,   0.0904,   0.7371,    0.0025 )^{\rm T}   &  \text{if  $z \geq 2.0 \,{\rm m}$,}  
       \end{cases} \\ 
       \bS^0(z) & =  \begin{cases}   \boldsymbol{0}    &  \text{if  $z<2.0 \,{\rm m}$,} \\
       ( 0.04,    0.0026,    0.0,    0.0333,    0.0004,    0.0009 )^{\rm T}   &  \text{if  $z \geq 2.0 \,{\rm m}$} 
        \end{cases}       
\end{align*}
(with units as in Table~\ref{table:AMS1_vari}) while the feed concentrations are \cite{Henze2000ASMbook}
\begin{align*}
 \bCf(t) &= \frac{\Xf(t)}{(0.04+0.16+0.096+1\cdot 10^{-6}) c}
{(0.04,\ 0.16-0.01828,\ 0.096,\ 1\cdot 10^{-6},\ 0,\ 0.01828)^\rmT},\\
  \bSf(t) &= ( 0.04,\    0.064,\    0.0,\    0.001,\    0.0125   ,\ 0.0101 )^{\rm T},
\end{align*}
where the total solids feed concentration~$\Xf(t)$ is given by  Table~\ref{table:Example1}.
When plotting the particulate concentrations, we prefer to not plot $\smash{C^{(2)}=X_{\mathrm{S-ND}}}$, but rather $\smash{X_\mathrm{S}=C^{(2)}+X_{\mathrm{ND}}=C^{(2)}+C^{(6)}}$ and $\smash{X_{\mathrm{ND}}=C^{(6)}}$.
All results are shown after transforming back to the original coordinates.

%%%%%%%%%%%%%%%%%%%%%%%%%%%%%%%%
\subsection{Example~1}

\begin{table}[t]
	\caption{Example~1: time functions for the simulation. \label{table:Example1}} 	
	\smallskip 
	\centering
	{\small
		\begin{tabular}{lcccccc} \toprule 
			Stage  & Time period [h]  & $X_\mathrm{f}(t) [\rm kg/m^3]$ & $\Qf(t) [\rm m^3/h]$ & $\Qu(t) [\rm m^3/h]$ & $\Qe(t) [\rm m^3/h]$ & Model\\
			\midrule
			Fill   & $0\leq t<1$   & 5 & 790 & 0  & 0    & PDE\\
			React  & $1\leq t<3$   & 0 & 0   & 0  & 0    & ODE\\
			Settle & $3\leq t<5$   & 0 & 0   & 0  & 0    & PDE\\
			Draw   & $5\leq t<5.5$ & 0 & 0   & 0  & 1570 & PDE\\
			Idle   & $5.5\leq t<6$ & 0 & 0   & 10 & 0    & PDE\\
			\bottomrule 
	\end{tabular}} 
\end{table}%

\begin{figure}[t]
	\centering 
	\centering 
	\includegraphics{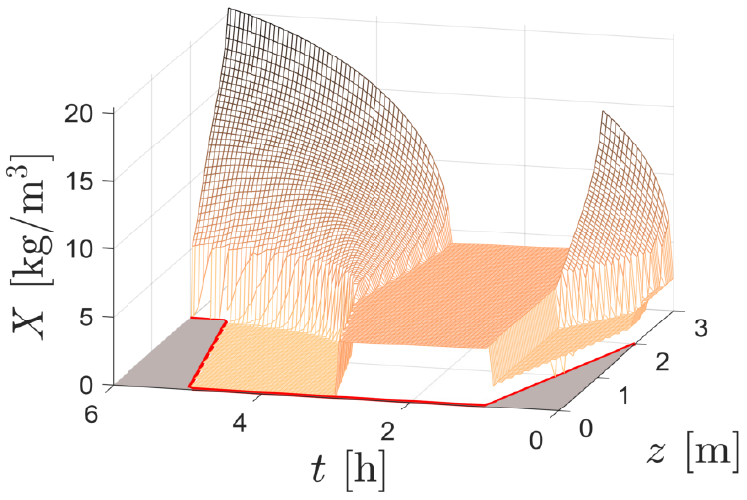} 
	\caption{Example~1: simulated concentration (semi-implicit scheme, $N = 100$) of 
	 total suspended solids.}   \label{fig:Concen1}
\end{figure}%

\begin{figure}[t] 
	\centering 
	\includegraphics{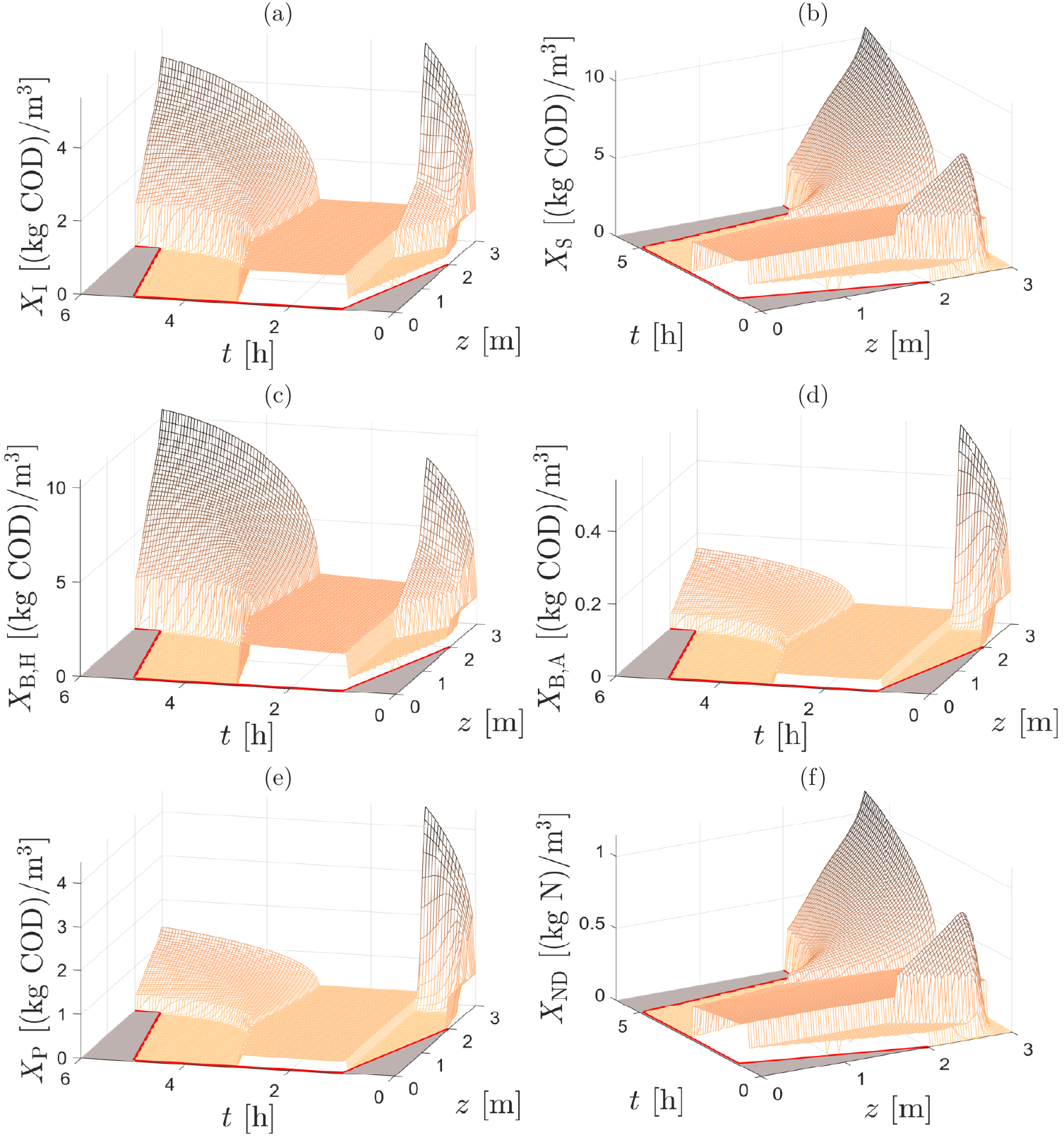} 
	\caption{Example~1: simulated concentrations   of  solid components (semi-implicit scheme, $N = 100$): 
		(a) particulate inert organic matter, 			(b) slowly biodegradable substrate, (c) active heterotrophic biomass, 
			 (d)  active autotrophic biomass, (e) particle products from biomass decay, 
			  (f) particulate biodegradable organic nitrogen.} \label{fig:Concen2}
	\end{figure}%
	
	\begin{figure}[t]
		\centering 
	\includegraphics{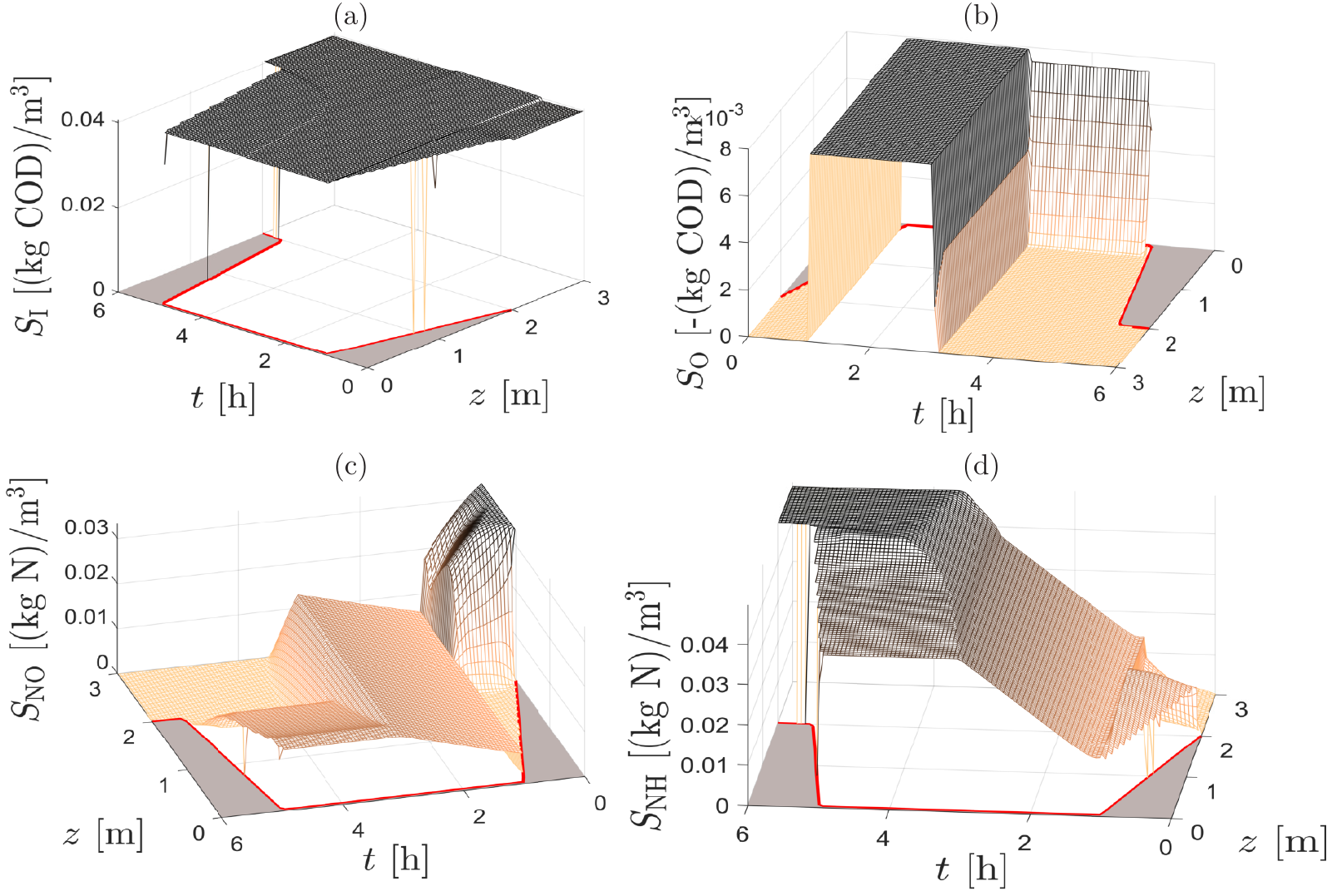} 
	\caption{Example~1: simulated densities (semi-implicit scheme, $N = 100$) of soluble components (part~1): 
	(a) soluble inert organic matter, (b) oxygen, (c) nitrate and nitrite nitrogen, 
	(d)  $\mathrm{NH}_{4}^+ \ + \ \mathrm{NH}_3$ nitrogen.} \label{fig:Concen3}
	\end{figure}%
	
	\begin{figure}[t]
		\centering 
	\includegraphics{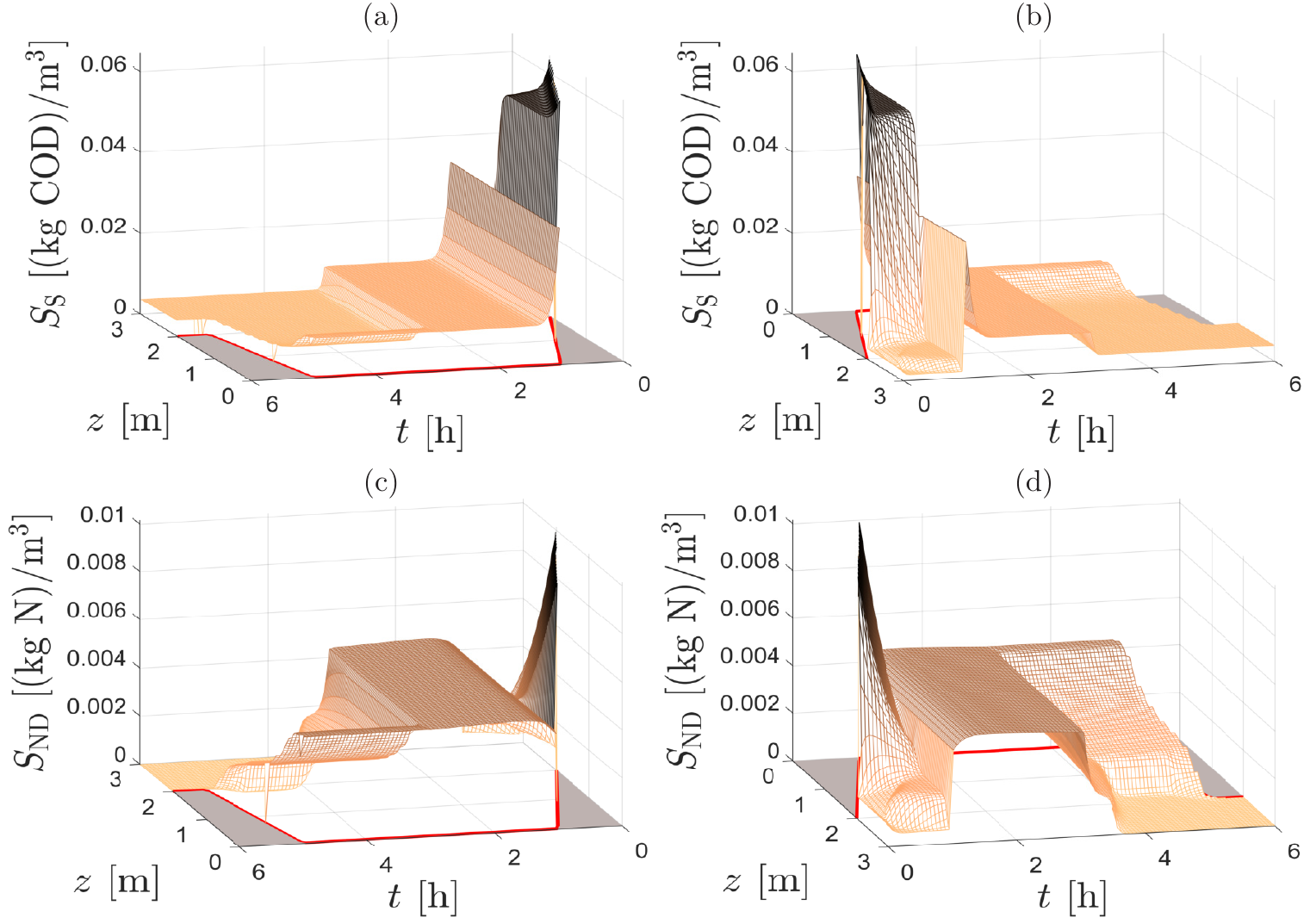} 
	\caption{Example~1: simulated densities (semi-implicit scheme, $N = 100$) of soluble components (part~2): 
	(a, b) readily biodegradable substrate (two different views), 
	  (c, d)  soluble biodegradable organic nitrogen (two different views).} \label{fig:Concen4}
	\end{figure}%

 It is the 
purpose of this example to illustrate the SBR model as a whole. (The performance of the three numerical schemes SBR2 in~\cite{SDAMM_SBR2}, and the explicit and semi-implicit ones here in terms of errors and efficiency is studied in Examples~2 and~3.) We simulated the five stages of an SBR as outlined in Table~\ref{table:Example1}. The duration of the 
  whole cycle of stages during a couple of hours is realistic.  The results are illustrated in Figures~\ref{fig:Concen1} to \ref{fig:Concen4}, which depict the concentration profiles of total suspended solids, particulate, and soluble components within the reactor vessel, respectively.

\subsection{Example~2}

\begin{figure}[t] 
	\centering 
	\includegraphics{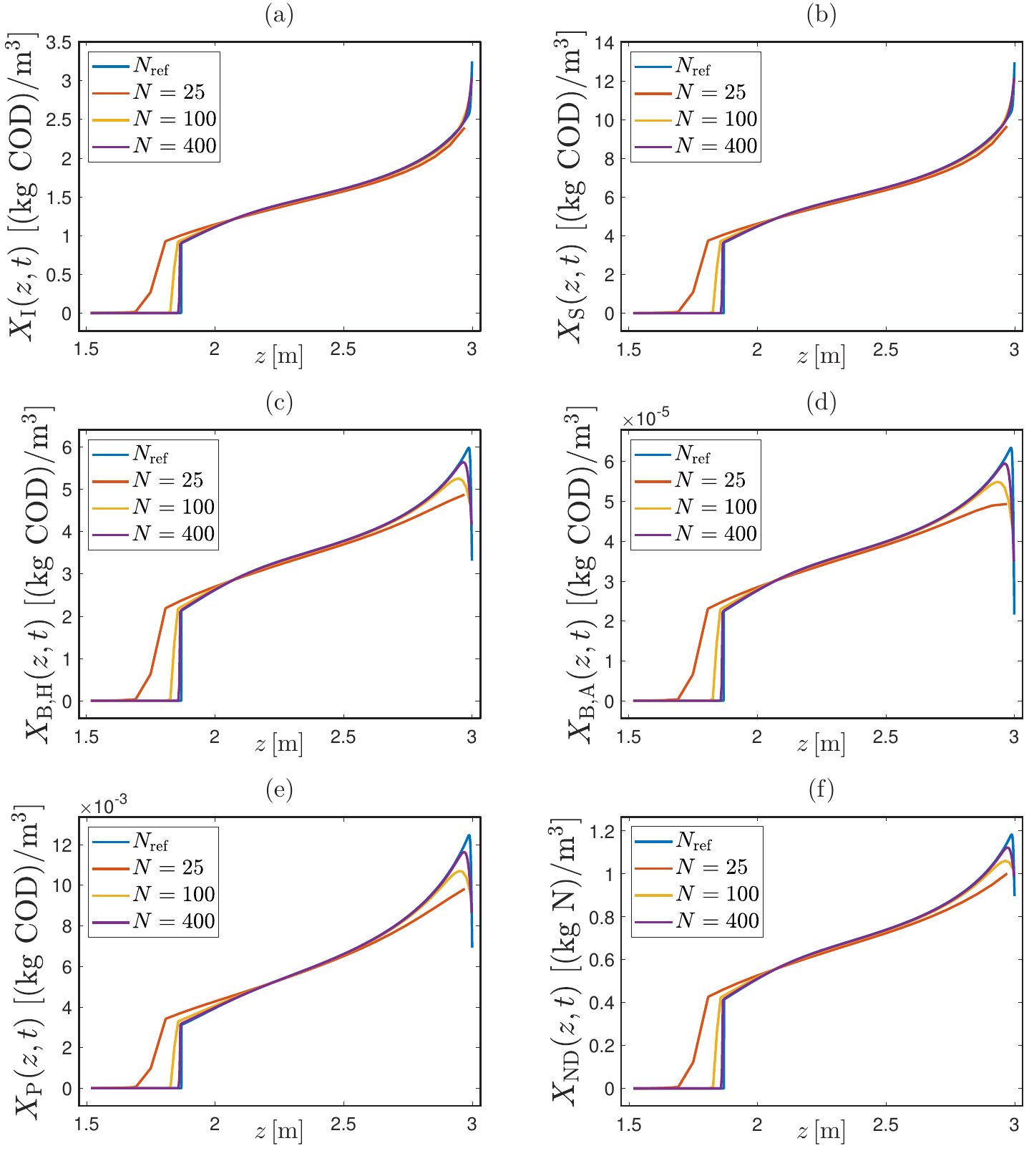} 
	\caption{Example~2:  simulated concentrations  (semi-implicit scheme) at $T=1 \, \mathrm{h}$   of 
	solid components: 
		(a) particulate inert organic matter, 			(b) slowly biodegradable substrate, (c) active heterotrophic biomass, 
			 (d)  active autotrophic biomass, (e) particle products from biomass decay, 
			  (f) particulate biodegradable organic nitrogen.} \label{fig:Perf3a} 
	\end{figure}%

\begin{figure}[t]
		\centering 
		\includegraphics{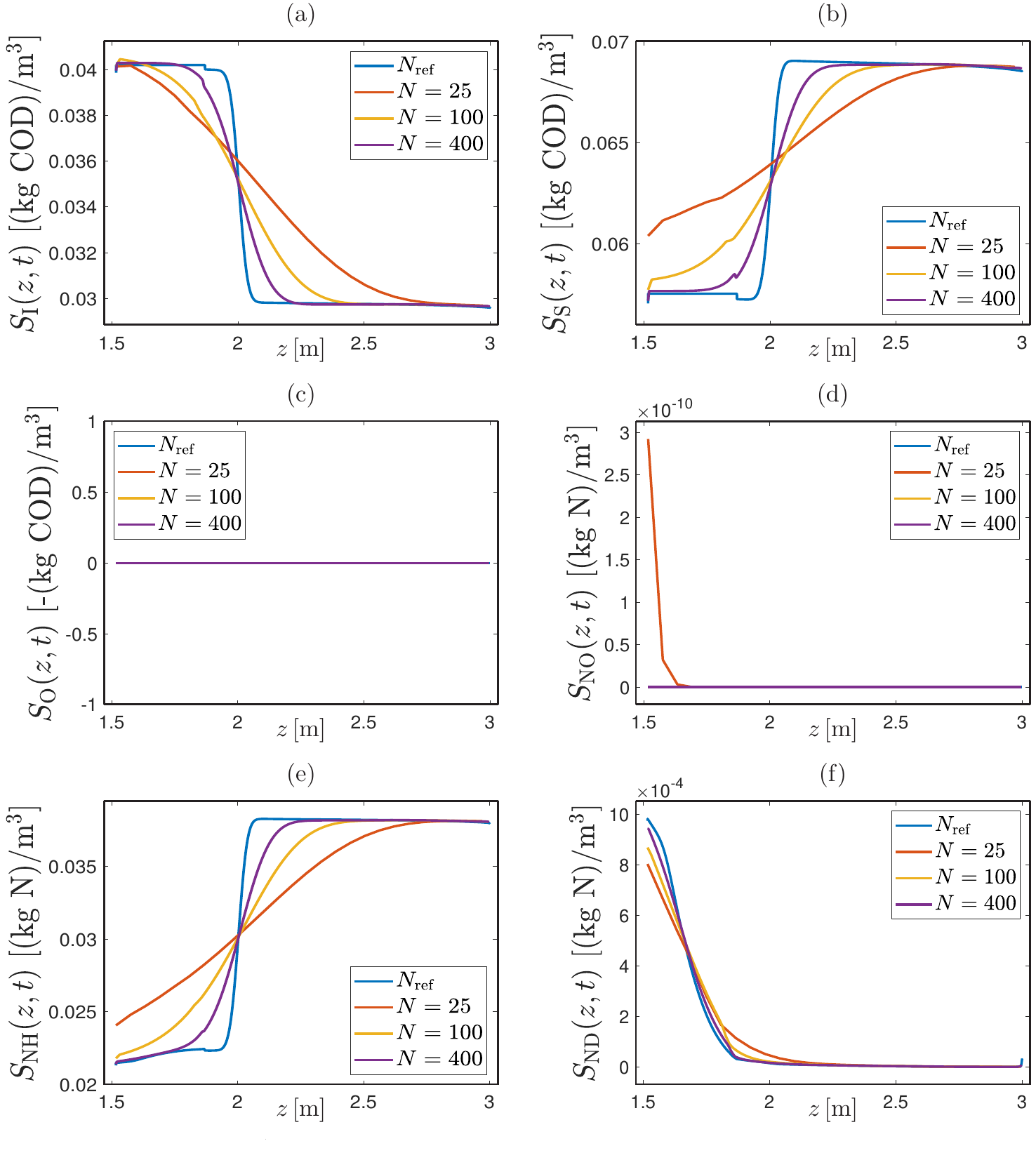}  
		\caption{Example~2: simulated concentrations (semi-implicit scheme)    at $T=1 \, \mathrm{h}$ 
	  of 
	soluble  components: (a) soluble inert organic matter, 
	(b) readily biodegradable substrate, 
	(c) oxygen, (d) nitrate and nitrite nitrogen, 
	(e)  $\mathrm{NH}_{4}^+ \ + \ \mathrm{NH}_3$ nitrogen, (f) soluble biodegradable organic nitrogen.} \label{fig:Perf3b}  
	\end{figure}%

The prime motivation of this example is to study  the numerical errors of the two new schemes defined for the PDE model and the scheme SBR2 of~\cite{SDAMM_SBR2}. To this end we 
 consider the scenario defined by Table~\ref{table:Example_Sol_Ref}, which refers to a shorter period of  total simulated time 
 ($ T = 1\,\mathrm{h}$) and does not include a react stage (for which the original SBR model, as studied in Example~1, stipulates
  a model by a system of ODEs). We calculated a reference  solution, using the same feed conditions as in Example~1 with $N = N_{\mathrm{ref}} \coloneqq 4800$.
This solution was obtained by the explicit scheme of Section~\ref{subsec:explicheme}.
The relative approximate numerical error
\begin{align*}
	e_N^{\mathrm{rel}}(t)\coloneqq \sum_{k=1}^{k_{\bC}} \frac{\|C_N^{(k)}-C_{N_{\mathrm{ref}}}^{(k)}(\cdot, t)\|_{L^1(0, B)}}{\|C_{N_{\mathrm{ref}}}^{(k)}(\cdot, t)\|_{L^1(0, B)}}+\sum_{k=1}^{k_{\bS}} \frac{\|S_N^{(k)}-S_{N_{\mathrm{ref}}}^{(k)}(\cdot, t)\|_{L^1(0, B)}}{\|S_{N_{\mathrm{ref}}}^{(k)}(\cdot, t)\|_{L^1(0, B)}}. 
\end{align*}
compares the approximate solution to the reference solution at a given time point~$t$ for a specific  number of cells~$N$.

In Figures~\ref{fig:Perf3a} and~\ref{fig:Perf3b}, we demonstrate the convergence of the numerical solutions, all produced with the semi-implicit scheme 
 and a tolerance  $\varepsilon =10^{-8}$, to the reference solution.  That  value was chosen 
  by previous experience. 
 The effect of the tolerance parameter~$\varepsilon$ itself is studied 
in Table~\ref{table:Epsilon}
 for simulations done with $N=50$,  $100$,  and $200$, 
  in terms of the   average number of iterations during the simulation, the relative error, and CPU  time. 
   It turns out that the relative error depends only marginally on the choice of~$\varepsilon$. As one should 
    expect, the average number of iterations (of the Newton-Raphson scheme, per time step), 
     as well as the CPU time for  the whole simulation,  
     consistently increase  with decreasing~$\varepsilon$.  Overall, it appears that 
      the relative error and CPU time, and therefore efficiency, does not depend critically 
       on~$\varepsilon$ (at least for the model functions and parameters chosen here). 
  
  For the same configuration, we compare in Table~\ref{table:errors_rel} and 
 Figure~\ref{fig:Error_rel} the relative errors and CPU times obtained by the three schemes SBR2 of 
  \cite{SDAMM_SBR2}, and the explicit and semi-implicit schemes of Sections~\ref{sec:numscheme} and~\ref{scheme:si}, respectively. 
   For both explicit and semi-implicit schemes, simulations were performed for $N=25,\dots,1600$ using the Engquist-Osher flux~\eqref{EOjphdef}. The semi-implicit scheme was calculated with the tolerance $\varepsilon=10^{-8}$.
   It  turns out that 
    with the exception of very coarse discretizations, the semi-implicit scheme 
     is  significantly more efficient in error reduction per CPU time than its 
      explicit counterpart due to its more favorable CFL condition, and is more 
       efficient also than SBR2 for sufficiently fine discretizations.  It is noteworthy that all 
        schemes converge to the same solution.

\begin{table}[t]
	\caption{Example~2: time functions for the simulation of the reference solution. \label{table:Example_Sol_Ref}}
	\smallskip 
	\centering
	{\small  \begin{tabular}{cccccc} \toprule
			Time period $[\mathrm{h}]$ & $X_\mathrm{f}(t) [\rm kg/m^3]$& $\Qf(t) [\rm m^3/h]$& $\Qu(t) [\rm m^3/h]$& $\Qe(t) [\rm m^3/h]$ & {Model} \\
			\midrule
			$0 \ \  \leq t<0.3$ & 5 & 2660 & 0& 0 & PDE\\
			$0.3 \ \leq t<0.85$ &0  &  0 & 0& 0 & PDE\\
			$0.85\leq t<0.95$ &0  &  0 & 0& 6000 & PDE\\ 
			$\hspace{-3.0mm}0.95\leq t<1$ & 0  &  0 & 100& 0 & PDE\\
			\bottomrule
	\end{tabular}}  
\end{table}%

\begin{table}[t]
		\caption{Example~2: effect of the tolerance $\varepsilon$ on the average number of iterations in the Newton-Raphson method, the errors $e_N^{\mathrm{rel}}$ and CPU times during a simulation of $T = 1\,\mathrm{h}$.
The errors are computed with the reference solution obtained by the explicit scheme with $N= 4800$ and the Godunov numerical flux.}  \label{table:Epsilon}
		\smallskip 
		\centering
			{\small \begin{tabular}{llcll} \toprule
								&$ \varepsilon$ & Avg. iterations  & $e_{N}^{\mathrm{rel}}(t)$  & CPU [$\mathrm{s}$]\\ \midrule 
			\multirow{12}{*}{$N=50$}	& 1E-1  &  1.00 &  0.0497756500328602 &  0.1250 \\ 
				& 1E-2  &  1.03 &  0.0497756500328602 &  0.1406 \\ 
				& 1E-3  &  1.60 &  0.0497679356628912 &  0.1797 \\ 
				& 1E-4  &  2.00 &  0.0497674399512266 &  0.1953 \\ 
				& 1E-5  &  2.01 &  0.0497674389897926 &  0.1875 \\ 
				& 1E-6  &  2.04 &  0.0497674143163951 &  0.1797 \\ 
				&1E-7  &  2.12 &  0.0497674123749513 &  0.1953 \\ 
				&1E-8  &  2.77 &  0.0497674121051164 &  0.2109 \\ 
				&1E-9  &  2.92 &  0.0497674120899535 &  0.2344 \\ 
				&1E-10 &  2.96 &  0.0497674120878119 &  0.2188 \\ 
				&1E-11 &  3.07 &  0.0497674120875128 &  0.2422 \\ 
				& 1E-12 &  3.69 &  0.0497674120874927 &  0.2560 \\ 
				\midrule
		\multirow{12}{*}{$N=100$}	&	1E-1  &  1.00 &  0.0244467141865713 &  0.3984 \\ 
			&	1E-2  &  1.01 &  0.0244467141865713 &  0.4062 \\ 
			&	1E-3  &  1.21 &  0.0244432794902836 &  0.4453 \\ 
			&	1E-4  &  2.00 &  0.0244427209345141 &  0.6719 \\ 
			&	1E-5  &  2.01 &  0.0244427236620445 &  0.5938 \\ 
			&	1E-6  &  2.02 &  0.0244427043886650 &  0.6172 \\ 
			&	1E-7  &  2.06 &  0.0244427032349412 &  0.6250 \\ 
			&	1E-8  &  2.17 &  0.0244427031258883 &  0.7188 \\ 
			&	1E-9  &  2.87 &  0.0244427030983648 &  0.7656 \\ 
			&	1E-10 &  2.97 &  0.0244427030958998 &  0.7812 \\ 
			&	1E-11 &  3.01 &  0.0244427030956931 &  0.7734 \\ 
			&	1E-12 &  3.11 &  0.0244427030956539 &  0.8672 \\ \midrule
		\multirow{12}{*}{$N=200$}
			&	1E-1  &  1.00 &  0.0115032639159754 &  1.4609 \\ 
			&	1E-2  &  1.00 &  0.0115032639159754 &  1.4766 \\ 
			&	1E-3  &  1.07 &  0.0115029463554388 &  1.6406 \\ 
			&	1E-4  &  1.86 &  0.0115009448742330 &  2.0234 \\ 
			&	1E-5  &  2.01 &  0.0115008746422637 &  2.2578 \\ 
			&	1E-6  &  2.02 &  0.0115008746587289 &  2.2188 \\ 
			&	1E-7  &  2.03 &  0.0115008741995477 &  2.1250 \\ 
			&	1E-8  &  2.08 &  0.0115008740687287 &  2.2656 \\ 
			&	1E-9  &  2.24 &  0.0115008740547202 &  2.3281 \\ 
			&	1E-10 &  2.91 &  0.0115008740528361 &  2.7812 \\ 
			&	1E-11 &  3.00 &  0.0115008740526303 &  2.7812 \\ 
			&	1E-12 &  3.03 &  0.0115008740526086 &  2.9609 \\ 
				\bottomrule
			\end{tabular}} 
	\end{table}%

\begin{table}[t]
	\caption{Example~2: errors $e_N^{\mathrm{rel}}$ and CPU times at simulated time $T=1\,$h. The errors have been computed with the reference solution obtained by explicit scheme with $N= 4800$ and the Godunov numerical flux. The abbreviation SBR2 refers to the scheme in~\cite{SDAMM_SBR2} (without variable transformation).}
	\label{table:errors_rel}
\smallskip 
\centering 
		{\small \begin{tabular}{@{}lrcrcrcr@{}} \toprule
		&	& \multicolumn{2}{c}{SBR2} &\multicolumn{2}{c}{Explicit} & \multicolumn{2}{c }{Semi-implicit} \\
			\midrule
		&	$N$ &  $e_{N}^{\mathrm{rel}}(t)$  & CPU $[\mathrm{s}]$ &  $e_{N}^{\mathrm{rel}}(t)$ & CPU $[\mathrm{s}]$  &  $e_{N}^{\mathrm{rel}}(t)$ & CPU $[\mathrm{s}]$    \\ \midrule 
	\multirow{7}{*}{$T=0.4\, \mathrm{h}$}	&	25   &  1.6483 &  0.0391 &  1.2053 &  0.0391 &  1.2099 &  0.0469 \\ 
		&		50   &  0.9230 &  0.0703 &  0.7688 &  0.1035 &  0.7732 &  0.0938 \\ 
		&		100  &  0.4704 &  0.2090 &  0.4368 &  0.3340 &  0.4414 &  0.2070 \\ 
		&		200  &  0.2551 &  0.7969 &  0.2384 &  1.7969 &  0.2416 &  0.5508 \\ 
		&		400  &  0.1352 &  5.2441 &  0.1261 & 11.9395 &  0.1286 &  2.0234 \\ 
		&		800  &  0.0743 & 40.9023 &  0.0645 & 89.9941 &  0.0665 &  7.4453 \\ 
		&		1600 &  0.0409 & 316.4102 &  0.0379 & 717.4668 &  0.0397 & 27.5156 \\ \midrule 
	\multirow{7}{*}{$T=0.75\, \mathrm{h}$}	&	25   &  1.5691 &  0.0900 &  1.2901 &  0.0920 &  1.2959 &  0.0880 \\ 
		&		50   &  0.8327 &  0.1320 &  0.7835 &  0.1590 &  0.7880 &  0.1360 \\ 
		&		100  &  0.4204 &  0.3240 &  0.4392 &  0.4590 &  0.4451 &  0.3190 \\ 
		&		200  &  0.2125 &  1.4766 &  0.2304 &  3.1797 &  0.2345 &  1.0234 \\ 
		&		400  &  0.1103 &  9.5469 &  0.1218 & 20.7266 &  0.1250 &  3.6797 \\ 
		&		800  &  0.0595 & 79.1719 &  0.0620 & 175.6562 &  0.0644 & 14.2188 \\ 
		&		1600 &  0.0327 & 646.8438 &  0.0371 & 1421.2190 &  0.0392 & 55.6470 \\ \midrule 
		\multirow{7}{*}{$T=1\, \mathrm{h}$} &		25   &  1.1617 &  0.0469 &  1.0747 &  0.0508 &  1.0573 &  0.0566 \\ 
		&		50   &  0.6580 &  0.0957 &  0.6966 &  0.1191 &  0.7078 &  0.1289 \\ 
		&		100  &  0.3760 &  0.3555 &  0.4519 &  0.6797 &  0.4627 &  0.3984 \\ 
		&		200  &  0.1877 &  1.9004 &  0.2821 &  4.0605 &  0.2919 &  1.3066 \\ 
		&		400  &  0.1075 & 12.4375 &  0.1658 & 27.0586 &  0.1737 &  4.8535 \\ 
		&		800  &  0.0644 & 104.6543 &  0.0896 & 237.3340 &  0.0966 & 18.8379 \\ 
		&		1600 &  0.0418 & 822.3809 &  0.0439 & 2362.8570 &  0.0495 & 75.5801 \\ 
			\bottomrule 
		\end{tabular}} 
\end{table}%

\begin{figure}[t]
\centering 
	\includegraphics{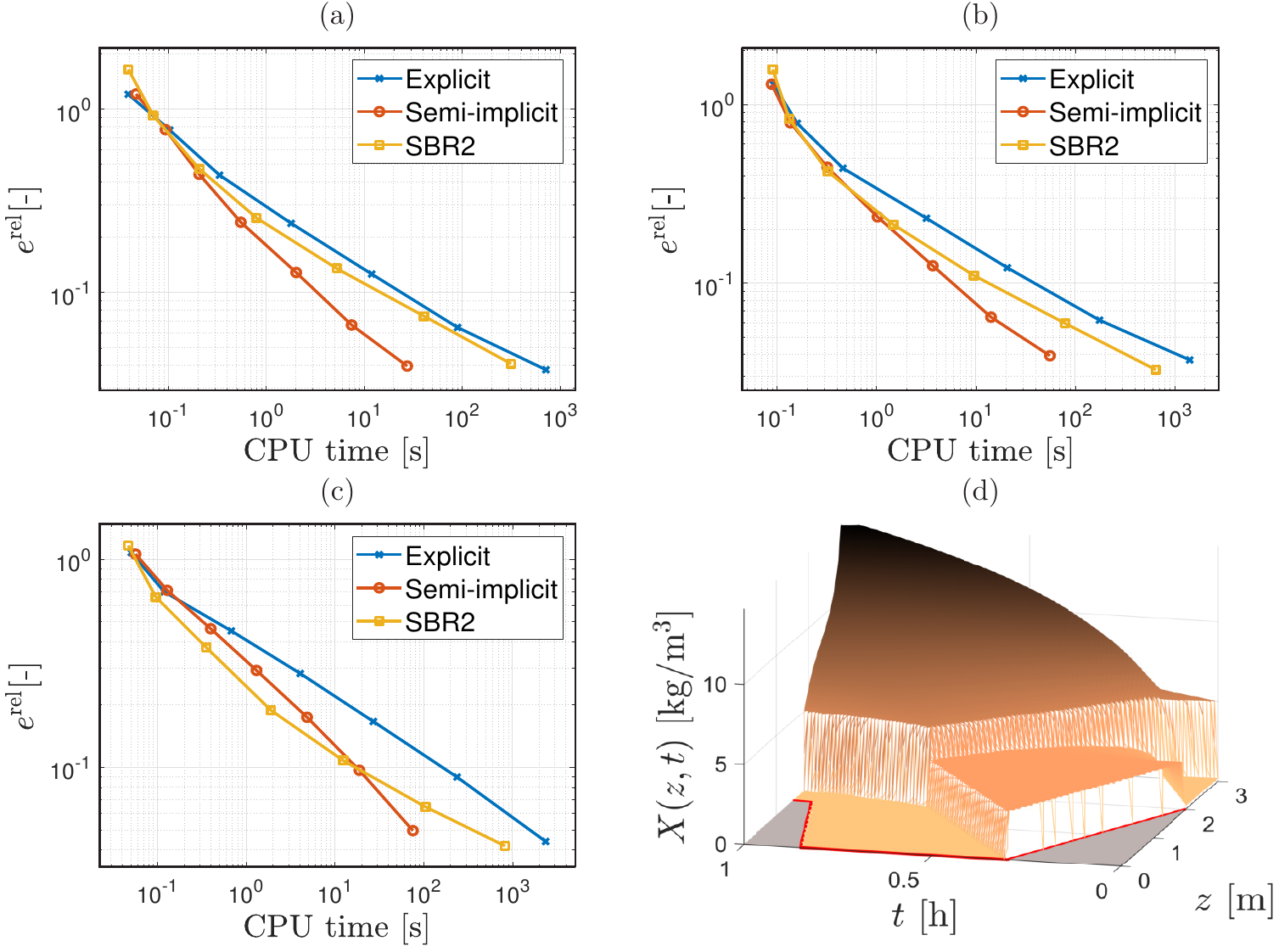} 
	\caption{Example~2: error versus CPU time for various values of~$N$ (Table~\ref{table:errors_rel})
	at simula times (a) $T=0.4 \, \mathrm{h}$, (b) $T=0.75 \, \mathrm{h}$ and (c) $T=1 \, \mathrm{h}$;  
(d) simulated total suspended solids concentration ($N=4800$) during $1\, \mathrm{h}$ (reference solution).}  \label{fig:Error_rel}
\end{figure}%

\subsection{Example~3}

We investigate the impact of a moving mesh with three different spatial grids, compared to a fixed grid; see Figure~\ref{fig:EfecMovMesh}. The simulation conditions are given in Table~\ref{table:Example3}.
The semi-implicit scheme is used with the Engquist-Osher flux and $\varepsilon =10^{-8}$. The data  in Table~\ref{table:Example3}, which contemplate no solids feed or discharge but a moving boundary, have  been chosen in such a way that the solution converges quickly in time to a stationary one with a layer of sediment at the bottom produced by the settling 
 of the initially homogeneous suspension that is not affected by the moving boundary. Figures~\ref{fig:EfecMovMesh}~(a), (c) and~(e) show that the 
  numerical solution of the SBR2 scheme reproduces this property at any of the 
   chosen discretizations $N=100$, $N=200$ or $N=400$ while the solution produced by the semi-implicit scheme is not stationary due 
    to the effect of the moving mesh. The deviation from a stationary solution seems, however, to decrease with increasing~$N$.

\begin{table}[t]
	\caption{Example~3: Time functions for the simulation. \label{table:Example3}}
	\smallskip 
	\centering
	{\small  \begin{tabular}{cccccc} \toprule
			Time period $[\mathrm{h}]$ & $\Xf(t) [\rm kg/m^3]$& $\Qf(t) [\rm m^3/h]$& $\Qu(t) [\rm m^3/h]$& $\Qe(t) [\rm m^3/h]$ & {Model} \\
			\midrule
			$0  \leq t<25$  & 0  &  0  & 0  & 0    & PDE\\
			$25  \leq t<35$ & 0  &  0  & 0  & 84    & PDE\\
			$35  \leq t<45$ & 0  &  0  & 0  & 0    & PDE\\
			$45  \leq t<60$ & 0  &  40  & 0  & 0    & PDE\\ 
			$60\leq t<70$   & 0  &  0  & 0  & 0    & PDE\\
			\bottomrule
	\end{tabular}}
\end{table}%

\begin{figure}[t]
\centering 
	\includegraphics{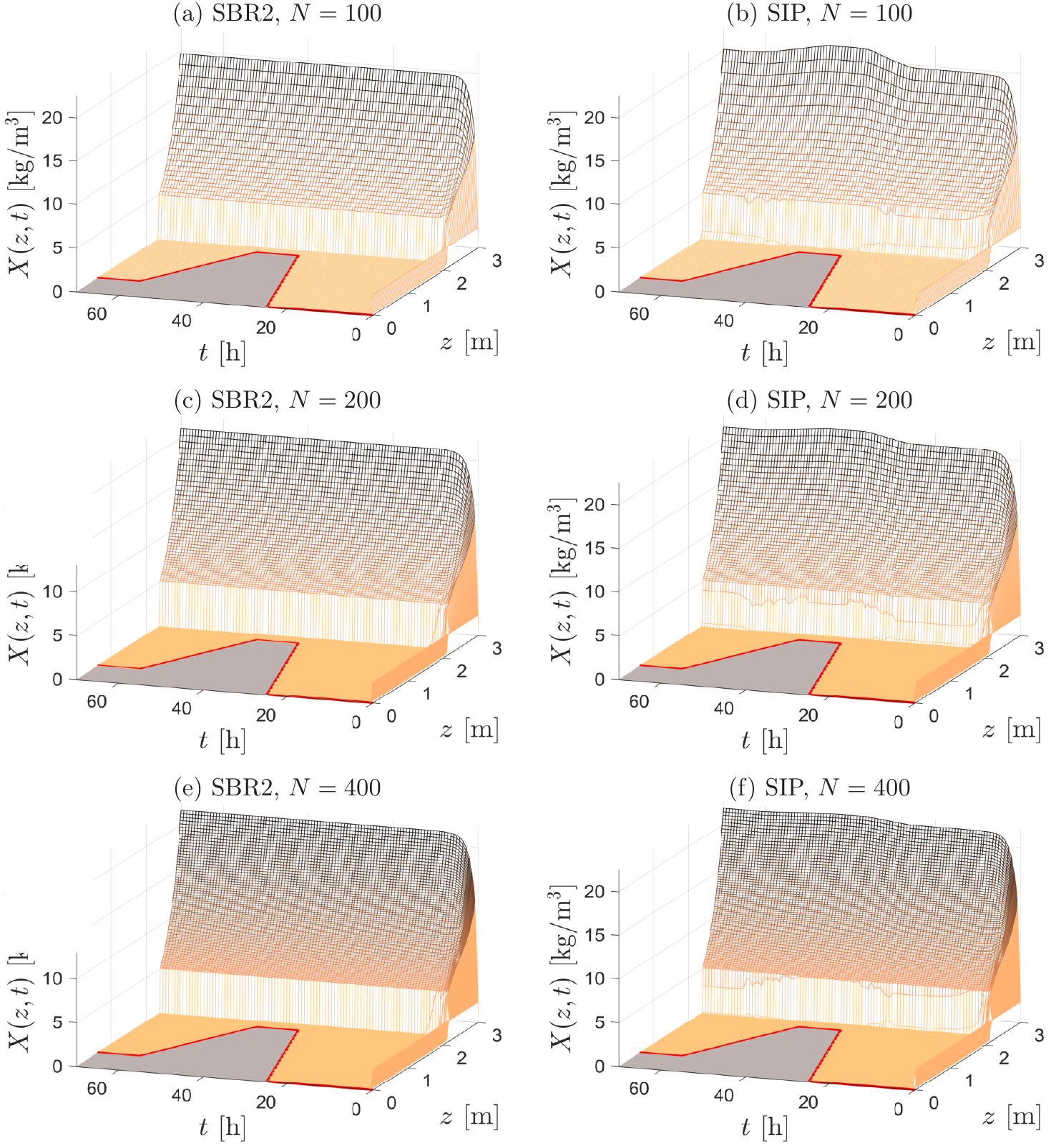} 
	\caption{Example~3: simulations of the scenario of Table~\ref{table:Example3} with  the indicated schemes and
	 values of~$N$. The fixed-mesh scheme SBR2 reproduces a steady-state  profile while numerical solutions 
	  by scheme SIP are affected by the moving mesh.}  \label{fig:EfecMovMesh}
\end{figure}%

\section{Conclusions} \label{sec:conc} 

The numerical results demonstrate that the two versions of the numerical scheme presented herein,  the explicit one of 
 Section~\ref{subsec:explicheme} and the semi-implicit one of Section~\ref{scheme:si}, generate physically relevant solutions (as is expressed by the invariant region principle) and are  working alternatives 
  to the scheme SBR2 in~\cite{SDAMM_SBR2} of 
  the PDE model advanced in \cite{SDAMM_SBR1}. Moreover, as Example~2 clearly illustrates, there is a substantial gain in CPU times and 
   efficiency of the semi-implicit version if compared with the explicit version. This gain is of course due to the 
    more favorable CFL condition~\eqref{CFLSI} (for the semi-implicit scheme) 
     that permits a larger time step than \eqref{eq:CFL}  (for the explicit scheme). The degree of this advantage depends, of course, on 
      the constitutive functions and parameters used, so it is important to emphasize that it appears for choices of these 
       constitutive ingredients that are largely considered typical for activated sludge \cite{SDcace_reactive,SDwatres_Torfs2017,Burger2023,kirimpaper}. 

The SBR2 scheme of~\cite{SDAMM_SBR2} is based on a fixed grid with respect to~$z$, and the moving interface~$\bar{z} (t)$ 
needs to be tracked explicitly. This movement gives rise to a number of  cases of boundary cells of time-dependent size 
 that need to be handled separately. That treatment is fairly involved and the  approach presented in the
   present work is in general easier to implement, and for the semi-implicit variant turns out to be more efficient. 
    That said, we mention that the fixed-grid approach of~\cite{SDAMM_SBR2} 
    has certain advantages in situations when the method is required to leave a consolidated 
     bed invariant under movements of~$\bar{z} (t)$ above it (as is illustrated by Example~3), and can be 
     extended more easily to additional inlets or outlets in the reactor. 
It is therefore 
      interesting to note that the observed convergence of the three schemes discussed 
       %(the fixed-grid SBR2 of~\cite{SDAMM_SBR2} and the explicit and semi-implicit ones of the present work) 
        to the same solution, as evidenced by Table~\ref{table:errors_rel} and Figure~\ref{fig:Error_rel}~(a), 
         supports that the treatment in~\cite{SDAMM_SBR2} is correct.   
    
With respect to limitations, we mention that the approach of \cite{SDAMM_SBR2} involves a vessel with variable cross-sectional area; this model 
ingredient has been left out here for notational convenience and  is not  expected to make the formulation and analysis significantly more difficult. We also note that while it has been proven that the schemes are monotone and obey an invariant-region principle and numerical evidence of convergence has been 
 presented, a rigorous convergence proof is still missing but can possibly be obtained by combining analyses of 
  weakly coupled degenerate parabolic systems~\cite{holden03}, numerical methods for zero-flux boundary value problems~\cite{Burger&C&S2006,kt17}, 
   and studies of triangular systems of conservation laws and degenerate convection-diffusion equations~\cite{bdmv22,bdmv23,coc09}.  
   (In the latter reference,  convergence proofs decisively depend on usage of the Engquist-Osher flux with its separable 
    upwind and downwind contributions; this has partly motivated usage of that particular numerical flux in the present work.)

Several extensions, improvements  and applications of the  model and numerical methods are conceivable. For instance, both scheme versions
 are only first-order accurate in space and time; this order can possibly be improved by applying more sophisticated 
  variable extrapolation or maximum-principle-preserving 
  weighted essentially non-oscillatory (WENO) techniques    (cf., e.g., \cite{jiang13} and references cited in that work)   
  combined with changing between the implicit and explicit steps in a more involved manner as is done 
   in implicit-explicit (IMEX) schemes, cf., e.g., \cite{bosc15}. On the other hand, one can think of an SBR 
    whose bottom works as a filter or a filtration cell with chemical reactions. In 
     such applications, the trajectory of~$\bar{z}(t)$ is no longer prescibed but arises from an 
      applied pressure, and is balanced by the hydraulic resistance of the filter medium plus that of the 
       forming sediment (filter cake). The composition of the latter is part of the solution, 
        so that possible extension forms a free-boundary problem \cite{bck01,bfk02}.  
        
       We emphasize that within the present approach soluble components are supposed to be transported passively with the fluid and are 
       subject to the  reaction terms. In some of the figure this property causes fairly sharp concentration profiles 
       when one expects that the evolution of solute  concentrations should be subject to additional effects such as dispersion/diffusion.
        In fact, diffusion driven by the gradient of the respective concentration was the unique mechanism of 
           spatial propagation of solutes considered in our first effort of modelling reactive settling 
            by PDEs \cite{SDcace_reactive}. It would not be a problem to include diffusion to the 
             present approach as an additional mechanism of solute transport, with the effect to be an overall 
              blurring of profiles.

\section*{Acknowledgements}
RB and JC are supported by ANID (Chile) through project Centro de Modelamiento Matem\'atico (BASAL projects ACE210010 and FB210005). In addition, RB is supported by Fondecyt project 1210610; Anillo project ANID/PIA/ACT210030; and CRHIAM, project ANID/FONDAP/15130015. SD acknowledges support from the Swedish Research Council (Vetenskapsr{\aa}det 2019-04601). RP is supported by ANID scholarship ANID-PCHA/Doctorado Nacional/2020-21200939.

\bibliographystyle{abbrv}
\bibliography{ref_copy}

\appendix

\section{A modified ASM1 model}\label{app}

\begin{table}[t]
\caption{The variables in the modified ASM1 model.}
\label{table:AMS1_vari}  
\smallskip 
\centering 
{\small \begin{tabular}{lll} \toprule
 \multicolumn{1}{l} { Material } & Notation & Unit \\
\midrule
Particulate inert organic matter 			   & $X_{\rm I}$  & $\rm (kg \ COD)\,m^{-3}$\\
 Slowly biodegradable substrate less part.\ organic nitrogen	   & $X_{\rm S-ND}$  & $\rm (kg \ COD)\,m^{-3}$\\
 Active heterotrophic biomass					   & $X_{\rm B, H}$& $\rm (kg \ COD)\,m^{-3}$\\
 Active autotrophic biomass 					   & $X_{\rm B, A}$& $\rm (kg \ COD)\,m^{-3}$\\
 Particulate products arising from biomass decay & $X_{\rm P}$  & $\rm (kg \ COD)\,m^{-3}$\\
 Particulate biodegradable organic nitrogen      & $X_{\rm ND}$ & $\rm (kg \ N)\,m^{-3}$\\
 Soluble inert organic matter 				   & $S_{\rm I}$  & $\rm (kg \ COD)\,m^{-3}$\\
 Readily biodegradable substrate 				   & $S_{\rm S}$  & $\rm (kg \ COD)\,m^{-3}$\\
 Oxygen 										   & $S_{\rm O}$  & $\rm -(kg\ COD)\,m^{-3}$\\
 Nitrate and nitrite nitrogen 				   & $S_{\rm NO}$ & $\rm (kg\ N)\,m^{-3}$\\
 $\mathrm{NH}_{4}^{+}+\mathrm{NH}_{3}$ nitrogen  & $S_{\rm NH}$ & $\rm (kg\ N)\,m^{-3}$\\
 Soluble biodegradable organic nitrogen 		   & $S_{\rm ND}$ & $\rm (kg\ N)\,m^{-3}$\\
\bottomrule
\end{tabular}} 
\end{table}%

\begin{table}[t]
\caption{Stoichiometric and kinetic parameters.} \label{table:Name_para}
\smallskip 
\centering 
{\small 
\begin{tabular}{lp{9.2cm}lp{35mm}} \toprule
 Symbol            & Name & Value & Unit \\ 
\midrule $Y_{\rm A}$   & Yield for autotrophic biomass & 0.24 & $\rm (g\,COD)(g\,N)^{-1}$ \\
 $Y_{\rm H}$   & Yield for heterotrophic biomass & 0.67 & $\rm  (g\,COD)(g\,COD)^{-1}$ \\
 $f_{\rm P}$   & Fraction of biomass leading to particulate products & 0.08 & dimensionless\\
 $i_{\rm XB}$  & Mass of nitrogen per mass of COD in biomass & 0.086 & $\rm (g\,N)(g\,COD)^{-1}$  \\
 $i_{\rm XP}$  & Mass of nitrogen per mass of COD in products from biomass & 0.06 & $\rm (g\,N)(g\,COD)^{-1}$ \\
 $\mu_{\rm H}$ & Maximum specific growth rate for heterotrophic biomass & 6.0 & $\rm d^{-1}$ \\
 $K_{\rm S}$   & Half-saturation coefficient for heterotrophic biomass & 20.0 & $\rm (g\,COD)\,m^{-3}$ \\
 $K_{\rm O, H}$& Oxygen half-saturation coefficient for heterotrophic biomass & 0.2 & $\rm -(g\,COD)\,m^{-3}$ \\
 $K_{\rm NO}$  & Nitrate half-saturation coefficient for denitrifying heterotrophic biomass & 0.5 & $\rm (g\,NO_{3}\text{-}N)\, m^{-3}$ \\
 $b_{\rm H}$   & Decay coefficient for heterotrophic biomass & 0.62 & $\rm d^{-1}$ \\
 $\eta_{\rm g}$& Correction factor for $\mu_{\rm H}$ under anoxic conditions & 0.8 & dimensionless \\
 $\eta_{\rm h}$& Correction factor for hydrolysis under anoxic conditions & 0.4 & dimensionless \\
 $k_{\rm h}$   & Maximum specific hydrolysis rate & 3.0 & $\rm (g\,\text{COD})\, (g\,\text{COD})^{-1}\rm d^{-1}$ \\
 $K_{\rm X}$   & Half-saturation coefficient for hydrolysis of slowly biodegradable substrate & 0.03 & $\rm (g\,\text{COD})(g\,\text{COD})^{-1}$ \\ 
 $\mu_{\rm A}$ & Maximum specific growth rate for autotrophic biomass & 0.8 & $\rm d^{-1}$ \\
 $\bar{K}_{\rm NH}$ & Ammonia half-saturation coefficient for aerobic and anaerobic growth of heterotrophs & 0.05 & $\rm (g\,NH_{3}\text{-}N)\, m^{-3}$ \\
 $K_{\rm NH}$ & Ammonia half-saturation coefficient for autotrophic biomass & 1.0 & $\rm (g\,NH_{3}\text{-}N)\, m^{-3}$\\
 $b_{\rm A}$   & Decay coefficient for autotrophic biomass & 0.15 & $\rm d^{-1}$ \\
 $K_{\rm O, A}$& Oxygen half-saturation coefficient for autotrophic biomass & 0.4 & $\rm -(g\,COD)\,m^{-3}$\\
 $k_{\rm a}$   & Ammonification rate & 0.08 & $\rm m^{3}(g\, COD)^{-1}d^{-1}$  \\
\bottomrule
\end{tabular}}  
\end{table}%

The ASM1 model is described in, for example, ~\cite{Henze2000ASMbook}.
For the reason of reformulating the PDE model to include percentages of the particulate concentrations, we have redefined the second component from $X_{\mathrm{S}}$ to $X_{\mathrm{S-ND}}\coloneqq X_{\mathrm{S}}-X_{\mathrm{ND}}$.
Then the variables are
\begin{align*}
\bC =\big(X_\mathrm{I}, X_\mathrm{S-ND}, X_\mathrm{B,H}, X_\mathrm{B,A}, X_\mathrm{P}, X_{\rm ND}\big)^\rmT,\quad 
\bS =\big(S_\mathrm{I},S_\mathrm{S},S_\mathrm{O},S_\mathrm{NO},S_\mathrm{NH},S_\mathrm{ND}\big)^\rmT,
\end{align*}
with units given in Table~\ref{table:AMS1_vari}.
The stoichiometric matrices of the modified ASM1 are given by 
\begin{align*} 
\bsigmaC \coloneqq  \begin{bmatrix}
0&0&0&0&0&0&0&0 \\
0&0&0& 1-f_{\rm P}(1+i_{\rm XP})-i_{\rm XB} & 1-f_{\rm P}(1+i_{\rm XP})-i_{\rm XB} & 0 & -1 & 1 \\
1 & 1 & 0 & -1 & 0 & 0 & 0 & 0 \\
0 & 0 & 1 & 0 & -1 & 0&0&0 \\
0 & 0 & 0 & f_{\rm P} & f_{\rm P} & 0 & 0 & 0 \\
0 & 0 & 0 & i_{\rm XB}-f_{\rm P}i_{\rm XP} & i_{\rm XB}-f_{\rm P}i_{\rm XP} & 0 & 0 & -1
                  \end{bmatrix}
\end{align*} 
for the solid components and 
\begin{align*}                   
\bsigmaS \coloneqq  \begin{bmatrix}
0&0&0&0&0&0&0&0 \\
-1/Y_{\rm H}  & -1 / Y_{\rm H} &0&0&0&0& 1 & 0\\ 
- (1-Y_{\rm H})/Y_{\rm H}  & 0 & - (4.57-Y_{\rm A})/Y_{\rm A}  &0&0&0&0&0 \\ 
0& - (1-Y_{\rm H})/(2.86Y_{\rm H}) & 1 / Y_{\rm A}  &0&0&0&0&0 \\ 
-i_{\rm XB} & -i_{\rm XB} & -i_{\rm XB}- 1 / Y_{\rm A} &0&0& 1 & 0 & 0 \\ 
0&0&0&0&0& -1 & 0 & 1
                  \end{bmatrix}
\end{align*}
for the substrates,  where the constants are given in Table~\ref{table:Name_para}
and the vector of reaction rates is 
\begin{equation*}
\br(\bC,\bS)\coloneqq 
\begin{pmatrix}
\mu_{\rm H}  \mu (S_{\rm N H},  \bar{K}_{\rm N H}) \mu ( S_{\rm S}, K_{\rm S}) \mu(  S_{\rm O},  K_{\rm O, H} ) X_{\rm B, H}\\[2pt]  
\mu_{\rm H}  \mu (S_{\rm N H},  \bar{K}_{\rm N H}) \mu(  S_{\rm S},  K_{\rm S}) \mu ( K_{\rm O, H},  S_{\rm O})  
 \mu ( S_{\rm N O},  K_{\rm N O} ) \eta_{\rm g} X_{\rm B, H}\\[2pt]
\mu_{\rm A} \mu( S_{\rm N H}, K_{\rm N H}) \mu (   S_{\rm O},  K_{\rm O, A})  X_{\rm B, A}\\[2pt]  
b_{\rm H} X_{\rm B, H}\\[2pt]
b_{\rm A} X_{\rm B, A}\\[2pt] 
k_{\rm a} S_{\rm N D} X_{\rm B, H}\\[2pt]
k_{\rm h} \mu_7(X_{\rm S},X_{\rm B, H})\bigl( \mu(S_{\rm O},  K_{\rm O, H}) +\eta_{\rm h} \mu(  K_{\rm O, H},  S_{\rm O}) 
 \mu ( S_{\rm N O},  K_{\rm N O}) \bigr)\\[2pt] 
k_{\rm h} \mu_8(X_{\rm B, H},X_{\rm N D})\bigl(  \mu(S_{\rm O}, K_{\rm O, H}) +\eta_{\rm h} \mu(K_{\rm O, H},S_{\rm O}) 
 \mu( S_{\rm N O}, K_{\rm N O})\bigr)
\end{pmatrix}.
\end{equation*}
Here we define the Monod expression $\mu (A,B):= A/(A+B)$ and 
\begin{align*}
\mu_7(X_{\rm S},X_{\rm B, H})
&\coloneqq \begin{cases}
0&\text{if $X_{\rm S}=0$ and $X_{\rm B, H}=0$,}\\
 \dfrac{X_{\rm S}X_{\rm B, H}}{K_{\rm X} X_{\rm B,H}+X_{\rm S}}&\text{otherwise,}
\end{cases}
\\
\mu_8(X_{\rm B, H},X_{\rm N D})
&\coloneqq \begin{cases}
0 &\text{if $X_{\rm S}=0$ and $X_{\rm B, H}=0$,}\\
 \dfrac{X_{\rm B, H}X_{\rm N D}}{K_{\rm X}X_{\rm B,H} + X_{\rm S}} &\text{otherwise.}
\end{cases}
\end{align*}

\end{document}